\ifdefined\moorsubmission
\documentclass[moor,sglanonrev]{informs4}
\input{moor-prelims}
\else
\documentclass[11pt]{article}
\usepackage[left=3cm,right=3cm,top=2.54cm,bottom=2.54cm]{geometry}
\fi

\usepackage{url}            
\usepackage{booktabs}       
\usepackage{nicefrac}       
\usepackage{microtype}      
\usepackage{amssymb,amsmath,amsfonts}
\usepackage{psfrag}
\usepackage{xcolor}
\usepackage{enumerate}

\ifdefined\moorsubmission
\else
\usepackage{amsthm}
\usepackage[numbers,square]{natbib}

\usepackage{xifthen}

\makeatletter
\def\@maketitle{%
  \newpage
  \null
  \vskip -1em%
  \begin{center}%
    \let \footnote \thanks
    \vskip -1em%
    {\Large \textsc{\@title} \par}%
    \vskip 1em%
    {\large
      \lineskip .5em%
      \begin{tabular}[t]{c}%
        \@author
      \end{tabular}\par}%
    \vskip 1em%
    {\large \@date}%
  \end{center}%
  \par
  \vskip 1.5em}
\makeatother

\newtheorem{theorem}{Theorem}
\newtheorem{lemma}{Lemma}
\newtheorem*{lemma*}{Lemma}
\newtheorem{corollary}{Corollary}
\newtheorem{proposition}{Proposition}
\newtheorem{definition}{Definition}[section]

\newtheorem{assumption}{Assumption}

\newcounter{example}

\newenvironment{example}[1][]{
  \refstepcounter{example}
  \ifthenelse{\isempty{#1}}{%
    \noindent \textbf{Example \theexample:}\hspace*{.05em}
  }{%
    \noindent \textbf{Example \theexample} ({#1})\textbf{:}\hspace*{.05em}
  }
}{%
  $\Diamond$ \medskip
}

\title{Geometry, Computation, and Optimality in \\ Stochastic Optimization}
\author{Chen Cheng$^1$ ~~~~~~ John C.\ Duchi$^{1,2}$ ~~~~~~ Daniel Levy$^3$
  \\[.25cm]  
  Departments of $^1$Statistics, $^2$Electrical Engineering, and $^3$Computer Science \\
  Stanford University\footnote{
  CC, JCD, and DL partially
  by NSF-CAREER Award 1553086, ONR-YIP N00014-19-1-2288,
  and the Stanford DAWN Project. JCD and CC additionally
  supported by NSF Award IIS-2006777 and
  ONR Award N00014-22-1-2669.  
}}

\date{July 2025}

\fi

\usepackage{xspace}

\newcommand{\norm}[1]{\left\|{#1}\right\|} 
\newcommand{\lone}[1]{\norm{#1}_1} 
\newcommand{\ltwo}[1]{\norm{#1}_2} 
\newcommand{\linf}[1]{\norm{#1}_\infty} 
\newcommand{\dnorm}[1]{\norm{#1}_*} 
\newcommand{\lfro}[1]{\left\|{#1}\right\|_{\rm Fr}} 
\newcommand{\matrixnorm}[1]{\left|\!\left|\!\left|{#1}
  \right|\!\right|\!\right|} 
\newcommand{\normbigg}[1]{\bigg\|{#1}\bigg\|} 
\newcommand{\norms}[1]{\|{#1}\|} 
\newcommand{\ltwos}[1]{\norms{#1}_2} 

\newcommand{\<}{\langle}
\renewcommand{\>}{\rangle}
\newcommand{\wb}[1]{\overline{#1}}

\newcommand{\opnorm}[1]{\matrixnorm{#1}_{\rm op}} 

\newcommand{\floor}[1]{\lfloor{#1}\rfloor}

\newcommand{\denselist}{\itemsep 0pt\partopsep 0pt}
\newcommand{\bitem}{\begin{itemize}\denselist}
\newcommand{\eitem}{\end{itemize}}
\newcommand{\benum}{\begin{enumerate}\denselist}
\newcommand{\eenum}{\end{enumerate}}

\newcommand{\openright}[2]{\left[{#1}, {#2}\right)}
\newcommand{\openleft}[2]{\left({#1}, {#2}\right]}

\newcommand{\thm}{\begin{theorem}}
\newcommand{\lem}{\begin{lemma}}
\newcommand{\pro}{\begin{proposition}}
\newcommand{\dfn}{\begin{definition}}
\newcommand{\rem}{\begin{remark}}
\newcommand{\xam}{\begin{example}}
\newcommand{\cor}{\begin{corollary}}

\newcommand{\ethm}{\end{theorem}}
\newcommand{\elem}{\end{lemma}}
\newcommand{\epro}{\end{proposition}}
\newcommand{\edfn}{\end{definition}} 
\newcommand{\erem}{\bbox\end{remark}}
\newcommand{\exam}{\bbox\end{example}}
\newcommand{\ecor}{\end{corollary}}

\newcommand{\beqn}{\begin{equation}}
\newcommand{\eeqn}{\end{equation}}

\newcommand{\dom}{\mathop{\rm dom}}

\newcommand{\bbox}{\vrule height7pt width4pt depth1pt}

\newcommand{\minimize}{\mathop{\rm minimize}}

\newcommand{\invert}[1]{{#1}^{-1}}

\newcommand{\commentout}[1]{}

\newcommand{\tr}{{\rm tr}}

\newcommand{\R}{\mathbf{R}}  

\newcommand{\N}{\mathbf{N}}

\newcommand{\stepsize}{\alpha}

\newcommand{\sign}{\mathop{\rm sign}}

\def\twofigboxnolabelFive#1#2{%
\begin{minipage}{\textwidth}%
\hbox to 0.5in{}\epsfxsize=0.35\maxfigwidth
\noindent \epsffile{#1}\hfill
\epsfxsize=0.35\maxfigwidth
\epsffile{#2}\hbox to 0.5in{}\\
\end{minipage}
}

\newlength{\maxfigwidth}
\newcommand{\cd}{\stackrel{d}{\to}}
\newcommand{\diam}{\textup{diam}}

\newcommand{\E}{\mathbf{E}}
\renewcommand{\P}{\mathbb{P}}
\newcommand{\var}{\mathop{\rm Var}}
\newcommand{\defeq}{:=}
\newcommand{\normal}{\mathsf{N}}

\newcommand{\half}{\frac{1}{2}}

\newcommand{\zeros}{\mathbf{0}}
\newcommand{\ones}{\mathbf{1}}

\usepackage{xcolor}
\definecolor{inneralgcolor}{rgb}{.95,.98,1}
\definecolor{outeralgcolor}{rgb}{0,0,.3}

\definecolor{innerboxcolor}{rgb}{.9,.95,1}
\definecolor{outerlinecolor}{rgb}{.6,0,.2}

\newtheorem{observation}{Observation}[section]

\newcommand{\ball}[3]{\mathbf{B}_{#1}(#2, #3)}

\newcommand{\dotp}[2]{#1^\top#2}
\newcommand{\inprod}[2]{\left\<{#1}, {#2}\right\>}
\newcommand{\breg}[3]{\mathrm{D}_{#1}(#2, #3)}

\newcommand{\hinge}[1]{\left({#1}\right)_+}

\providecommand{\argmin}{\mathop{\textup{argmin}}}
\providecommand{\argmax}{\mathop{\textup{argmax}}}

\ifdefined\moorsubmission
\renewcommand{\mc}[1]{\mathcal{#1}}
\else
\newcommand{\mc}[1]{\mathcal{#1}}
\fi

\newcommand{\what}[1]{\widehat{#1}}
\newcommand{\simiid}{\stackrel{\textup{iid}}{\sim}}
\newcommand{\tvnorm}[1]{\left\|{#1}\right\|_{\mathsf{tv}}}

\newcommand{\tvnormbig}[1]{\big\|{#1}\big\|_{\mathsf{tv}}}
\newcommand{\dkl}[2]{D_{\textup{kl}}\left({#1} |\!| {#2}\right)}
\newcommand{\dhel}{d_{\textup{hel}}}
\newcommand{\indic}[1]{1\!\left({#1}\right)}
\newcommand{\qhull}{\mathsf{QHull}}
\newcommand{\squaredhull}{\mathsf{SqHull}}

\newcommand{\convexindic}[1]{\mathbf{1}_{#1}} 

\newcommand{\conv}{\textup{Conv}}

\newcommand{\F}[2]{\mathcal{F}^{#1, #2}}

\newcommand{\minimax}{\mathfrak{M}_k}
\newcommand{\minimaxs}{\minimax^{\mathsf{S}}}
\newcommand{\minimaxr}{\minimax^{\mathsf{R}}}
\newcommand{\mi}[2]{\mathsf{I}(#1; #2)}

\newcommand{\diag}{\textup{diag}}
\newcommand{\eucgrad}{Euclidean gradient methods\xspace}
\newcommand{\dscaled}{diagonally scaled gradient methods\xspace}
\providecommand{\stepsize}{\alpha}
\newcommand{\bregman}[3]{\breg{#1}{#2}{#3}}

\newcommand{\regret}{\mathsf{Regret}}

\newcommand{\f}{F}
\newcommand{\ff}{f}

\newcommand{\statrv}{S}
\newcommand{\statval}{s}
\newcommand{\statdomain}{\mc{S}}
\newcommand{\optvar}{x}
\newcommand{\optdomain}{X}

\newcommand{\packval}{v}
\newcommand{\packset}{\mc{V}}
\newcommand{\width}{w}
\newcommand{\nlwidth}{w_{\textup{nl}}}

\newcommand{\lipnorm}[1]{\norm{#1}_{\textup{Lip}}}
\newcommand{\measures}{\mc{P}}
\newcommand{\dopt}{\mathsf{d}_{\textup{opt}}}
\newcommand{\gradomain}{G}

\newcommand{\dham}{\mathsf{d}_{\textup{Ham}}}

\definecolor{darkblue}{rgb}{0,0,.5}
\usepackage[colorlinks=true,allcolors=darkblue]{hyperref}

\begin{document}

\maketitle
\ifdefined\moorsubmission
\else
\vspace{-1cm}
\begin{abstract}
%
We study computational and statistical consequences of problem geometry in
stochastic and online optimization. By focusing on constraint set and
gradient geometry, we characterize the problem families for which
stochastic- and adaptive-gradient methods are (minimax) optimal and,
conversely, when nonlinear updates---such as those mirror descent
employs---are necessary for optimal convergence.  When the constraint set
is quadratically convex, diagonally pre-conditioned stochastic gradient
methods are minimax optimal. We provide quantitative converses showing
that the ``distance'' of the underlying constraints from quadratic
convexity determines the sub-optimality of subgradient methods. These
results apply, for example, to any $\ell_p$-ball for $p < 2$, and the
computation/accuracy tradeoffs they demonstrate exhibit a striking analogy
to those in Gaussian sequence models.

\end{abstract}
\fi

\section{Introduction}


The default procedures for solving the stochastic optimization problem
\begin{equation}
  \label{eqn:problem}
  \minimize_{\optvar \in \optdomain}
  ~ \ff_P(\optvar) \defeq \E_P\left[\f(\optvar, \statrv)\right]
  = \int \f(\optvar, \statval) dP(\statval),
\end{equation}
where $\{\f(\cdot, \statval), \statval \in \statdomain\}$ are
convex functions $\f(\cdot,\statval) : \R^n \to \R$,
$P$ is a distribution on $\statdomain$, and $\optdomain \subset \R^n$
is a closed convex set,
are variants of the stochastic subgradient method, where
one iteratively draws $\statrv_k \simiid P$ and updates
\begin{equation}
  \label{eqn:sgm-update}
  \optvar_{k + 1} \defeq \optvar_k - \stepsize_k g_k,
  ~~ \mbox{where} ~~ g_k \in \partial \f(\optvar_k, \statrv_k).
\end{equation}
The simplicity and scalability of this update make stochastic subgradient
methods the \emph{de facto} choice for large-scale
optimization~\cite{RobbinsMo51, NemirovskiJuLaSh09, BottouCuNo18}.  The
geometry of the underlying underlying constraint set $\optdomain$ and
subgradients $\partial \f(\cdot, \statval)$ impact the performance of
algorithms for problem~\eqref{eqn:problem}, and so a question arises: are
such linear updates~\eqref{eqn:sgm-update} enough to obtain (minimax rate)
optimal convergence guarantees for the problem~\eqref{eqn:problem}, or does
does the structure of the problem \emph{necessitate} nonlinearity to achieve
optimization efficiency?  Convergence guarantees for stochastic gradient
methods depend on the $\ell_2$-diameter of $\optdomain$ and $\partial
\f(\cdot, \statval)$, while for non-Euclidean geometries (e.g.\ when
$\optdomain$ is an $\ell_1$- or $\ell_\infty$-ball)
mirror descent, dual averaging and adaptive
gradient methods
provide better
convergence guarantees~\cite{NemirovskiYu83, NemirovskiJuLaSh09, BeckTe03,
  Nesterov09, DuchiHaSi11, CutkoskyOr18}.  We investigate these gaps by
precisely quantifying convergence for different method families,
highlighting a particular way to trade between computational power---which
we treat as whether purely linear operations suffice to optimally solve
problem~\eqref{eqn:problem}, or nonlinear updates are necessary---and
optimization and statistical efficiency.






To set the stage, let us revisit
\citeauthor*{DonohoLiMa90}'s study of optimal estimation
in Gaussian sequence models~\cite{DonohoLiMa90}.
One observes a vector
$\optvar \in \optdomain$ corrupted by Gaussian noise,
\begin{equation}
  \label{eqn:sequence-model-obs}
  y = \optvar + \normal(0, \sigma^2 I),
\end{equation}
and seeks to estimate $\optvar$.
For such problems, one can consider linear estimators---$\what{\optvar} =
AY$ for a $A \in \R^{n\times n}$---or potentially non-linear
estimators
\begin{equation*}
  \what{\optvar} = \Phi(y)
\end{equation*}
where $\Phi : \R^n \to \optdomain$ is otherwise arbitrary. When $\optdomain$
is quadratically convex, meaning the set $\optdomain^2 \defeq
\{(\optvar_j^2)_{j \ge 1} \mid \optvar \in \optdomain\}$ is convex,
\citeauthor{DonohoLiMa90} show there exist minimax rate-optimal linear
estimators; conversely, there are non-quadratically convex $\optdomain$ for
which rate-optimal estimators $\what{\optvar}$ \emph{must} be
nonlinear in $y$.
In particular, as we discuss in Section~\ref{sec:role-quadratic-convexity},
this gap depends on the difference
between the Kolmogorov (linear) $n$-width of $\optdomain$ and
its ``nonlinear'' $n$-width, that is,
\begin{equation}
  \label{eqn:n-widths-intro}
  \width^2(n) \defeq \sup_{v \in \conv(X^2)}
  \sum_{j > n} v_{(j)}
  ~~~ \mbox{versus} ~~~
  \nlwidth^2(n) \defeq \sup_{v \in X^2}
  \sum_{j > n} v_{(j)},
\end{equation}
where $|v_{(1)}| \ge |v_{(2)}| \ge \cdots$ denote the elements of
$v$ sorted by magnitude.
We show how these results follow from convex duality,
and the difference between $\width^2(n)$ and $\nlwidth^2(n)$
allows a quantitative characterization of how far
$\optdomain$ is from being quadratically convex and the impact this distance
has on the (sub)optimality of linear estimators.

Our results show how stochastic and online convex optimization
analogize these sequence models.
To build the analogy, consider dual
averaging~\cite{Nesterov09}, where for a strongly convex $h : \optdomain \to
\R \cup \{+\infty\}$, one iteratively receives $\statrv_k
\in \statdomain$, chooses $g_k \in \partial \f(\optvar_k, \statrv_k)$, and
for a stepsize $\stepsize_k > 0$ updates
\begin{equation}
  \label{eq:da}
  \optvar_{k+1} \defeq \argmin_{\optvar}
  \bigg\{ \sum_{i\leq k}g_i^\top \optvar
  + \frac{1}{\stepsize_k} h(\optvar)\bigg\},
\end{equation}
By taking $h(\optvar) = h_0(\optvar)
+ \convexindic{\optdomain}(\optvar)$, where $h_0$ is strongly
convex on $\optdomain$ and $\convexindic{\optdomain}(\optvar)
= 0$ when $\optvar \in \optdomain$ and $+\infty$ otherwise,
$h$ can incorporate constraints in the update~\eqref{eq:da}.
When $h$ is Euclidean, that is, $h(\optvar) = \half
\dotp{\optvar}{A\optvar}$ for some $A \succ 0$, the updates are linear in the
observed gradients $g_i$, as
\begin{equation*}
  \optvar_k = -\stepsize_k A^{-1} \sum_{i \le k}
  g_i.
\end{equation*}
Adapting common practice in online
learning~\cite{CesaBianchiLu06} (and
in sequence model estimation~\eqref{eqn:sequence-model-obs}),
we allow \emph{improper} methods, which
may output estimates $\what{\optvar} \not \in \optdomain$,
saying that $\what{\optvar}$ is $\epsilon$-sub-optimal relative to
$\optdomain$ if
\begin{equation*}
  f(\what{\optvar}) \le \inf_{\optvar \in \optdomain} f(\optvar) +
  \epsilon.
\end{equation*}

Drawing a parallel between $\Phi$ in the Gaussian sequence model and $h$ in
dual averaging~\eqref{eq:da}, we show that because of duality gaps in
certain min-max problems, a dichotomy holds for stochastic and online convex
optimization similar to that holding for the Gaussian sequence model: if
$\optdomain$ is quadratically convex, there is a Euclidean $h$ (yielding
linear updates~\eqref{eq:da}) that is minimax rate optimal for
problem~\eqref{eqn:problem}, while there exist non-quadratically convex
$\optdomain$ for which Euclidean distance-generating $h$ are arbitrarily
suboptimal.
Taking a computational perspective, this means that for some
problems one \emph{must} use more sophisticated methods than ``linear''
updates.
We show that this analogy holds, though the measurement of a set's
deviance from quadratic convexity, and hence the gap in attainable
performance between linear and nonlinear methods, differs between Gaussian
sequence models and stochastic optimization: there are constraint sets
$\optdomain$ for which linear estimators are (rate) optimal in the Gaussian
sequence model but not for stochastic optimization, and vice versa.
%

More precisely, we prove that for orthosymmetric
quadratically convex bodies $\optdomain$, subgradient methods with a
fixed diagonal re-scaling are minimax rate optimal.
This guarantees that for
a large collection of constraints (e.g.\ $\ell_2$-balls, weighted
$\ell_p$-bodies for $p\geq 2$, or hyperrectangles) a diagonal re-scaling
suffices.
This is important in, e.g., machine learning problems of appropriate
geometry, such as in linear classification problems where the data
(features) are sparse, so using a dense predictor $\optvar$ is
natural~\cite{DuchiHaSi11, DuchiJoMc13}. Conversely, we show that if the
constraint set $\optdomain$ is a (scaled) $\ell_p$-ball, $1 \le p < 2$,
then, considering unconstrained updates~\eqref{eq:da}, the regret of the
best method of linear type (i.e.\ $h$ quadratic)
can be $\sqrt{n / \log n}$ times larger than the
minimax rate in online convex optimization.
As part of this, we provide new
information-theoretic lower bounds on optimization for general convex
constraints $\optdomain$. In contrast to the frequent (but
illogical) practice
of comparing convergence upper bounds, we demonstrate that
the gap between linear and non-linear methods must hold.
Sections~\ref{sec:hard-euclidean-problems}
and~\ref{sec:role-quadratic-convexity}
also show
how the departure from
quadratic convexity affects convergence guarantees:
comparing the $\ell_1$ diameters of $\optdomain$ and
its second-order lifts via
\begin{equation*}
  \sup_{x \in \optdomain} \lone{\optvar}
  ~~ \mbox{versus} ~~
  \sup_{v} \left\{\lone{v}
  \mid v^2 \in \conv\{\diag(xx^\top), x \in \optdomain\})
  \right\},
\end{equation*}
the gap between the left and right quantities (essentially) characterizes
the gap in performance between linear and nonlinear methods for stochastic
optimization, while Kolmogorov $n$-widths~\eqref{eqn:n-widths-intro} capture
that in the Gaussian sequence model.

We extend our results to an additional computational consideration: whether
an algorithm must be adaptive, that is, it must change its update rules over
time based on observations. We demonstrate that non-adaptive linear methods
necessarily suffer slower convergence rates than adaptive methods in online
problems.  One perspective on our results is thus computational, though with
a different angle than most current work on tradeoffs between statistics and
computational complexity. Much of this literature takes as inspiration the
classical perspective that the gap between polynomial and non-polynomial
time algorithms forms the great watershed in computational complexity, thus
necessitating a class of ``hard'' problems while allowing essentialy
arbitrary algorithms~\cite{BerthetRi13, BrennanBrHu18, CaiWu20}.  We take an
alternative perspective that allows more nuance in the types of convergence
rates we can achieve---differentiating between various polynomials---by
restricting the algorithms we consider to those in families common in
optimization.

Our conclusions relate to the growing literature in adaptive algorithms
\cite{BartlettHaRa07, DuchiHaSi11, OrabonaCr10, OrabonaPa18,
  CutkoskySa19}. Our results effectively prescribe that these adaptive
algorithms are useful when the constraint set is quadratically convex, as
this guarantees a minimax optimal diagonal pre-conditioner. More,
different sets suggest different regularizers. For example, when the
constraint set is a hyperrectangle, AdaGrad has regret at most $\sqrt{2}$
times that of the best post-hoc pre-conditioner, which we show is minimax
optimal, while (non-adaptive) standard gradient methods can be $\sqrt{n}$
suboptimal on such problems.
Conversely, our results show that these
diagonal rescaling methods may be quite
rate sub-optimal for non-quadratically convex constraint sets.
Our
results thus clarify existing convergence guarantees~\cite{Nesterov09,
  NemirovskiJuLaSh09, DuchiHaSi11, WilsonRoStSrRe17}: when the geometry of
$\optdomain$ and $\partial \f$ is appropriate for adaptive gradient methods
or Euclidean algorithms, one should use them; when it is not---the
constraints $\optdomain$ are not quadratically convex---other
methods admit substantially better convergence.


\paragraph{Notation}
We use $n$ to refer to the dimension of problems, and we use $k$
to denote either a sample size or
number of iterations. We let $\R^\N = \{(\optvar_j)_{j = 1}^\infty\}$
denote sequence space. For a norm
$\gamma$, the set
$\ball{\gamma}{x_0}{r} \defeq \{x  \mid
\gamma(x-x_0) \leq r\}$
denotes the ball of radius $r$ around $x_0$ in the $\gamma$ norm.  For $p
\in [1, \infty]$ we use the shorthand $\ball{p}{x_0}{r} :=
\ball{\|\cdot\|_p}{x_0}{r}$.  The dual norm of $\gamma$ is $\gamma^*(z) =
\sup_{\gamma(x) \le 1} \inprod{x}{z}$. For $\optvar, \tau \in \R^n$ or $\R^\N$,
we abuse
notation and define $\optvar^2 := (\optvar^2_j)_{j \ge 1}$, $|\optvar| :=
(|\optvar_j|)_{j \ge 1}$, $\frac{\optvar}{\tau} := (\optvar_j/\tau_j)_{j \ge 1}$
and $\optvar \odot \tau \defeq (\optvar_j \tau_j)_{j \ge 1}$, and
similarly for sets $X$, if $f : \R \to \R$ then we let $f(X) =
\{(f(x_j))_{j \ge 1} \mid x \in X\}$ be the elementwise application of $f$
to elements of $X$.  The function $h$
denotes a \textit{distance generating function}, i.e.\ a function strongly
convex with respect to a norm $\norm{\cdot}$; $\breg{h}{x}{y} = h(x) - h(y)
- \inprod{\nabla h(y)}{x - y}$ denotes the Bregman divergence, where recall that
$h$ is strongly convex with respect to $\norm{\cdot}$ if and only if
$\breg{h}{x}{y} \ge \half \norm{x - y}^2$.  The subdifferential of
$\f(\cdot, \statval)$ at $\optvar$ is $\partial_\optvar \f(\optvar,
\statval)$. $\mi{X}{Y}$ is the (Shannon) mutual information between random
variables $X$ and $Y$. For a set $\Omega$ and $f, g:\Omega \to \R$, we write
$f\lesssim g$ if there exists a finite numerical constant $C$ such that
$f(t) \le C g(t)$ for $t \in \Omega$, and $f \asymp g$ if $g \lesssim f
\lesssim g$.


\section{Preliminaries and Background}
\label{sec:prelim}

We begin by reviewing the classical results in Gaussian sequence
models and presenting and
defining the minimax framework in which we analyze procedures.  We also review
standard stochastic subgradient methods and introduce the relevant geometric
notions of convexity we require.  As part of this, we give a new argument
showing the optimality of linear estimators for Gaussian sequence models
when the underlying constraint set is quadratically convex (which we define
presently).

\subsection*{Quadratic convexity and orthosymmetry}
A few geometric quantities are central to our development. 
For a set $\optdomain$, let $\optdomain^2
:= \{\optvar^2, \optvar\in\optdomain\}$ denote its (elementwise) square. The
set $\optdomain$ is \emph{quadratically convex} if $\optdomain^2$ is convex;
typical examples of quadratically convex sets are weighted $\ell_p$ bodies
for $p\geq 2$ or hyperrectangles. We let $\qhull(\optdomain)$ be the
quadratic convex hull of $\optdomain$, meaning the smallest convex and
quadratically convex set containing $\optdomain$.  The set $\optdomain
\subset \R^n$ or $\optdomain \subset \R^\N$ is
\emph{orthosymmetric} if it is invariant to flipping the signs of any
coordinate: if $\optvar\in\optdomain$ then $\sigma_j \in \{\pm 1\}$
implies $(\sigma_j \optvar_j)_{j \ge 1} \in \optdomain$.
Similarly, a norm $\gamma$ is orthosymmetric if $\gamma(g) =
\gamma(|g|)$ for all $g$, and $\gamma$ is
quadratically convex if it induces a quadratically convex unit ball
$\ball{\gamma}{0}{1}$.
For any set $\optdomain$, we define the squared convex hull
and square root
\begin{equation*}
  \squaredhull(\optdomain) \defeq \conv \left\{(\optvar_j^2)
  \mid \optvar \in \optdomain\right\}
  ~~ \mbox{and} ~~
  \sqrt{\squaredhull(\optdomain)}
  = \{(\sqrt{y_j}) \mid y \in \squaredhull(\optdomain)\},
\end{equation*}
the latter of which is always convex by the concavity of the square root.
For orthosymmetric $\optdomain$,
\begin{equation*}
  \qhull(\optdomain)
  = \left\{s \odot \optvar \mid \optvar \in \sqrt{\squaredhull(\optdomain)},
  s_j \in \{\pm 1\} ~ \mbox{for~all~} j
  \right\}.
\end{equation*}

\subsection{The Gaussian sequence model}
\label{sec:gsm}

\newcommand{\gsmrisk}{R}
\newcommand{\maxgsmrisk}{R^*}
\newcommand{\lineargsmrisk}{R_{\textup{lin}}^*}
\newcommand{\noise}{\xi}
\newcommand{\rectangles}{\mc{R}}

Gaussian sequences provide a model for analyzing parametric and
nonparametric statistical procedures, and tools developed in their analysis
form a bedrock of modern statistical estimation~\cite{Tsybakov09,
  Johnstone17}; we provide some perspective on estimation in the sequence
model.
In the Gaussian sequence model, we begin with a (typically convex and
compact) set $\optdomain \subset \R^n$ or in sequence space $\R^\N$, and for
an unknown $\optvar \in \optdomain$ observe $y = \optvar + \noise$, where
$\noise \sim \normal(0, \sigma^2 I)$.  The goal is to estimate $\optvar$ in
some sense optimally, and frequently one considers sequences with
$\optdomain \subset \R^n$ and $\sigma^2$ scaling as $1/n$, which analogizes
estimation with $n$ observations, so that rates of
convergence as $\sigma^2 \downarrow 0$ become the main
focus~\cite{Johnstone17}.  An interesting point of contrast is when linear
estimators are sufficient to achieve (near) optimal performance or nonlinear
estimators are necessary. An estimator $\what{\optvar} = \what{\optvar}(y)$
is \emph{linear} if it has the form $\what{\optvar} = A y$ for a linear
operator $A$ and nonlinear otherwise.  We consider the \emph{risk}
\begin{equation*}
  \gsmrisk(\what{\optvar}, \optvar)
  \defeq \E\left[\ltwos{\what{\optvar}(y) - \optvar}^2\right],
\end{equation*}
where for linear estimators of the form $\what{\optvar} = Ay$, we use the
shorthand
\begin{equation}
  \gsmrisk(A, x) = \E\left[\ltwo{Ay - x}^2\right]
  = \E\left[\ltwo{(A - I)x + A \noise}^2\right]
  = \ltwo{(A - I) x}^2 + \sigma^2 \lfro{A}^2.
  \label{eqn:linear-gsm-shorthand}
\end{equation}
The \emph{maximum risk} of an estimator over the set $X$ is
\begin{equation*}
  \maxgsmrisk(\what{x}, X) \defeq \sup_{x \in X} \gsmrisk(\what{x}, x),
\end{equation*}
while the minimax risk and linear minimax risk are
\begin{equation*}
  \maxgsmrisk(X) \defeq \inf_{\what{x}} \maxgsmrisk(\what{x}, X)
  ~~~ \mbox{and} ~~~
  \lineargsmrisk(X) \defeq \inf_A \maxgsmrisk(A, X).
\end{equation*}

\citet{DonohoLiMa90} prove fundamental results relating the minimax risk
and linear minimax risk for the Gaussian sequence model, and among their
main results is that if the set $\optdomain$ is an orthosymmetric
quadratically convex body, then $\lineargsmrisk(\optdomain) \le 1.25
\maxgsmrisk(\optdomain)$, and moreover, $\lineargsmrisk(\optdomain) =
\sup_{H \subset \optdomain} \lineargsmrisk(H)$, where $H$ is a
(hyper)rectangle.
We provide an alternative approach to some of
these arguments via convex duality, which allows us to put these
arguments and the rest of our development on similar intellectual
footing, though we defer proofs to appendices.
In brief, the Sion and Fan minimax theorems~\cite{Sion58, Fan53},
coupled with quadratic lifts of the set $\optdomain$, play an essential role
in all of our results.
\begin{proposition}
  \label{proposition:gsm-linear}
  Assume $\sigma^2 > 0$ and $X \subset \R^n$ is an orthosymmetric convex
  body. Then the matrix $A$ minimizing $\maxgsmrisk(A, X)$ is diagonal and
  unique. Moreover,
  \begin{equation*}
    \inf_A \maxgsmrisk(A, X)
    = \inf_{d \in \R^n}
    \sup_{x \in X} \left\{\sum_{j = 1}^n
    (d_j - 1)^2 x_j^2 + \sigma^2 d_j^2 \right\}.
  \end{equation*}
\end{proposition}
\noindent
See Appendix~\ref{sec:proof-gsm-linear} for the proof of the result.
%
%
Then a quick proof via saddle point duality characterizes the
linear minimax risk for the Gaussian
sequence model (see Appendix~\ref{sec:proof-finite-dim-linear-risk-gsm}).
\begin{corollary}
  \label{corollary:finite-dim-linear-risk-gsm}
  Let $\optdomain \subset \R^n$ be an orthosymmetric
  quadratically convex body. Then
  \begin{equation*}
    \inf_A \maxgsmrisk(A, \optdomain)
    = \sup_{q \in \qhull(\optdomain)}
    \sum_{j = 1}^n \frac{\sigma^2 q_j^2}{q_j^2 + \sigma^2}.
  \end{equation*}
\end{corollary}
\noindent
An approximation argument extends
Corollary~\ref{corollary:finite-dim-linear-risk-gsm} to sequence space (see
Appendix~\ref{sec:proof-linear-risk-general-gsm}).

\begin{corollary}
  \label{corollary:linear-risk-general-gsm}
  Let $\optdomain \subset \R^\N$ be orthosymmetric and compact for
  $\ell_2(\N)$. Then
  \begin{equation*}
    \inf_A \maxgsmrisk(A, \optdomain)
    = \sup_{q \in \qhull(\optdomain)}
    \sum_{j = 1}^\infty \frac{q_j^2 \sigma^2}{q_j^2 + \sigma^2}.
  \end{equation*}
\end{corollary}


With these results in place, we can provide an alternative proof that
if $\optdomain$ is quadratically convex, then
the
linear minimax risk for the Gaussian sequence model is equal to linear
minimax risk over all rectangular subsets of $\optdomain$
(recovering \citep[Theorem
  7]{DonohoLiMa90}).  Let $\rectangles(\optdomain)$ denote the collection of
orthosymmetric rectangular subsets of $X$, that is, sets of the form
$[-\optvar_j, \optvar_j]_{j \ge 1} \subset \optdomain$.
\begin{corollary}
  \label{corollary:qc-gsm}
  Let $\optdomain$ be quadratically convex, compact, and orthosymmetric. Then
  \begin{equation*}
    \inf_{\what{x} = Ay} \sup_{x \in X} \E[\ltwos{\what{x} - x}^2]
    = \inf_A \maxgsmrisk(A, X)
    = \sup_{H \in \rectangles(X)} \inf_A \maxgsmrisk(A, H)
    = \sigma^2 \sup_{x \in X} \sum_{j \ge 1} \frac{x_j^2}{x_j^2 + \sigma^2}.
  \end{equation*}
\end{corollary}
\begin{proof}
  Because $\optdomain = \qhull(\optdomain)$,
  Corollary~\ref{corollary:linear-risk-general-gsm}
  implies
  $\inf_A \maxgsmrisk(A, \optdomain)
  = \sup_{\optvar \in \optdomain} \sum_{j \ge 1} \frac{\sigma^2 \optvar_j^2}{
    \optvar_j^2 + \sigma^2}$.
  For any $\optvar \in \optdomain$,
  the hyperrectangle $H = \bigotimes_{j \ge 1} [-|\optvar_j|, |\optvar_j|]$
  satisfies
  $\inf_A \maxgsmrisk(A, H) = \sum_{j \ge 1}
  \frac{\sigma^2 \optvar_j^2}{\optvar_j^2 + \sigma^2}$,
  again by Corollary~\ref{corollary:linear-risk-general-gsm}.
  Take a supremum.
\end{proof}

\subsubsection{Fundamental limits for the Gaussian sequence model}

To introduce the ideas for lower bounds we employ in the remainder of the
paper and highlight quadratic
convexity,
we also review some of the fundamental lower and upper bounds in
Gaussian sequence models for general orthosymmetric sets $\optdomain$.
The following
result, which for completeness we prove in
Appendix~\ref{sec:proof-gsm-lecam-lower}, is typical; it provides
worse constants than those available by more careful constructions
(cf.~\cite[Chapter 4]{Johnstone17}), but it introduces
some of our main types of arguments.
\begin{proposition}
  \label{proposition:gsm-lecam-lower}
  Let $\optdomain$ be an orthosymmetric convex set. Then
  for any $\optvar \in \optdomain$,
  \begin{equation*}
    \maxgsmrisk(\optdomain) \ge \frac{1}{10}
    \sum_{j \ge 1} \optvar_j^2 \wedge \sigma^2.
  \end{equation*}
\end{proposition}
\noindent
As an immediate consequence to Proposition~\ref{proposition:gsm-lecam-lower},
we see that linear estimators are minimax rate optimal whenever
$\optdomain$ is quadratically convex.
\begin{corollary}
  \label{corollary:gsm-linear-same-as-nonlinear}
  Let $\optdomain$ be quadratically convex, orthosymmetric, and compact.
  Then
  \begin{equation*}
    \maxgsmrisk(\optdomain) \le \inf_A \maxgsmrisk(A, \optdomain)
    \le 10 \maxgsmrisk(\optdomain).
  \end{equation*}
\end{corollary}
\begin{proof}
  Observe that $\optvar_j^2 \wedge \sigma^2
  \ge \frac{\optvar_j^2 \sigma^2}{\optvar_j^2 + \sigma^2}$,
  and then apply Corollary~\ref{corollary:qc-gsm}
  and Proposition~\ref{proposition:gsm-lecam-lower}.
\end{proof}

\subsubsection{Soft-thresholding and nonlinear estimators}
\label{sec:soft-thresh}

\newcommand{\softthresh}{\mathsf{S}}

Whenever $\optdomain$ is quadratically convex, linear estimators
are (nearly) minimax optimal, so two fundamental questions remain:
first, when do we
indeed require nonlinearity, and second, whether there
is a generally rate-optimal nonlinear method.
Soft-thresholding,
which estimates $\optvar$ by elementwise applying
the soft-thresholding operator
\begin{equation*}
  \softthresh_\lambda(y) \defeq \sign(y) \cdot \hinge{|y| - \lambda},
\end{equation*}
provides this nearly optimal method for Gaussian
sequence estimation.
\citet[Corollary 8.4]{Johnstone17} gives a paradigmatic bound for the
risk of soft thresholding on a single coordinate:
\begin{corollary}
  \label{corollary:soft-threshold}
  Let $y \sim \normal(\optvar, \sigma^2)$ and
  $\delta \in \openleft{0}{1}$. Then for $\lambda = \sqrt{2 \sigma^2
    \log \delta^{-1}}$,
  \begin{equation*}
    \E[(\softthresh_\lambda(y) - \optvar)^2]
    \le \delta \sigma^2 + \left(1 + 2 \log \delta^{-1}\right)
    \cdot \optvar^2 \wedge \sigma^2.
  \end{equation*}
\end{corollary}
\noindent
If $\optdomain \subset \R^n$ and $\delta = \frac{1}{n}$,
the coordinatewise estimator
$\what{\optvar}_j = \softthresh_\lambda(y_j)$, for
$\lambda = \sqrt{2 \sigma^2 \log n}$, obtains
\begin{equation*}
  \E[\ltwos{\what{\optvar} - \optvar}^2]
  \le \sigma^2 + (2 \log n + 1) \sum_{j = 1}^n \optvar_j^2 \wedge \sigma^2,
\end{equation*}
whose risk is essentially optimal (cf.~\cite[Proposition 8.8]{Johnstone17}
and Proposition~\ref{proposition:gsm-lecam-lower}).


One typically considers the risk in Gaussian sequence models as $\sigma^2
\to 0$---for example, with the scaling $\sigma^2 \propto 1/n$ in an
$n$-dimensional model---so that when $\optdomain$ is ``non-pathological''
we wish to develop a general estimator that
generically achieves optimal scaling, as we will also develop for
stochastic optimization in the sequel.
For this, we adapt \citeauthor{DonohoLiMa90}'s
results~\citep[Thm.~12]{DonohoLiMa90} by combining projecting coordinates
to zero and soft-thresholding.
Define
\begin{equation*}
  N(\sigma, \optdomain) \defeq \inf \left\{n \in \N ~ \mbox{s.t.} ~
  \sup_{\optvar \in \optdomain} |\optvar_j| \le \sigma
  ~ \mbox{for~all~} j \ge n\right\},
\end{equation*}
and for $\lambda = \sqrt{2 \sigma^2 \log N(\sigma, \optdomain)}$,
define the truncated soft-thresholding estimator
\begin{equation}
  \label{eqn:truncated-soft-est}
  \what{\optvar}_j = \begin{cases} \softthresh_\lambda(y_j)
    & \mbox{if}~ j \le N(\sigma, \optdomain) \\
    0 & \mbox{otherwise}.
  \end{cases}
\end{equation}
We have the following corollary, whose proof we defer to
appendix~\ref{sec:proof-soft-thresh-minimax}.
\begin{corollary}
  \label{corollary:soft-thresh-minimax}
  The truncated soft-thresholding estimator~\eqref{eqn:truncated-soft-est}
  satisfies
  \begin{equation*}
    \E[\ltwos{\what{\optvar} - \optvar}^2]
    \le \sigma^2 + (2 \log N(\sigma, \optdomain) + 1) \sum_{j \ge 1}
    \optvar_j^2 \wedge \sigma^2.
  \end{equation*}
  If $N(\sigma, \optdomain)$ is polynomial in $1/\sigma$
  as $\sigma \to 0$, there exists $C(\sigma) \le O(1) \log \frac{1}{\sigma}$
  such that
  \begin{equation*}
    \gsmrisk(\what{\optvar}, \optdomain)
    \le C(\sigma) \cdot \maxgsmrisk(\optdomain).
  \end{equation*}
\end{corollary}

The assumption that $N(\sigma, \optdomain)$ is polynomial in $1/\sigma$ as
$\sigma \to 0$ is typically lenient; any $\optdomain \subset \R^n$ evidently
satisfies it, and so do sets contained in $\ell_p$ bodies $\{\optvar \in
\R^\N \mid \sum_{j = 1}^\infty a_j |\optvar_j|^p \le 1\}$ so long as $a_j
\to \infty$ polynomially quickly in the index $j$.  In brief, the
(nonlinear) truncated soft-thresholding
estimator~\eqref{eqn:truncated-soft-est} is nearly minimax rate-optimal: to
within a logarithmic factor it achieves the minimax optimal rate for any
``sufficiently compact'' set $\optdomain$. Moreover,
because
$\half (a \wedge b) \le \frac{ab}{a + b}
\le a \wedge b$ for $a, b \ge 0$,
the difference between the quantities
\begin{equation*}
  \sup_{\optvar \in \optdomain} \sum_j \optvar_j^2 \wedge \sigma^2
  ~~ \mbox{and} ~~
  \sup_{\optvar \in \qhull(\optdomain)}
  \sum_j \optvar_j^2 \wedge \sigma^2
\end{equation*}
evidently determines whether nonlinear estimators are necessary.

\subsection{Minimax rates for convex stochastic optimization}
\label{sec:minimax-rates-convex-stochastic}

We turn to stochastic optimization.
We measure the complexity of problem
families in two familiar ways: stochastic minimax complexity and
regret~\cite{NemirovskiYu83, AgarwalBaRaWa12, CesaBianchiLu06, Duchi18}. Let
$\optdomain \subset \R^n$ be a closed convex set, $\statdomain$ a sample
space, and $\mc{F}$ a collection of functions
$\f:\R^n\times\statdomain\to\R$.  For a collection $\mc{P}$ of distributions
over $\statdomain$, recall~\eqref{eqn:problem} that $\ff_P(\optvar) \defeq
\int \f(\optvar, \statval) dP(\statval)$ is the expected loss of the point
$\optvar$. Then the \emph{minimax stochastic risk} is
\begin{equation*}
  \minimaxs(\optdomain, \mc{F}, \mc{P})
  \defeq \inf_{\what{\optvar}_k} \sup_{\f \in \mc{F}}
  \sup_{P \in \mc{P}} \E\left[ \ff_P(\what{\optvar}_k(\statrv_1^k))
    - \inf_{\optvar \in \optdomain} \ff_P(\optvar) \right],
\end{equation*}
where the expectation is taken over $\statrv_1^k \simiid P$ and the infimum
ranges over all measurable functions $\what{\optvar}_k$ of $\statdomain^k$.
The average \emph{minimax regret} provides a related notion, instead taking
a supremum over samples $\statval_1^k \in \statdomain^k$ and measuring
instantaneous losses.
In this case, an algorithm consists of a sequence of
decisions $\what{\optvar}_1, \what{\optvar}_2, \ldots, \what{\optvar}_k$,
where $\what{\optvar}_i$ is chosen conditional on samples
$\statval_1^{i-1}$, so that
\begin{equation*}
  \minimaxr(\optdomain, \mc{F}, \statdomain)
  \defeq \inf_{\what{\optvar}_{1:k}} \sup_{\f \in \mc{F},
    \statval_1^k \in \statdomain^k}
  \sup_{\optvar \in \optdomain} \frac{1}{k}
  \sum_{i = 1}^k \left[\f\big(\what{\optvar}_i\big(\statval_1^{i-1}\big),
    \statval_i\big)
    - \f\!\left(\optvar, \statval_i\right)\right].
\end{equation*}
In each definition, the point estimates $\what{\optvar}$ may lie outside the
constraint set $\optdomain$, making our lower bounds stronger results;
in the language of learning
theory~\cite{CesaBianchiLu06, ShalevShSrSr09}, we allow ``improper''
predictions, but compare to the best elements $\optvar \in \optdomain$.
(In our cases, this does not change regret by more than a
constant factor.)
As online-to-batch conversions make
clear~\cite{CesaBianchiCoGe04}, we always have $\minimaxs \le \minimaxr$;
thus we typically provide lower bounds on $\minimaxs$ and upper bounds on
$\minimaxr$. In many of the cases we consider, these quantities are
essentially equivalent~\cite[e.g.][]{ShalevShSrSr09}, though in cases where we
wish to provide explicit lower bounds on algorithms we typically use regret.

Lipschitz continuity properties form a central lever for demonstrating
convergence in general (potentially non-smooth) stochastic convex
optimization~\cite{NemirovskiYu83, AgarwalBaRaWa12, Duchi18}, and
consequently, we study functions for which a norm $\gamma$ on $\R^n$
($\gamma$ as a mnemonic for gradient) specifies these:
\begin{equation}
  \label{eqn:func-family}
  \mc{F}^{\gamma, r}
  \defeq \left\{\f : \R^n \times \statdomain \to \R
  \mid
  \mbox{for~all~} \optvar \in \R^n,
  ~ g\in\partial_\optvar \f(\optvar, \statval), ~ \gamma(g) \leq r \right\},
\end{equation}
The gradient bound condition $\gamma(g) \le r$ is equivalent to the
Lipschitz condition $|\f(\optvar, \statval) - \f(\optvar', \statval)| \le
r\gamma^*(\optvar - \optvar')$, where $\gamma^*$ is the dual norm to
$\gamma$.  We use the shorthands
\begin{equation*}
  \minimaxr(\optdomain, \mathcal{F}) := \sup_\mathcal{\statrv} \minimaxr(\optdomain,
  \mathcal{F}, \mathcal{\statrv}) \mbox{~~and~~}
  \minimaxs(\optdomain, \mathcal{F}) := \sup_{\mathcal{\statrv}}
  \sup_{\mathcal{P}\subset \mathcal{P}(\mathcal{\statrv})} \minimaxs(\optdomain,
  \mathcal{F}, \mathcal{P})
\end{equation*}
as well as
\begin{equation*}
  \minimaxr(\optdomain, \gamma) := \sup_\mathcal{\statrv} \minimaxr(\optdomain,
  \mathcal{F}^{\gamma, 1}, \mathcal{\statrv}) \mbox{~~and~~}
  \minimaxs(\optdomain, \gamma) := \sup_{\mathcal{\statrv}}
  \sup_{\mathcal{P}\subset \mathcal{P}(\mathcal{\statrv})} \minimaxs(\optdomain,
  \mathcal{F}^{\gamma, 1}, \mathcal{P})
\end{equation*}
as the Lipschitzian properties of $\mc{F}$ in relation to $\optdomain$
determine the minimax regret and risk.
We now discuss the main techniques for lower-bounding the (stochastic)
minimax risk, then recapitulate the familiar bounds for online
optimization coming from mirror descent and dual averaging,
before we move to the more precise results we prove using these techniques
and known bounds.


\subsubsection{Techniques for minimax optimization lower bounds}
\label{sec:est-to-test}

We review the Le Cam, Fano and Assouad
methods~\cite{Assouad83,Yu97, AgarwalBaRaWa12, Wainwright19} for proving lower
bounds.
Each reduces estimation to testing then uses information theoretic tools to
bound the probability of error in the hypothesis tests.

\paragraph{Le Cam and Fano methods}

A standard reduction from estimation to testing
(see \cite[Ch.~5]{Duchi18}) forms the basis for most of our proofs.
To set the stage,
let $\mc{P}$ be a collection of distributions over $\statdomain$ and
$L:\R^n \times \mc{P} \to \R_+$ satisfy
\begin{equation*}
  \inf_{\optvar\in\optdomain} L(\optvar, P) = 0
  ~~ \mbox{for~} P \in \mc{P}.
\end{equation*}
For distributions $P, Q \in \mc{P}$, define the separation
\begin{equation*}
  \mathsf{sep}_L(P, Q) := \sup\left\lbrace \delta \ge 0
  \;\middle|\;\; \mbox{for all~}
  \optvar,
  \begin{array}{c} L(\optvar, P) \leq \delta ~\mbox{implies}~
    L(\optvar, Q) \geq \delta \\
    L(\optvar, Q) \leq \delta ~\mbox{implies}~
    L(\optvar, P) \geq \delta
  \end{array}\right\rbrace.
\end{equation*}
We say a collection $\{P_v\}_{v \in \mc{V}} \subset \mc{P}$ indexed
by $\mc{V}$ is $\delta$-separated
if $\mathsf{sep}_L(P_v, P_{v'}) \ge \delta$ for
$v \neq v'$, and any such family induces a testing lower bound:
\begin{lemma}[From estimation to testing]
  \label{lemma:est-to-test}
  Let $\{P_v\}_{v \in \mc{V}} \subset \mc{P}$ be $\delta$-separated.
  Then for any joint distribution $\P$ constructed by
  first drawing $V \in \mc{V}$ and then, conditional on
  $V = v$, sampling $\statrv_1^k \simiid P_v$,
  \begin{equation*}
    \inf_{\what{\optvar}} \sup_{P\in\mc{P}} \E_P L(\what{\optvar}(\statrv_1^k), P) \geq
    \delta \inf_{\psi}\P(\psi(\statrv_1^k) \neq V).
  \end{equation*}
\end{lemma}

With this, the classical Le Cam and Fano methods are straightforward
combinations of Lemma~\ref{lemma:opt-to-est} with (respectively) Le Cam's
lemma~\cite[Lemma 1]{Yu97} and Fano's inequality~\cite[Theorem
2.10.1]{CoverTh06}.
\begin{proposition}[Le Cam's method]
  \label{prop:le-cam}
  Let $P_0, P_1 \in \mc{P}$ and
  $\delta > 0$ satisfy $\mathsf{sep}_L(P_0, P_1) \ge \delta$.
  Then
  \begin{equation*}
    \inf_{\what{\optvar}}\sup_{P\in\mc{P}} \E_P L(\what{\optvar}(\statrv_1^k), P) \ge
    \frac{\delta}{2}\left(1 - \tvnormbig{P^k_0 - P^k_1}\right).
  \end{equation*}
\end{proposition}

\begin{proposition}[Fano's method]
  \label{prop:fano}
  Let $\mc{V}$ be a finite index set for which $\{P_v\}_{v\in\mc{V}}$ is
  $\delta$-separated.
  Then
  \begin{equation*}
    \inf_{\what{\optvar}}\sup_{P\in\mc{P}} \E_P L(\what{\optvar}(\statrv_1^k), P) \ge \delta\left(1 - \frac{\mathsf{I}(\statrv_1^k;V) + \log 2}{\log |\mc{V}|} \right).
  \end{equation*}
\end{proposition}

With these tools, minimax lower bounds on the stochastic risk $\minimaxs$ in
Section~\ref{sec:prelim} follow by (i) demonstrating an appropriate loss $L$
and (ii) separation.
We construct an object that simultaneously accomplishes both of these:
for an index set $\mc{V}$ with $\{P_v\}_{v \in \mc{V}} \subset \mc{P}$,
objective $F$ and population objectives
$\ff_v(\optvar) \defeq \E_{P_v}[\f(\optvar, \statrv)]$, define
\begin{equation}
  \label{eqn:optization-distance}
  \dopt(v, v', \optdomain) \defeq \inf_{\optvar}
  \left[ f_v(\optvar) + f_{v'}(\optvar) - \inf_{\optvar \in \optdomain}
  f_v(\optvar)
  - \inf_{\optvar \in \optdomain}
  f_{v'}(\optvar) \right].
\end{equation}
When $\inf_{\optvar \in \optdomain} \ff_v(\optvar) = \inf_{\optvar} \ff_v(\optvar)$,
which we require because we allow improper predictions,
this quantity guarantees appropriate separation.
Indeed, under this condition,
the loss $L(\optvar, P) \defeq \ff_P(\optvar) - \inf_{\optvar \in \optdomain}
\ff_P(\optvar)$ satisfies $\inf_\optvar L(\optvar, P) = 0$
for $P \in \{P_v\}_{v \in \mc{V}}$.
Moreover, if $L(\optvar, P_v) \le \half \dopt(v, v', \optdomain)$, then
evidently $L(\optvar, P_{v'}) \ge \half \dopt(v, v', \optdomain)$,
and so $\mathsf{sep}_L(P_v, P_{v'}) \ge \half \dopt(v, v', \optdomain)$.
Lemma~\ref{lemma:est-to-test} then implies
the following result, which is essentially
present in the paper~\cite{AgarwalBaRaWa12} (cf.~\cite{Duchi18}).

\begin{lemma}[From optimization to function estimation]
  \label{lemma:opt-to-est}
  Let $\{P_\packval\}_{\packval \in \packset} \subset \mc{P}$
  be a collection of distributions on $\statdomain$ and
  $\mc{F}$ be
  a collection a functions $\R^n\times\statdomain \to \R$,
  where for some $F \in \mc{F}$, the optimization
  distance~\eqref{eqn:optization-distance} satisfies
  $\dopt(v, v', \optdomain) \ge \delta \ge 0$
  for all $v \neq v' \in \mc{V}$.
  Then
  \begin{equation*}
    \minimaxr(\optdomain, \mc{F})
    \ge \minimaxs(\optdomain, \mc{F}, \mc{P})
    \ge \minimaxs(\optdomain, \{F\}, \mc{P})
    \ge \frac{\delta}{2}\inf_{\psi}\P(\psi(\statrv_1^k) \neq V).
  \end{equation*}
\end{lemma}

Our general strategy for proving lower bounds on $\minimaxs$ is as follows:
\begin{itemize}
	\item Choose a function $F\in\mc{F}$ and define $\mc{V}$ and
	$\{P_v\}_{v\in\mc{V}} \subset \mc{P}$ such that
	$\dopt(v, v', \optdomain) \ge \delta > 0$.
	\item Lower bound the testing error
	$\inf_{\psi} \P(\psi(\statrv_1^k) \neq V)$, and choose
	the largest separation $\delta$ to make this testing error a positive
	constant.
\end{itemize}

To showcase this proof technique, we prove that
minimax stochastic risk for $1$-dimensional optimization has
lower bound $1 / \sqrt{k}$; we use this to address technicalities
in later proofs.

\begin{lemma}
  \label{lem:1-d}
  Let
  $\mc{F}^{n=1} = \{\f:\R \times \statdomain \to \R \mid \f(\cdot, \statval)
  ~ \mbox{is} ~ \mbox{convex} ~ \mbox{and}~ 1\mbox{-Lipschitz}\}$. Then
  \begin{equation*}
    \minimaxs([-1, 1], \mc{F}^{n=1}) \ge \frac{1}{4 \sqrt{6 k}}.
  \end{equation*}
\end{lemma}
\begin{proof}
  Let $\optdomain = [-1, 1]$ and $\statdomain = \{\pm 1\}, \mc{V} = \{\pm 1\}$.
  To see the separation condition, let $F(\optvar, \statval) \defeq |\optvar
  - \statval|$.
  For
  $\delta \in [0, \half]$, we define $P_v$ so that for $\statrv\sim P_v$,
  \begin{equation*}
    P_\packval(\statrv = \packval) = \frac{1 + \delta}{2},
    ~~
    P_\packval(\statrv = -\packval) = \frac{1 - \delta}{2}.
  \end{equation*}
  Then
  $f_{v}(\optvar) = \frac{1+\delta}{2}|\optvar - v| +
  \frac{1-\delta}{2}|\optvar + v|$ and
  $\inf_\optvar f_{v}(\optvar) = \inf_{\optvar \in [-1, 1]}
  f_\packval(\optvar) = \frac{1-\delta}{2}$.
  To lower bound the separation, note that
  \begin{equation*}
    f_1(\optvar) + f_{-1}(\optvar) - \inf_{\optdomain}f_1 - \inf_{\optdomain}f_{-1} =
    |\optvar - 1| + |\optvar+1| - (1-\delta) \ge \delta.
  \end{equation*}
  This yields $\dopt(1, -1, \optdomain) \ge \delta$.
  
  We lower bound the testing error via Proposition~\ref{prop:le-cam}:
  by Pinsker's inequality that $\tvnorm{P - Q}^2 \le \half \dkl{P}{Q}$
  for any distributions $P$ and $Q$, and that
  $\dkl{P^k}{Q^k} = k \dkl{P}{Q}$, we obtain
  \begin{equation*}
    \inf_{\psi:\statdomain^k\to \{\pm 1\}}\P(\psi(\statrv_1^k)\neq V) =
    \half\left(1 - \tvnormbig{P_1^k - P_{-1}^k}\right) \ge
    \half\left(1-\sqrt{\frac{k}{2}\dkl{P_1}{P_{-1}}}\right).
  \end{equation*}
  Noting that
  $\dkl{P_1}{P_{-1}} = \delta\log\frac{1+\delta}{1-\delta} \le 3\delta^2$
  for $\delta \in [0, \half]$,
  setting $\delta = 1 / \sqrt{6k}$ yields the result.
\end{proof}


\paragraph{The Assouad method}
Assouad's method~\cite{Assouad83} reduces the problem
of estimation (or optimization) to one of multiple binary hypothesis tests.
In this case, we index a set of distributions $\mc{P} = \{P_v\}_{v \in
  \mc{V}}$ on a set $\statdomain$ by the hypercube $\mc{V} = \{\pm 1\}^n$. For
a function $\f : \R^n \times \statdomain \to \R$, we define
$\ff_v(\optvar) \defeq \E_{P_v}[\f(\optvar, \statrv)]$. Then for
a vector $\delta \in \R^n_+$, 
following \citet[Lemma 5.3.2]{Duchi18}, we say that
the functions $\{\ff_v\}$ induce a $\delta$-separation in Hamming
metric if
\begin{equation}
	\label{eqn:hamming-separation}
	\ff_v(\optvar) - \inf_{\optvar \in \optdomain} \ff_v(\optvar)
	\ge \sum_{j = 1}^n \delta_j \indic{\sign(\optvar_j) \neq v_j}.
\end{equation}
\begin{lemma}[Generalized Assouad's method]\label{lem:assouad}
  Let $\statrv_1^k \simiid P_V$, where $V \sim \mathsf{Uniform}(\{\pm 1\}^n)$.  Define
  the averages
  \begin{equation*}
    \P_{+j} \defeq \frac{1}{2^{n-1}} \sum_{v : v_j = 1} P_v^k
    ~~ \mbox{and} ~~
    \P_{-j} \defeq \frac{1}{2^{n-1}} \sum_{v : v_j = -1} P_v^k.
  \end{equation*}
  Assume that the collection $\{\ff_v\}$ for $\ff_v =
  \E_{P_v}[\f(\cdot, \statrv)]$ induces a
  $\delta$-separation~\eqref{eqn:hamming-separation}.
  Then
  \begin{equation*}
    \minimaxs(\optdomain, \{F\}, \mc{P})
    \ge \half \sum_{j = 1}^n \delta_j \left(1 - \tvnorm{\P_{+j} - \P_{-j}}
    \right).
  \end{equation*}
\end{lemma}

\subsubsection{Stochastic gradient methods, mirror descent, and regret}

Let us briefly review the canonical algorithms for solving the
problem~\eqref{eqn:problem} and their associated convergence guarantees.
For an algorithm outputing points $\optvar_1, \ldots, \optvar_k$, the
\emph{regret} on the sequence $\f(\cdot, \statval_i)$ with respect to a
point $\optvar$ is
\begin{equation*}
  \regret_k(\optvar) \defeq \sum_{i = 1}^k [
    \f(\optvar_i, \statval_i) - \f(\optvar, \statval_i)].
\end{equation*}
Recalling the definition $\breg{h}{\optvar}{\optvar_0} = h(\optvar) -
h(\optvar_0) - \inprod{\nabla h(\optvar_0)}{\optvar - \optvar_0}$ of the
Bregman divergence, the mirror descent
algorithm~\cite{NemirovskiYu83,BeckTe03} iteratively sets
\begin{equation}
  \label{eqn:smd}
  g_i \in \partial_\optvar \f(\optvar_i, \statval_i)
  ~~ \mbox{and~updates}~~
  \optvar^{\mathsf{MD}}_{i + 1}
  \defeq \argmin_{\optvar}
  \left\{g_i^\top \optvar + \frac{1}{\stepsize} \bregman{h}{\optvar}{\optvar_i}
  \right\}
\end{equation}
where $\stepsize > 0$ is a stepsize.
When the function $h$ is $1$-strongly convex with respect to a norm
$\norm{\cdot}$ with dual norm $\dnorm{\cdot}$ over $\dom h$,
the iterates~\eqref{eqn:smd} and the iterates~\eqref{eq:da} of
dual averaging satisfy
(cf.~\cite{BeckTe03, CesaBianchiLu06, Nesterov09})
\begin{equation}
  \label{eqn:md-regret}
  \regret_k(\optvar)
  \le \frac{\bregman{h}{\optvar}{\optvar_0}}{\stepsize}
  + \frac{\stepsize}{2} \sum_{i \le k} \dnorm{g_i}^2
  ~~\mbox{for all}~ \optvar \in \dom h.
\end{equation}
The choice $h(\optvar) = \frac{1}{2} \ltwo{\optvar}^2$ recovers the
classical stochastic gradient method, while the $p$-norm
algorithms~\cite{Gentile03, Shalev07, NemirovskiJuLaSh09, Duchi18}, defined
for $1 < p \le 2$, use $h(\optvar) = \frac{1}{2(p - 1)} \norm{\optvar}_p^2$;
each is strongly convex with respect to the $\ell_p$-norm $\norm{\cdot}_p$.
If $G = \{g \in \partial_\optvar \f(\optvar, \statval) \mid \optvar \in X,
\statval \in \statdomain\}$ denotes the set of possible subgradients, the
regret guarantee~\eqref{eqn:md-regret} becomes
\begin{equation*}
  \regret_k(\optvar)
  \le \frac{\norm{\optvar}_p^2}{2 (p - 1) \stepsize}
  + \frac{k \stepsize}{2} \sup_{g \in G} \norm{g}^2_q
\end{equation*}
if $x_0 = 0$ and $q = \frac{p}{p - 1}$ is conjugate to $p$. The
choice $\stepsize = \frac{1}{\sqrt{k}} \sup_{\optvar \in \optdomain}
\norm{\optvar}_p / \sup_{\gamma(g) \le 1} \norm{g}_q$ gives
the following now standard minimax regret
bound~\cite[cf.][Corollary 2.18]{Shalev12}.
\begin{proposition}
  \label{prop:rate-md}
  Let $\optdomain$ be closed convex orthosymmetric,
  $\gamma$ a norm, and $1 < p \le 2$,
  $q = \frac{p}{p-1}$.
  Mirror descent with $h(\optvar) := \frac{1}{2(q-1)} \norm{\optvar}_p^2$
  and stepsize $\stepsize = \frac{\sup_{\optvar \in \optdomain}
    \norm{\optvar}_p}{ \sqrt{k(p-1)}\sup_{\gamma(g) \le 1} \norm{g}_q}$ achieves
  regret
  \begin{equation*}
    \minimaxr(\optdomain, \gamma) \le
    \frac{\sup_{\optvar\in\optdomain} \norm{\optvar}_p
      \sup_{g\in\ball{\gamma}{0}{1}} \|g\|_{q}}{\sqrt{k(p - 1)}}.
  \end{equation*}
\end{proposition}

As we state in our definitions of minimax risk and regret, we
allow the point estimates to lie outside the constraint set $\optdomain$,
though as with the dual averaging update~\eqref{eq:da}, one
may incorporate constraints by replacing
$h$ with $h + \convexindic{\optdomain}$.
Even without such constraints,
the form~\eqref{eqn:md-regret} and Proposition~\ref{prop:rate-md}
still capture the
regret for all common constraint sets $\optdomain$~\cite{Shalev07}.
\begin{example}[$p$-norm algorithms]
  Let $n \ge 3$.
  With $h(\optvar) = \frac{1}{2(p-1)} \norm{\optvar}_p^2$,
  taking $p = 1 +
  \frac{1}{\log n}$ and $q = 1 + \log n$,
  we have $\norm{g}_q \le e \linf{g}$ for all $g \in \R^n$,
  while $\norm{x}_p \le \lone{x}$.
  Assuming that gradients $g$ belong to the $\ell_\infty$-ball of radius 1,
  inequality~\eqref{eqn:md-regret} and Proposition~\ref{prop:rate-md} show
  that for any $\optdomain \subset \R^n$, the stepsize $\stepsize = e^{-1}
  \sqrt{\log n / k} \sup_{\optvar \in \optdomain} \lone{\optvar}$
  yields
  \begin{equation*}
    \sup_{\optvar \in \optdomain} \regret_k(\optvar)
    \le
    e \sqrt{k \log n}
    \sup_{x \in X} \lone{x}.
  \end{equation*}
  Taking the comparator class $\optdomain$ to be the $\ell_1$-ball gives
  regret $e \cdot \sqrt{k \log n}$.
\end{example}

\subsubsection{Euclidean gradient methods and quadratic convexity}

We frequently focus on distance generating functions of the form $h(\optvar)
= \half \inprod{\optvar}{A\optvar}$ for a fixed positive semi-definite
matrix $A$.
For an arbitrary $A$, we will refer to these methods as
\textbf{\eucgrad} and for a diagonal $A$ as \textbf{\dscaled}.
In this case,
the mirror descent update is the stochastic gradient update with
$\invert{A}g$, where $g$ is a stochastic subgradient.
We refer to all
such methods as \textbf{methods of linear type}, as their update sequence
guarantees the linearity
\begin{equation*}
  x_k = -A^{-1} \sum_{i = 1}^{k-1} g_i.
\end{equation*}
With updates of this type, convex quadratic hulls of both the domain
$\optdomain$ and the collections of gradients $\partial \f(x, \statval)$
appear in convergence bounds (the sequel shows this appearance is no
accident).

The \dscaled (componentwise
re-scaling of the subgradients) are equivalent to using
$h_\Lambda(\optvar) \defeq \half \dotp{\optvar}{\Lambda \optvar}$ for
$\Lambda = \mathrm{diag}(\lambda) \succeq 0$ in the mirror descent
update~\eqref{eqn:smd}.
In this case, for any norm $\gamma$ on the gradients,
the minimax regret bound~\eqref{eqn:md-regret} becomes
\begin{equation*}
  \sup_{\optvar\in\optdomain} \regret_{k, \Lambda}(\optvar) \le
  \frac{1}{2k} \left[
    \sup_{\optvar \in \optdomain} \optvar^\top \Lambda \optvar
    + \sum_{i \le k} g_i^\top \Lambda^{-1} g_i \right]
  \le \frac{1}{2 k}
  \left[\sup_{\optvar\in\optdomain} \optvar^\top\Lambda\optvar
    + k \sup_{g\in\ball{\gamma}{0}{1}}g^\top \Lambda^{-1} g \right].
\end{equation*}
The rightmost term upper bounds the minimax regret, so we may take an
infimum over $\Lambda$, yielding
\begin{equation}
  \label{eq:adagrad-regret}
  \minimaxr(\optdomain, \gamma) \leq
  \frac{1}{2k} \inf_{\lambda \succeq 0}
  \sup_{\optvar\in\optdomain}
  \sup_{g\in\ball{\gamma}{0}{1}} \bigg[
    \sum_{j\leq n}\lambda_j
    \optvar_j^2 + k\sum_{j\leq n} \frac{1}{\lambda_j} g^2_j \bigg]
\end{equation}
The regret bound~\eqref{eq:adagrad-regret} holds without assumptions on
$\optdomain$ or $\gamma$.
The key is that
strong duality allows us to relate this quantity to the quadratic
convex hulls (recall Sec.~\ref{sec:prelim}) of $\optdomain$ and
$\ball{\gamma}{0}{1}$, which will in turn allow
us to prove minimax optimality of Euclidean procedures in the coming
sections.

\begin{proposition}
  \label{prop:minimax-thm}
  Let $G, \optdomain \subset \R^n$ be compact convex sets. Then
  \begin{equation*}
    \inf_{\lambda \succ 0} \sup_{\optvar \in \optdomain, g \in G}
    \left\{ \lambda^\top \optvar^2 + \left(\frac{1}{\lambda}\right)^\top
    g^2\right\}
    =
    \sup_{\optvar \in \qhull(\optdomain),
    g \in \qhull(G)} \inf_{\lambda \succ 0}
    \left\lbrace\lambda^\top \optvar^2 + \left(\frac{1}{\lambda}\right)^\top
    g^2\right\rbrace.
  \end{equation*}
\end{proposition}
\begin{proof}
  The
  (weighted) squared $2$-norm becomes a linear functional when lifted to the
  squared sets $\optdomain^2 \defeq \{\optvar^2 \mid \optvar \in
  \optdomain\}$ and $G^2$.
  Indeed, defining $J(\tau, w, \lambda) \defeq \lambda^\top \tau +
  (\frac{1}{\lambda})^\top w$, the function $J$ is concave-convex: it is
  linear (a fortiori concave) in $(\tau, w)$ and convex in $\lambda \succ 0$.
  Thus, using that the set $\{\lambda\in\R^n_+ \}$ is convex and
  $\squaredhull(\optdomain) \times \squaredhull(G)$ is convex compact
  (because $\optdomain$ and $G$ compact), Sion's
  minimax theorem~\cite{Sion58} implies
  \begin{align*}
    \inf_{\lambda \succ 0}
    \sup_{\optvar \in \optdomain, g \in G}
    \left\{ \lambda^\top \optvar^2 +
    \left(\frac{1}{\lambda}\right)^\top g^2\right\}
    & = \inf_{\lambda \succ 0}
    \sup_{\tau \in \squaredhull(\optdomain),
      w \in \squaredhull(G)}
    \left\{ \lambda^\top \tau +
    \left(\frac{1}{\lambda}\right)^\top w\right\}
    \\
    & =
    \sup_{\tau \in \squaredhull(\optdomain),
      w \in \squaredhull(G)}
    \inf_{\lambda \succ 0}
    \left\{\lambda^\top \tau + \left(\frac{1}{\lambda}\right)^\top w \right\}.
  \end{align*}
  Replacing $\tau$ with $\optvar^2$ and $w$ with $v^2$ gives the result.
\end{proof}

Proposition~\ref{prop:minimax-thm} provides a powerful hammer
for diagonally scaled Euclidean algorithms, as we
can choose an optimal scaling for any \emph{fixed} pair $\optvar, g$, taking
a worst case over such pairs:
\begin{corollary}
  \label{cor:ub} Let $\optdomain$ be a convex
  compact set and $\gamma$ a norm. Then
  \begin{equation*}
    \minimaxr(\optdomain, \gamma)
    \leq \frac{1}{\sqrt{k}}
    \sup_{g \in \qhull(\ball{\gamma}{0}{1})}
    \sup_{\optvar\in \qhull(\optdomain)}
    \optvar^\top g,
  \end{equation*}
  and \dscaled achieve this regret.
\end{corollary}
\begin{proof}
  Apply Proposition~\ref{prop:minimax-thm}
  to the regret~\eqref{eq:adagrad-regret}
  and use $\inf_{\lambda > 0} a\lambda + b / \lambda
  = 2\sqrt{ab}$ for $a, b \ge 0$.
\end{proof}

Corollary~\ref{cor:ub} allows concrete upper bounds on minimax risk and
regret, and it is suggestive that the quadratic hulls of (respectively) the
comparator class $\optdomain$ and subgradient sets relate essentially to
convergence rates, similarly to the case with Gaussian sequence models.
The remainder of the paper develops the analogy of linear and nonlinear
updates in stochastic optimization problems with those in Gaussian sequence
models, highlighting when methods of linear type are minimax rate optimal,
and when more computational power---we \emph{require}
nonlinearity---is necessary.

\section{Minimax optimality and quadratically convex constraint sets}
\label{sec:qc}

We begin by providing lower bounds on the minimax risk and
matching upper bounds on the minimax regret of convex optimization over
quadratically convex constraint sets, where \dscaled achieve the regret
bounds.
While the analogy with the Gaussian sequence model is nearly
complete, in distinction to the work of~\citeauthor{DonohoLiMa90} (where
results depend solely on the constraints $\optdomain$, as in
Corollary~\ref{corollary:qc-gsm}), our results necessarily depend on the
geometry of the subdifferential. Consequently, we distinguish throughout
this section between quadratically and non-quadratically convex geometry of
the gradients.
To set the stage for our contributions, we begin with the classical case of
$\optdomain = \ball{p}{0}{1}$ with $p\in [2, \infty]$ (so that $\optdomain$
is quadratically convex) and norm $\gamma = \norms{\cdot}_r$ with $r \ge
1$. We then turn to arbitrary quadratically convex constraint sets and first
show results in the case of general quadratically convex norms on the
subgradients. We conclude the section by proving that, when the subgradients
do not lie in a quadratically convex set but lie in a weighted $\ell_r$ ball
(for $r\in[1, 2]$), \dscaled are still minimax rate optimal.

\subsection{A warm-up: $p$-norm constraint sets for $p \ge 2$}
\label{sec:warm-up}

Though the results for the basic case that $\optdomain$ is an
$\ell_p$-ball while the gradients belong to a different $\ell_r$-ball
are special cases of the theorems to come, the proofs
provide intuition for the later results.
We distinguish
between two cases depending on the value of $r$ in the gradient norm.  The case
that $r \in [1, 2]$ corresponds roughly to ``sparse'' gradients, while the case
$r \ge 2$ corresponds to harder problems with dense gradients. We provide
information theoretic proofs of the following two results in
Sec.~\ref{prf:prop-lp-ball-q12} and~\ref{prf:prop-lp-ball-q2infty},
respectively.
\begin{proposition}[Sparse gradients]
  \label{prop:lp-ball-q12}
  Let $\optdomain = \ball{p}{0}{1}$ with $p\geq 2$ and $\gamma(\cdot)
  = \norm{\cdot}_r$
  where $r \in [1, 2]$.
  Then
  \begin{equation*}
    1 \wedge \frac{n^{\half - \frac{1}{p}}}{\sqrt{k}}
    \lesssim \minimaxs(\optdomain, \gamma)
    \le \minimaxr(\optdomain, \gamma)
    \lesssim 1 \wedge \frac{n^{\half - \frac{1}{p}}}{\sqrt{k}}.
  \end{equation*}
\end{proposition}

\begin{proposition}[Dense gradients]
  \label{prop:lp-ball-q2infty} Let
  $\optdomain = \ball{p}{0}{1}$ with $p\geq 2$ and $\gamma(\cdot) = \|\cdot\|_r$
  with $r \geq 2$. Then
  \begin{equation*}
    1 \wedge \frac{n^{\half - \frac{1}{p}} n^{\half - \frac{1}{r}}}{\sqrt{k}}
    \lesssim \minimaxs(\optdomain, \gamma)
    \le \minimaxr(\optdomain, \gamma)
    \lesssim 1 \wedge
    \frac{n^{\half - \frac{1}{p}} n^{\half - \frac{1}{r}}}{\sqrt{k}}.
  \end{equation*}
\end{proposition}

\noindent
In both cases, the stochastic gradient method achieves the regret upper
bound via a straightforward application of the regret
bounds~\eqref{eqn:md-regret} with $h(\optvar) = \half \ltwo{\optvar}^2$ or
Corollary~\ref{cor:ub}; a method of linear type is optimal.

\subsubsection{Proof of Proposition~\ref{prop:lp-ball-q12}}
\label{prf:prop-lp-ball-q12}

We use the general information-theoretic framework of reduction from estimation
to testing presented in Section~\ref{sec:est-to-test} to prove the lower
bound.
Define the constant $c_p = n^{-1/p}$,
which satisfies $c_p \norm{\ones}_p = 1$ as $\norm{\ones}_p = n^{1/p}$.

\paragraph{Separation}
Consider the sample space $\statdomain = \{\pm e_j\}_{j\leq n}$ and
the function
\begin{equation*}
  \f(\optvar, \statval) \defeq \<|\statval|,
  | \optvar - c_p \statval| \>
  = \sum_{j = 1}^n |\statval_j| |\optvar_j - c_p \statval_j|
\end{equation*}
so $\f$
belongs to $\F{\gamma}{1}$
as $\partial_\optvar \f(\optvar, \statval) \subset
\cup_{j = 1}^n [-1, 1] \cdot e_j$ when $\statval \in \statdomain$.
Letting $\delta \in [0, 1/2]$ to be determined,
for $v \in \{\pm 1\}^n$, we define $P_v$ such that for $\statrv\sim P_v$ we
have
\begin{equation*}
  \statrv =
  \begin{cases}
    v_j e_j & ~~\mbox{with probability}~~ \frac{1+\delta}{2 n} \\
    - v_j e_j & ~~\mbox{with probability}~~ \frac{1-\delta}{2 n}.
  \end{cases}
\end{equation*}
We then have
\begin{equation*}
  \ff_\packval(\optvar)
  = \frac{1}{2n}
  \sum_{j = 1}^n (1 + \delta v_j) |x_j - c_p| + (1 - \delta \packval_j)
  |x_j + c_p|,
\end{equation*}
so $\argmin_\optvar \ff_\packval(\optvar)
= c_p \packval \in \ball{p}{0}{1}$,
whence
\begin{equation*}
  \ff_v^* \defeq
  \inf \ff_v = 
  \inf_{\optvar \in \optdomain} \ff_v(\optvar)
  = (1 - \delta) c_p.
\end{equation*}
where $p^\ast$ is such that $1/p + 1/p^\ast = 1$.
For $v, v' \in \{\pm 1\}^n$, we have
\begin{align*}
  \lefteqn{f_v(\optvar) + f_{v'}(\optvar)} \\
  & = \frac{1}{n}
  \sum_{j : \packval_j = \packval_j'}
  \left[(1 + \delta \packval_j)|\optvar_j - c_p|
    + (1 - \delta \packval_j) |\optvar_j + c_p|
    \right]
  +
  \frac{1}{n} \sum_{j : \packval_j \neq \packval'_j}
  \left[|\optvar_j - c_p| + |\optvar_j + c_p|\right],
\end{align*}
whose minimizer is $c_p (\packval + \packval') / 2 \in \ball{p}{0}{1}$.
The optimization distance then satisfies
\begin{equation*}
  \dopt(v,v',\optdomain)
  = \frac{2 c_p}{n} [ (n - \dham(v, v')) (1 - \delta)
    + \dham(v, v')]
  - 2 (1 - \delta) c_p
  = 2 c_p \delta \frac{\dham(v, v')}{n},
\end{equation*}
where $\dham(v, v')$ is the Hamming distance between $v$ and $v'$. The
Gilbert-Varshimov bound~\cite[Example 5.3]{Wainwright19} guarantees the
existence of a packing set $\packset \subset \{\pm 1\}^n$ of cardinality at
least $\exp(n/8)$ satisfying $\dham(\packval, \packval') \ge \frac{n}{4}$
for each $\packval \neq \packval' \in \packset$.
Then
\begin{equation*}
  \dopt(v, v',\optdomain) \ge \frac{c_p \delta}{2}
  = \half n^{-1/p} \delta
  ~~ \mbox{for~all~} v \neq v' \in \mc{V},
\end{equation*}
and applying Lemma~\ref{lemma:opt-to-est} yields
\begin{equation*}
  \minimaxs(\optdomain, \gamma) \ge \frac{\delta}{4}
  n^{-1/p}\inf_{\psi}\P(\psi(\statrv_1^k)\neq V).
\end{equation*}

\paragraph{Bounding the testing error} We bound the testing error with Fano's
inequality and via the trivial
upper bound $\mi{\statrv_1^k}{V}
\le \max_{\packval,\packval'} \dkl{P_\packval^k}{P_{\packval'}^k}$
on the mutual information~\cite[Ch.~14]{Wainwright19}.
Using the identity $\delta\log\frac{1+\delta}{1-\delta} \le 3\delta^2$ for
$\delta \le \half$,
we have $\dkl{P_\packval}{P_{\packval'}} \le 3 \delta^2$,
and so
Fano's inequality (Proposition~\ref{prop:fano}) and
our choice that $|\packset| \ge \exp(n/8)$ imply
\begin{equation*}
	\inf_{\psi}\P(\psi(\statrv_1^k) \neq V) \ge \left(1 -
	\frac{3k \delta^2 + \log 2}{n / 8}\right).
\end{equation*}
In the case that $n \ge 32 \log 2$, choosing $\delta = \sqrt{\frac{n}{48 k}}$
yields the desired lower bound, once
we recognize that $\minimaxs$ is non-increasing in $k$.
In the case that $n < 32 \log 2$, with
$\mc{F}^{n=1}$ as in Lemma~\ref{lem:1-d}, that any $1$-dimensional
optimization problem may be embedded into a $n$-dimensional problem
yields
\begin{equation*}
	\minimaxs(\optdomain, \gamma) \ge \minimaxs([-1, 1], \mc{F}^{n=1})
	\gtrsim \frac{1}{\sqrt{k}}.
\end{equation*}
This gives the lower bound for all $n\in\N$.

To conclude the proof, we establish an upper bound on the minimax regret.
Because $p \ge 2$, $\qhull(\optdomain) = \optdomain = \ball{p}{0}{1}$,
while for $1 \le r \le 2$, $\norm{g}_r \le \ltwo{g}$ and so
$\wb{G} \defeq \qhull(\ball{\gamma}{0}{1}) \subset \ball{2}{0}{1}$.
Then $\sup_{\optvar \in \optdomain}
\sup_{g \in \wb{G}} \optvar^\top g
\le \sup_{\optvar \in \optdomain} \ltwo{\optvar}
= n^{\half - \frac{1}{p}}$.

\subsubsection{Proof of Proposition~\ref{prop:lp-ball-q2infty}}
\label{prf:prop-lp-ball-q2infty}

The proof is similar to Proposition~\ref{prop:lp-ball-q12}, so we forego
some of the details, again letting $c_p = n^{-1/p}$ so that
$c_p \ones \in \ball{p}{0}{1}$.

\paragraph{Separation}
Let $\eta > 0$ to be determined, and consider the sample space
$\statdomain = \{\pm 1\}^n$ and objectives
\begin{equation*}
  \f(\optvar, \statval)
  \defeq \eta \lone{\optvar - c_p \statval}.
\end{equation*}
For $v \in \{\pm 1\}^n$, we define $P_v$ so
that $\statrv\sim P_v$ has independent coordinates satisfying
\begin{equation*}
  P_\packval(\statrv_j = \packval_j)
  = \frac{1 + \delta}{2},
  ~~
  P_\packval(\statrv_j = -\packval_j) = \frac{1 - \delta}{2}.
\end{equation*}
This yields
\begin{equation*}
  \ff_v(\optvar) = \frac{\eta (1 + \delta)}{2}
  \lone{\optvar - c_p v}
  + \frac{\eta (1 - \delta)}{2} \lone{\optvar + c_p v},
\end{equation*}
and $\argmin_\optvar \ff_v(\optvar) = c_p \packval \in \ball{p}{0}{1}$,
with $\ff_v^* = n c_p \eta(1 - \delta)$.
For $\packval \neq \packval'$, a quick calculation yields
\begin{equation*}
  (\ff_\packval + \ff_{\packval'})^*
  = \inf_{\optvar}   (\ff_\packval(\optvar) + \ff_{\packval'}(\optvar))
  = 2 \eta c_p \dham(\packval, \packval')
  + 2 \eta c_p (1 - \delta)(n - \dham(\packval, \packval')).
\end{equation*}
Then considering again the Gilbert-Varshimov packing $\mc{V} \subset \{\pm 1
\}^n$, we have
\begin{equation*}
  \dopt(v, v', \optdomain)
  = 2 \eta c_p \delta \dham(\packval, \packval')
  \ge \half \eta \cdot n^{1 - 1/p}
  ~~ \mbox{for~} \packval \neq \packval'.
\end{equation*}

\paragraph{Bounding the testing error} Note that
for $\delta \le \half$,
\begin{equation*}
  \dkl{P_v}{P_{v'}} = \sum_{j : \packval_j \neq \packval_j'}
  \delta\log\frac{1+\delta}{1-\delta} \le 3n \delta^2,
\end{equation*}
so $\mi{\statrv_1^k}{V} \le 3k n \delta^2$. For $\f$ to remain in
$\F{\gamma}{1}$, we must have that $\eta \norm{\ones}_r \leq 1$;
noting that $\norm{\ones}_r = n^{1/r}$, we choose $\eta =
n^{-1/r}$. In the case that $n \ge 32 \log 2$, taking $\delta =
1/\sqrt{48 k}$ yields the minimax lower-bound
\begin{equation*}
	\minimaxs(\optdomain, \gamma) \gtrsim
	\frac{n^{1 - \frac{1}{p}} n^{-\frac{1}{r}}}{\sqrt{k}} =
	\frac{n^{\frac{1}{2}-\frac{1}{p}} n^{\frac{1}{2}-\frac{1}{r}}}{\sqrt{k}}.
\end{equation*}
In the case that $n < 32\log 2$, we once again refer Lemma~\ref{lem:1-d},
which concludes the proof for the lower bound on the minimax stochastic
risk.

The upper bound again follows immediately via Corollary~\ref{cor:ub},
as both $\optdomain$ and $\ball{\gamma}{0}{1}$ are quadratically
convex, and for $q = \frac{p}{p - 1}$,
\begin{equation*}
  \sup_{\norm{\optvar}_p \le 1}
  \sup_{\norm{g}_r \le 1} \optvar^\top g
  = \sup_{\norm{g}_r \le 1} \norm{g}_q
  = n^{\frac{1}{q} - \frac{1}{r}}
  = n^{1 - \frac{1}{p} - \frac{1}{r}}.
\end{equation*}



\subsection{General quadratically convex constraints}\label{sec:general}

\newcommand{\rect}{\mathsf{Rec}}

We now turn to the more general case that $\optdomain$ is an arbitrary
orthosymmetric quadratically convex body.
Our techniques build out of the ideas of \citet{DonohoLiMa90} in Gaussian
sequence estimation, where, as in Section~\ref{sec:gsm} the largest
hyperrectangle in $\optdomain$ governs the performance of linear estimators.
As in the previous section, we divide our analysis into cases, depending on
the quadratic convexity of the gradient norm $\gamma$ (analogs of $r
\lessgtr 2$ in Propositions~\ref{prop:lp-ball-q12}
and~\ref{prop:lp-ball-q2infty}).
Our starting point is a general lower bound relying on rectangular
structures in the primal $\optdomain$ and dual gradient spaces.
For the
proposition, we use a specialization of the function
families~\eqref{eqn:func-family} to rectangular sets, where for $M \in
\R_+^n$ we define
\begin{equation*}
  \mc{F}^M \defeq \Big\{ \f : \R^n \times \statdomain \to \R
  \mid \mbox{for~all~} \optvar \in \R^n
  ~ \mbox{and} ~
  g \in \partial_\optvar \f(\optvar, \statval), ~
  \max_{j \le n} \frac{|g_j|}{M_j} \le 1 \Big\}.
\end{equation*}
Then \citet{DuchiJoMc13} provide the following lower bound.
\begin{proposition}[\cite{DuchiJoMc13}, Proposition 1]
  \label{prop:lb}
  Let $M \in \R_+^n$ and $\mc{F}^M$ be as above.
  Let $a \in \R_+^n$ and assume the hyperrectangular
  containment $\prod_{j=1}^n [-a_j, a_j] \subset \optdomain$. Then
  \begin{equation*}
    \minimaxs(\optdomain, \mc{F}^M)
    \geq \frac{1}{8\sqrt{k \log 3}}
    \sum_{j=1}^n M_j a_j.
  \end{equation*}
\end{proposition}
\noindent
An immediate extension of this result provides lower bounds on the minimax
stochastic risk for all orthosymmetric norms $\gamma$, by which we mean
$\gamma(\sigma \odot v) = \gamma(v)$ for all sign vectors $\sigma$, and
convex sets $\optdomain$ without any restriction on qudaratic convexity.
(It will be the matching upper bounds that require this.)
\begin{corollary}
  \label{corollary:lb}
  Let $\optdomain \subset \R^n$ be an orthosymmetric convex set
  and $\gamma$ an orthosymmetric norm.
  Then
  \begin{equation*}
    \minimaxs(\optdomain, \gamma)
    \ge \frac{1}{8 \sqrt{k \log 3}} \sup_{\optvar \in \optdomain}
    \gamma^*(\optvar).
  \end{equation*}
\end{corollary}
\begin{proof}
  Define the hyperrectangle $\rect(\optvar) \defeq \prod_{j \le n}
  [-|\optvar_j|, |\optvar_j|]$, so that $\optdomain \supset \rect(\optvar)$
  for all $\optvar \in \optdomain$ by orthosymmetry.
  Additionally,
  recalling the notation~\eqref{eqn:func-family} of $\mc{F}^{\gamma,1}$ and
  $\mc{F}^M$, if $M \in \R^n_+$ satisfies $\gamma(M) \le 1$ then, by
  orthosymmetry of $\gamma$, $\mc{F}^{\gamma,1} \supset \mc{F}^M$.  Thus
  \begin{equation*}
    \minimaxs(\optdomain, \gamma) \ge
    \minimaxs(\rect(\optvar), \gamma) \ge \minimaxs(\rect(\optvar), \mc{F}^M)
    \ge \frac{1}{8\sqrt{k \log 3}} \sum_{j \le n} |\optvar_j|
    M_j
  \end{equation*}
  for all $M \in \ball{\gamma}{0}{1} \cap \R_+^n$ and $\optvar \in
  \optdomain$.
  Taking a supremum over $M \in \ball{\gamma}{0}{1}$ and
  $\optvar \in \optdomain$, we have
  \begin{equation*}
    \minimaxs(\optdomain,
    \gamma) \ge \frac{1}{8\sqrt{k \log 3}} \sup_{\optvar \in \optdomain}
    \sup_{\gamma(M) \le 1} \optvar^\top M = \frac{1}{8\sqrt{k \log 3}}
    \sup_{\optvar \in \optdomain} \gamma^*(\optvar)
  \end{equation*}
  as desired.
\end{proof}



\subsubsection{Orthosymmetric and quadratically convex gradient norms}

We now provide lower bounds on minimax risk
complementary to Corollary~\ref{cor:ub}, focusing first on the case
that the gradient norm $\gamma$ is quadratically convex.


\begin{assumption}\label{ass:gamma-qc}
  The norm $\gamma$ is orthosymmetric and quadratically convex,
  meaning $\gamma(\sigma \odot v) = \gamma(v)$ for all $\sigma \in \{\pm 1\}^n$
  and $\ball{\gamma}{0}{1}$ is
  quadratically convex.
\end{assumption}

With this, we have the following theorem, which shows that diagonally-scaled
gradient methods are minimax rate optimal, and that the constants are sharp
up to a factor of $9$, whenever the gradient norms are quadratically convex.
While the constant $9$ is looser than that \citet{DonohoLiMa90} provide for
Gaussian sequence models, this theorem highlights the essential structural
similarity between the sequence model case and stochastic optimization
methods.
\begin{theorem}
  \label{th:qc-qc}
  Let Assumption~\ref{ass:gamma-qc} hold and let $\optdomain$ be
  quadratically convex, orthosymmetric, and compact.  Then
  \begin{equation*}
    \frac{1}{8 \sqrt{\log 3}}
    \frac{1}{\sqrt{k}}
    \sup_{\optvar \in \optdomain} \gamma^*(\optvar) \le \minimaxs(\optdomain, \gamma)
    \le \minimaxr(\optdomain, \gamma) \le \frac{1}{\sqrt{k}}
    \sup_{\optvar \in \optdomain} \gamma^*(\optvar).
  \end{equation*}
  There exists $\lambda^\ast \in \R^n_+$ such that \dscaled with
  $\lambda^\ast$ achieve this rate.
\end{theorem}
\begin{proof}
  For the upper bound, we use Corollary~\ref{cor:ub}. Because
  $\ball{\gamma}{0}{1}$ is quadratically convex, we have
  $\qhull(\ball{\gamma}{0}{1}) = \ball{\gamma}{0}{1}$, so that $\sup_{g \in
    \qhull(\ball{\gamma}{0}{1})} \optvar^\top g = \gamma^*(\optvar)$.
  Then note that $\qhull(\optdomain) = \optdomain$ as well.
  The lower bound follows immediately from Corollary~\ref{corollary:lb}.
\end{proof}


\newcommand{\rescale}{\textup{res}} 

\subsubsection{Arbitrary gradient norms}

When the norm $\gamma$ on the gradients defines a non-quadratically convex norm
ball $\ball{\gamma}{0}{1}$---for example, when the gradients belong to an
$\ell_r$-norm ball for $r\in[1, 2]$---our results become less
general. Nonetheless,
when $\gamma$ is a weighted $\ell_r$-norm ball (for $r \in [1, 2]$),
\dscaled remain minimax rate optimal, as Corollary~\ref{cor:qc-wlp} will show.
When the norms $\gamma$ are arbitrary
we use the rescaled
vector $\rescale(\optvar, \gamma)
\defeq (\optvar_j / \gamma(e_j))_{j = 1}^n$, where $e_j$ are the
standard basis vectors:


\begin{theorem}
  \label{th:qc-not-qc}
  Let $\optdomain$ be an orthosymmetric, quadratically convex, convex and
  compact set and $\gamma$ an arbitrary norm. For any $d \in \N$,
  \begin{align*}
      \frac{1}{8 \sqrt{k \log 3}}
      \left(1 - \frac{d}{k \log 3} \right)
      \sup_{\optvar\in\optdomain,
        \norm{\optvar}_0 \le d}
      \ltwo{\rescale(\optvar, \gamma)}
      & \le \minimaxs(\optdomain, \gamma) \\
      & \le \minimaxr(\optdomain, \gamma) \le
      \frac{1}{\sqrt{k}}\sup_{\optvar\in\optdomain} \sup_{g\in
        \qhull(\ball{\gamma}{0}{1})} \dotp{\optvar}{g}.
  \end{align*}
\end{theorem}
\noindent
Corollary~\ref{cor:ub} gives the upper bound in the theorem.
The lower
bound consists of an application of Assouad's method
(Lemma~\ref{lem:assouad}), but, in parallel to the warm-up examples
(Propositions \ref{prop:lp-ball-q12} and \ref{prop:lp-ball-q2infty}), we
construct well-separated functions with ``sparse'' gradients. See
Appendix~\ref{sec:proof-qc-not-qc} for a proof.

We can develop a corollary of this result when the norm $\gamma$ is a
weighted-$\ell_r$ norm for $r\in[1, 2]$. While these do not induce
quadratically convex norm balls,
the previous theorem still guarantees that \dscaled are
minimax rate optimal.

\begin{corollary}
  \label{cor:qc-wlp}
  Let the conditions of Theorem~\ref{th:qc-not-qc} hold and assume that
  $\gamma(g) = \norm{\beta \odot g}_r$ with $r\in[1, 2]$, $\beta_j > 0$ and
  $(\beta \odot g)_j = \beta_j g_j$. Then for $k \ge 2n$,
  \begin{equation*}
    \frac{1}{16} \frac{1}{\sqrt{k}}
    \sup_{\optvar\in\optdomain}
    \ltwo{\rescale(\optvar, \gamma)}
    \le \minimaxs(\optdomain, \gamma)
    \le \minimaxr(\optdomain, \gamma) \le
    \frac{1}{\sqrt{k}}\sup_{\optvar\in\optdomain}
    \ltwo{\rescale(\optvar, \gamma)}.
  \end{equation*}
  There exists $\lambda^\ast \in \R^n_+$ such that \dscaled with
  $\lambda^\ast$ achieve this rate.
\end{corollary}
\noindent
A minor modification of Theorem~\ref{th:qc-not-qc} gives the lower
bound, while we obtain
the upper bound by noting that the quadratic hull of a weighted-$\ell_r$ norm
ball for $r\in[1, 2]$ is the weighted-$\ell_2$ norm ball. The dual norm of
$\gamma(g) = \norm{\beta \odot g}_2$ being $\gamma^\ast(g) = \norm{g /
  \beta}_2$, the upper bound holds by duality. See
Appendix~\ref{sec:proof-cor-qc-not-qc} for the (short) proof.

Theorem~\ref{th:qc-qc} and Corollary~\ref{cor:qc-wlp} show that for a large
collection of norms $\gamma$ on the gradients, \dscaled areminimax rate
optimal.
Arguing that \dscaled are minimax rate optimal when $\gamma$ is
neither a weighted-$\ell_r$ norm nor induces a quadratically convex unit
ball remains an open question, though weighted-$\ell_r$ norms for $r \in [1,
  \infty]$ cover the majority of practical applications of stochastic
gradient methods.

\section{Beyond quadratic convexity: the necessity of non-linear methods}
\label{sec:not-qc}


For $\optdomain \subset \R^n$ quadratically convex, the results in
Section~\ref{sec:qc} show that methods of linear type achieve optimal rates
of convergence.  When the constraint set is not quadratically convex, it is
unclear whether methods of linear type are sufficient to achieve optimal
rates.
As we now show, they are not:
we exhibit collections of problem instances in which the constraint
sets are orthosymmetric convex bodies
but not quadratically convex, and where methods of linear type must have
regret at least a factor $\sqrt{n / \log n}$ worse than the minimax optimal
rate, which (non-linear) mirror descent with appropriate distance generating
function achieves.
We also develop more general results to highlight the way in which
the quadratic hull of the underlying constraint set $\optdomain$
necessarily characterizes the regret of \eucgrad, which allows
for a more explicit delineation of those sets $\optdomain$
for which nonlinear methods are necessary:
when $\sup_{\optvar \in \optdomain} \lone{\optvar}$ is much smaller than
$\sup_{\optvar \in \qhull(\optdomain)} \lone{\optvar}$.

To construct these problem instances, we first turn to simple
non-quadratically convex constraint sets: $\ell_p$ balls for $p\in[1,
  2]$. We measure subgradient norms in the dual $\ell_{p^*}$ norm, $p^* =
\frac{p}{p - 1}$.
Our analysis consists of two steps: we first prove sharp
minimax rates on these problem instances and show that mirror descent with
the right (non-linear) distance generating function is minimax rate
optimal. These results extend those of~\citet{AgarwalBaRaWa12}, who provide
matching lower and upper bounds for $p \geq 1 + c$ for a fixed numerical
constant $c > 0$. In contrast, we prove sharp minimax rates for all $p \geq
1$. To precisely characterize the gap between linear and non-linear methods,
we show that for any linear pre-conditioner, we can exhibit functions for
which the regret of \eucgrad is nearly the simple upper regret bound of
standard gradient methods, Eq.~\eqref{eqn:md-regret} with $h(\optvar) =
\half \ltwo{\optvar}^2$. Thus, when $p$ is very close to $2$ (nearly
quadratically convex), the gap remains within a constant factor, whereas
when $p$ is close to $1$, the gap can be as large as $\sqrt{n / \log n}$.

\subsection{Minimax rates for $p$-norm constraint sets and general
  convex bodies}

For $p\in[1,2]$, we consider the constraint set $\optdomain = \ball{p}{0}{1}
\subset \R^n$ and bound gradients with norm $\gamma =
\|\cdot\|_{p^\ast}$. We begin by proving sharp minimax rates on this
collection of problems and show that, in these cases, non-linear mirror
descent is minimax optimal.
\begin{theorem}
  \label{th:sharp-md}
  Let $p \in [1, 2]$, $\optdomain = \ball{p}{0}{1} \subset \R^n$ and $\gamma
  = \|\cdot\|_{p^\ast}$.
  \begin{enumerate}[(i)]
  \item If $1 \leq p \leq 1 + 1 / \log(2n)$, then
    \begin{equation*}
      1 \wedge \sqrt{\frac{\log(2n)}{k}}
      \lesssim \minimaxs(\optdomain, \gamma)
      \le \minimaxr(\optdomain, \gamma)
      \lesssim
      1 \wedge \sqrt{\frac{\log(2n)}{k}}.
    \end{equation*}
  \item If $1 + 1/\log(2n) < p \le 2$, then
    \begin{equation*}
      1 \wedge \sqrt{\frac{1}{k (p-1)}}
      \lesssim \minimaxs(\optdomain, \gamma)
      \le
      \minimaxr(\optdomain, \gamma) \lesssim
      1 \wedge \sqrt{\frac{1}{k
          (p-1)}}
    \end{equation*}
  \end{enumerate}
  In both cases,
  mirror descent~\eqref{eqn:smd} with distance generating function
  $h(\optvar) \defeq \frac{1}{2(a-1)}\norm{\optvar}_a^2$ for $a =
  \max\{1 + \frac{1}{\log(2n)}, p\}$ achieves the optimal rate.
\end{theorem}
\noindent
The upper bound essentially follows from
Proposition~\ref{prop:rate-md}, and the lower bound uses
reductions from estimation to testing.
See Appendix~\ref{prf:th-sharp-md} for a proof.

\newcommand{\effdim}{\textup{effdim}}

We also provide a few results extending to
potentially infinite-dimensional spaces.
To reduce complexity, we focus on the case that the gradients are
bounded in $\ell_\infty$-norm---the ``most'' quadratically convex set---to
make dependence on $\optdomain$ the clearest. In the sequel it will
sometimes be useful to consider sequence space $\R^\N$, so we give a result
that allows the infinite dimensional containment $\optdomain \subset \R^\N$;
in this case we consider domains $\optdomain$ that are appropriately compact
for the $\ell_1$-norm.
A sufficient condition for the results is that
$\optdomain$ have
finite \emph{effective dimension},
which we define by
\begin{equation}
  \label{eqn:effective-dimension}
  \effdim(\optdomain)
  \defeq \inf_{\beta \ge 0} \left\{e^{1 / \beta}
  \mid \sup_{\optvar \in \optdomain}
  \sum_{j = 1}^\infty j^\beta |\optvar_j|
  \le e \sup_{\optvar \in \optdomain} \lone{\optvar}
  \right\},
\end{equation}
where $\effdim(\optdomain) = +\infty$
if only $\beta = 0$ satisfies the inequality.
It is immediate that if $\optdomain \subset \R^n$, then
$\effdim(\optdomain) \le n$
because $n^{1/\log n} = e$.
The next observation summarizes a few sufficient conditions for
the effective dimension~\eqref{eqn:effective-dimension}
to be finite.
We include the proof in
Appendix~\ref{sec:proof-self-similar} for completeness.


\begin{observation}
  \label{observation:self-similar}
  The following conditions are sufficient to guarantee $\effdim(\optdomain)
  < \infty$.
  \begin{enumerate}[i.]
  \item If $\optdomain \subset \R^n$, then $\effdim(\optdomain) \le n$.
  \item If $\lim_{\beta \downarrow 0} \sup_{\optvar \in \optdomain}
    \sum_{j \ge 1} j^\beta |\optvar_j| < \infty$,
    then $\effdim(\optdomain) < \infty$.
  \item
    Let $N < \infty$ and
    $\gamma > 0$ satisfy the tail condition that
    $\sup_{\optvar \in \optdomain} \sum_{j = 1}^N |\optvar_j|
    \ge n^{2\gamma} \sup_{\optvar \in \optdomain}
    \sum_{j > n} |\optvar_j|$ for all $n > N$. Then
    $\effdim(\optdomain) \le O(1) \cdot (N^2 + \exp(\frac{3}{\gamma}
    \log \frac{1}{\gamma}))$.
  \end{enumerate}
\end{observation}


The next proposition then shows that the $\ell_1$-diameter of $\optdomain$
always provides lower bounds on the (stochastic) minimax risk,
while $p$-norm-based mirror
descent algorithms achieve regret at most logarithmic in
the effective-dimension~\eqref{eqn:effective-dimension},
making the $\ell_1$-diameter of $\optdomain$ the central quantity
governing minimax risk.
\begin{proposition}
  \label{proposition:risks-with-l1-norms}
  Let $\optdomain$ be an orthosymmetric convex
  set.
  Then
  \begin{equation*}
    \frac{1}{8 \sqrt{\log 3}} \frac{1}{\sqrt{k}}
    \sup_{\optvar \in \optdomain} \lone{\optvar}
    \le \minimaxs(\optdomain, \linf{\cdot}).
  \end{equation*}
  If additionally $\optdomain$
  has finite effective dimension
  $N = \effdim(\optdomain)$ as in~\eqref{eqn:effective-dimension},
  then
  \begin{equation*}
    \minimaxr(\optdomain, \linf{\cdot})
    \le O(1) \cdot \sqrt{\log N} \frac{1}{\sqrt{k}}
    \sup_{\optvar \in \optdomain}
    \lone{\optvar}.
  \end{equation*}
  Letting $\beta = \frac{1}{\log N}$ and
  defining the operator $A$ by $A \optvar = (j^\beta \optvar_j)_{j \ge 1}$,
  the mirror descent method with distance generating
  function $h(\optvar) = \frac{1}{2(p - 1)} \norm{A \optvar}_p^2$
  for $p = \frac{2}{2 - \beta}$ achieves this regret.
\end{proposition}

To preview our discussion to come, note the similarities between
Proposition~\ref{proposition:risks-with-l1-norms} and
Corollary~\ref{corollary:soft-thresh-minimax} in the Gaussian sequence model
case: so long as $\optdomain$ is appropriately regular, there is a
``universal'' nonlinear method---in the case of Gaussian sequence models,
soft-thresholding with truncation~\eqref{eqn:truncated-soft-est}, in the
case of stochastic and online optimization, mirror descent with a
$p$-norm-based distance-generating function---that achieves nearly
rate-optimal minimax risk.
The analogy extends spiritually as well, as in both Gaussian sequence models
and stochastic optimization~\cite{Johnstone17, DonohoJo95, DonohoLiMa90,
  NemirovskiYu83, BeckTe03}, the universal methods were originally designed
for high-dimensional signals $\optvar$.

\begin{proof}
  The lower bound is a direct consequence of Corollary~\ref{corollary:lb}.
  To demonstrate the upper bound, we take $\optdomain \subset \R^\N$, as the
  special case of finite dimensions follows immediately.
  For notational
  simplicity, for a positive sequence $(a_j)_{j \in \N}$ to be chosen we let
  $A : \R^\N \to \R^\N$ be the diagonal linear operator $Ax = (a_j x_j)_{j
    \in \N}$.
  For $p \in \openleft{1}{2}$ to be chosen as well, define
  the distance generating function
  \begin{equation*}
    h(x) = \frac{1}{2(p - 1)} \norm{A x}_p^2.
  \end{equation*}
  By an analogous argument to the finite-dimensional
  case, this is strongly convex with respect to
  the norm $\norm{x} = \norm{Ax}_p$ and has dual norm
  $\dnorm{g} = \norm{A^{-1} g}_q$ for $q = \frac{p}{p-1}$.
  For any sequence of subgradients $g_i \in [-1, 1]^\N$,
  we thus obtain regret bound
  \begin{equation*}
    \regret_k(x) \le \frac{1}{2(p - 1) \stepsize} \norm{A x}_p^2
    + \frac{\stepsize}{2} \sum_{i = 1}^k \norm{A^{-1} g_i}_q^2.
  \end{equation*}
  Notably, for $g \in [-1, 1]^\N$ we have
  \begin{equation*}
    \norm{A^{-1} g}_q^q
    \le \sum_{j = 1}^\infty a_j^{-q}.
  \end{equation*}
  We use the assumption~\eqref{eqn:effective-dimension} that $\optdomain$
  has finite effective dimension $N$, so that for $\beta = \frac{1}{\log N}
  \le 1$, $\sup_{\optvar \in \optdomain} \sum_{j = 1}^\infty j^\beta
  |\optvar_j| \le e \sup_{\optvar \in \optdomain} \lone{\optvar}$.  Letting
  $a_j = j^\beta$, we take $q = \frac{2}{\beta}$, whence $\norms{A^{-1}
    g}_q^q = \sum_{j = 1}^\infty j^{-2} = \frac{\pi^2}{6}$.
  As this also
  gives conjugate $p = \frac{q}{q - 1} = 1 + \frac{\beta}{2 - \beta}$,
  we obtain the regret bound
  \begin{equation*}
    \regret_k(x) \le \frac{1}{\stepsize} \frac{2 - \beta}{2 \beta}
    \norm{A x}_p^2
    + O(1) \cdot \stepsize k
    \le O(1) \left[\frac{1}{\stepsize \beta} \sup_{\optvar \in \optdomain}
      \lone{\optvar}^2 + \stepsize k \right].
  \end{equation*}
  Take $\stepsize = \sup_{\optvar \in \optdomain}
  \lone{\optvar} / \sqrt{k \beta}$ to obtain the convergence
  upper bound.
\end{proof}

We make one final remark on these results, and include some
commentary in the final discussion on future work.
Corollary~\ref{corollary:lb} shows that
\begin{equation*}
  \frac{1}{8\sqrt{k \log 3}}
  \sup_{\optvar \in \optdomain} \sup_{g : \gamma(g) \le 1}
  \dotp{\optvar}{g}
  \le \minimaxs(\optdomain, \gamma)
\end{equation*}
for arbitrary orthosymmetric convex $\optdomain$ and norms $\gamma$.
This is sharp whenever $\gamma$ is an $\ell_q$ norm for $2 \le q \le
\infty$, as then mirror descent with distance generating function
$h(\optvar) = \frac{1}{2(p - 1)} \norm{\optvar}_p^2$, $p = \max\{\frac{q}{q
  - 1}, 1 + \frac{1}{\log n}\}$ achieves the minimax lower bound, with
average regret scaling at worst as $\sqrt{\log n} \sup_{\optvar \in
  \optdomain} \norm{\optvar}_q / \sqrt{k}$.
It is also \emph{not} a tight lower bound in general: when
$\optdomain = [-1, 1]^n$ and $\gamma(g) = \lone{g}$, then
Proposition~\ref{prop:lp-ball-q12} shows that
$\sqrt{n/k}$ is the correct scaling---a factor $\sqrt{n}$
larger than $\sup_{\optvar, g} \dotp{\optvar}{g}$.


\subsection{Hard problems for \eucgrad and quadratic hulls}
\label{sec:hard-euclidean-problems}

Theorem~\ref{th:sharp-md} shows that (non-linear) mirror descent methods are
minimax rate-optimal for $\ell_p$-ball constraint sets, $p \in [1, 2]$, with
gradients contained in the corresponding dual $\ell_{q}$-norm ball ($q =
\frac{p}{p-1}$). For such problems and $p$, standard subgradient methods achieve
worst-case regret $O(n^{1/2 - 1/q} / \sqrt{k})$.
This upper bound is in fact sharp: in the next theorem, we show that for any
method of linear type, we can construct a sequence of (linear) functions
such that the method's regret achieves the worst-case upper bound for
standard subgradient methods, precisely quantifying the gap between linear
and non-linear methods for this problem class.  To that end, let
\begin{equation*}
  \regret_{k,A}(\optvar) \defeq \sum_{i = 1}^k g_i^\top(\optvar_i - \optvar)
\end{equation*}
denote the regret of the Euclidean online mirror descent method with
distance generating function $h_A(\optvar) = \half \dotp{\optvar}{A
  \optvar}$ for functions $\f_i$ with subgradients $g_i$.
In the lower bounds to come,
we take $\f_i(\optvar) = \dotp{g_i}{\optvar}$ to be linear,
so that $\nabla \f_i(\optvar) = g_i$ is independent of $\optvar$.

\begin{theorem}
  \label{th:regret-ub-lp}
  For any $A \succeq 0$ and $p \in [1, 2]$ with $q = \frac{p}{p-1}$, there
  exists a sequence of vectors $g_i \in \R^n$, $\norm{g_i}_q \le 1$, and
  point $\optvar \in \R^n$ with $\norm{\optvar}_p \le 1$ such that
  \begin{equation*}
    \regret_{k, A}(\optvar) \ge
    \half \min\left\{k / 2, \sqrt{2k} \cdot n^{1/2 - 1/q} \right\}.
  \end{equation*}
  Scaled identity matrices $A = c \cdot I_n$ achieve
  these bounds to within a factor of $\sqrt{2}$ for
  $k \ge 2 n^{1 - 2/q}$.
\end{theorem}
\begin{proof}
  Let $A \succ 0$ be a positive semi-definite matrix for the distance
  generating function $h_A(\optvar) = \half \optvar^\top A \optvar$ defined above,
  and let $q = \frac{p}{p-1}$ be the conjugate to $p$.  We choose linear functions
  $\f_i(\optvar) := \dotp{g_i}{\optvar}$ where $g_i \in \ball{q}{0}{1}$. In this
  case, letting $\{\optvar_i\}_{i\leq k}$ be the points mirror descent plays, the
  regret with respect to $\optvar \in\R^n$ is
  \begin{equation*}
    \regret_{k, A}(\optvar) =
    \sum_{i\leq k} \f_i(\optvar_i) - \f_i(\optvar) = \sum_{i\leq k}
    \dotp{g_i}{(\optvar_i - \optvar)},
  \end{equation*}
  so that
  \begin{equation*}
    \regret^\ast_{k, A}
    \defeq
    \sup_{\norm{\optvar}_p \le 1}
    \regret_{k,A}(\optvar)
    = \normbigg{\sum_{i \le k} g_i}_q
    + \half \sum_{i \le k} \norm{g_i}_{A^{-1}}^2
    - \half \normbigg{\sum_{i \le k} g_i}_{A^{-1}}^2.
  \end{equation*}
  Now, we choose linear functions $f_i$ so that the regret is large.
  To do so, choose vectors
  \begin{equation}
    \label{eqn:choose-u-v}
    u \in \argmax_{\norm{\statval}_q \le 1} \dotp{\statval}{A^{-1} \statval}
    ~~ \mbox{and} ~~
    v \in \argmin_{\norm{\statval}_q = 1} \dotp{\statval}{A^{-1} \statval}.
  \end{equation}
  Then set the (gradient) vectors $g_i \in \R^n$ so
  that for a $\delta \in [0, 1]$ to be chosen,
  \begin{enumerate}[(a)]
  \item $g_i = u$ for $k/4$ of the indices $i \in [k]$
  \item $g_i = -u$ for $k/4$ of the indices $i \in [k]$
  \item $g_i = v$ for $\frac{k}{4}(1+\delta)$ of the indices $i \in [k]$
  \item $g_i = -v$ for $\frac{k}{4}(1-\delta)$ of the indices $i \in [k]$.
  \end{enumerate}
  With these choices, we obtain the regret lower bound
  \begin{align}
    \regret^\ast_{k, A}
    & \geq
    \sup_{\delta \le 1}
    \left[
      \frac{k}{2} \delta \norm{v}_q
      + \frac{k}{4} \dotp{u}{\invert{A}u} -
      \frac{\delta^2 k^2}{8}\dotp{v}{\invert{A}v}
      \right] \nonumber \\
    & \ge \frac{k}{4} \cdot \left[\dotp{u}{A^{-1} u} + \min\left\{
      1, \frac{2\norm{v}_q}{k \dotp{v}{A^{-1} v}}\right\}
      \norm{v}_q\right].
    \label{eqn:regret-with-A-q}
  \end{align}
  
  \newcommand{\opnormpq}[1]{\norm{#1}_{\ell_p \to \ell_q}}
  \newcommand{\opnormqtwo}[1]{\norm{#1}_{\ell_q \to \ell_2}}
  \newcommand{\opnormqtwos}[1]{\norms{#1}_{\ell_q \to \ell_2}}
  \newcommand{\opnormtwoq}[1]{\norm{#1}_{\ell_2 \to \ell_q}}
  \newcommand{\opnormtwoqs}[1]{\norms{#1}_{\ell_2 \to \ell_q}}
  
  We now consider two cases. In the first, $A$ is large enough that
  $\norm{v}_q \ge \half k \dotp{v}{A^{-1} v}$. Then the regret
  bound~\eqref{eqn:regret-with-A-q} becomes
  \begin{equation*}
    \regret^*_{k,A} \ge \frac{k}{4} \left[
      \dotp{u}{A^{-1} u} + \norm{v}_q\right]
    \ge \frac{k}{4},
  \end{equation*}
  as $\norm{v}_q = 1$ by the construction~\eqref{eqn:choose-u-v}. This gives the
  first result of the theorem. For the second claim, which holds in the case that
  $\norm{v}_q < \half k \dotp{v}{A^{-1} v}$, we consider the operator norms of
  general invertible linear operators. For a mapping $T : \R^n \to \R^n$, define
  the $\ell_p$ to $\ell_q$ operator norm
  \begin{equation*}
    \opnormpq{T} \defeq \sup_{x \neq 0} \frac{\norm{T(x)}_q}{\norm{x}_p}.
  \end{equation*}
  Then the construction~\eqref{eqn:choose-u-v} evidently yields
  \begin{equation*}
    \dotp{u}{A^{-1}u}
    = \opnormqtwos{A^{-1/2}}^2
    ~~~ \mbox{and} ~~~
    \frac{\norm{v}_q^2}{\dotp{v}{A^{-1} v}}
    = \sup_{x \neq 0} \frac{\norms{A^{1/2} x}_q^2}{
      \norm{x}_2^2}
    = \opnormtwoqs{A^{1/2}}^2.
  \end{equation*}
  Revisiting the regret~\eqref{eqn:regret-with-A-q}, we obtain
  \begin{equation*}
    \regret^*_{k,A} \ge
    \frac{k}{4} \cdot \left[
      \opnormqtwo{A^{-1/2}}^2
      + \frac{2}{k} \opnormtwoq{A^{1/2}}^2\right]
    \ge \sqrt{\frac{k}{2}} \opnormqtwos{A^{-1/2}} \opnormtwoqs{A^{1/2}},
  \end{equation*}
  where we have used that $ab \le \half a^2 + \half b^2$ for all $a, b$.
  But for any invertible linear operator, standard results
  on the Banach-Mazur distance~\cite[Corollary~2.3.2]{Vershynin09}
  imply that
  \begin{equation*}
    \inf_{A \succ 0}
    \opnormtwoq{A} \opnormqtwo{A^{-1}} \ge
    n^{1/2 - 1/q}.
  \end{equation*}
  This gives the lower bound.
  
  For the claimed upper bound, note
  for $h(\optvar) = \frac{1}{2 \stepsize} \ltwo{\optvar}^2$
  (i.e.\ $A = \frac{1}{\stepsize} I_n$) and
  initial point $\optvar_0 = \zeros$, we have
  $\regret_k(\optvar) \le \frac{1}{2 \stepsize} \ltwo{\optvar}^2
  + \frac{\stepsize}{2} \sum_{i = 1}^k \ltwo{g_i}^2$.
  As $\ltwo{x} \le 1$ whenever $\norm{\optvar}_p \le 1$ and
  and $\ltwo{g} \le n^{1/2 - 1/q}$ whenever $\norm{g}_q \le 1$,
  we have
  $\regret_k(\optvar) \le \frac{1}{2 \stepsize}
  + \frac{\stepsize}{2} k n^{1 - 2/q}$.
  Choose $\stepsize = (k n^{1 - 2/q})^{-1/2}$ to minimize this bound
  and achieve $\sup_{\norm{\optvar}_p \le 1} \regret_k(\optvar)
  \le \sqrt{k} n^{1/2 - 1/q}$.
\end{proof}
These results
explicitly exhibit a gap between methods of linear type and non-linear
mirror descent methods for this problem class. In contrast to the frequent
practice in literature of simply comparing regret upper bounds---prima facie
illogical---we demonstrate the gap indeed must hold.

In combination with Theorem~\ref{th:regret-ub-lp},
Proposition~\ref{prop:rate-md} precisely separates
linear and non-linear mirror descent on these problems for all values of
$p\in[1, 2]$.
Indeed, when $p=1$, for any pre-conditioner $A$, there exists
a problem on which \eucgrad have regret at least $\Omega(1)\sqrt{n / k}$.
On the same problem, non-linear mirror descent has regret at most
$O(1)\sqrt{\log n / k}$, showing the advertised $\sqrt{n / \log n}$
gap.
When $p \ge 2 - 1 / \log n$ (so $\optdomain$ is nearly quadratically
convex), the gap reduces to at most a constant factor.

The quadratically convex hulls of both the domain $\optdomain$ and the space
$G \defeq \{\partial \f(x; \statval)\}_{\optvar \in \optdomain, \statval \in
  \statdomain}$ of potential subgradients turn out to completely
govern the behavior of linear-type methods.
In this sense, Corollary~\ref{cor:ub} is sharp, even to the
leading constant factor: the cost (in a minimax sense) of using
linear methods is that regret necessarily scales with
$\qhull(\optdomain)$ and $\qhull(G)$.


\begin{theorem}
  \label{theorem:regret-with-quadratic-hull}
  Let $\optdomain \subset \R^n$ and $\gradomain \subset \R^n$ be
  orthosymmetric convex bodies.
  For any sequence $A(k) \succeq 0$, there exist
  sequences of vectors $g_i = g_i(k) \in \gradomain$, $i = 1, \ldots, k$,
  such that
  \begin{equation*}
    \liminf_k \frac{1}{\sqrt{k}}
    \sup_{\optvar \in \optdomain} \regret_{k,A(k)}(\optvar) \ge
    \sup_{\optvar
      \in \qhull(\optdomain)}
    \sup_{g \in \qhull(\gradomain)} \optvar^\top g.
  \end{equation*}
  Additionally, for each $k$ there exists a diagonal matrix $D$ such that
  for any sequence of convex functions $\f_i : \optdomain \to \R$
  for which $\partial \f_i \subset \gradomain$,
  \begin{equation*}
    \frac{1}{\sqrt{k}} \sup_{\optvar \in \optdomain} \regret_{k, D}(\optvar)
    \le \sup_{\optvar \in \qhull(\optdomain)}
    \sup_{g \in \qhull(\gradomain)} \optvar^\top g.
  \end{equation*}
\end{theorem}
\noindent
See Appendix~\ref{sec:proof-regret-with-quadratic-hull} for the proof
of the result, which parallels that of Theorem~\ref{th:regret-ub-lp}.

To make a neater analogy with Gaussian sequence models
$y = \optvar + \noise$ for $\noise \sim \normal(0, \sigma^2 I_n)$, where
the scale of noise $\noise_j$ on different coordinates is
identical,
we specialize to the case that $\gamma(g) = \linf{g}$,
that is,
$G = \ball{\infty}{0}{1}$.
Recalling the notation~\eqref{eqn:func-family} identifying
$\mc{F}^{\linf{\cdot}, 1}$ as those functions whose subgradients $g \in
\partial \f(\optvar, \statval)$ belong to the $\ell_\infty$-ball, the
next corollary then follows immediately.

\begin{corollary}
  \label{cor:regret-with-quadratic-hull-l1}
  Let $\optdomain \subset \R^n$ be an orthosymmetric convex body.
  For any sequence $A(k) \succeq 0$, there exist
  sequences of vectors $g_i = g_i(k) \in \R^n$, $i = 1, \ldots, k$,
  with $\linf{g_i} \le 1$, such that
  \begin{equation*}
    \liminf_k \frac{1}{\sqrt{k}}
    \sup_{\optvar \in \optdomain} \regret_{k,A(k)}(\optvar) \ge \sup_{\optvar
      \in \qhull(\optdomain)} \lone{\optvar}.
  \end{equation*}
  Additionally, for each $k$ there exists a diagonal matrix $D$ such
  that for any sequence of convex functions $\f_i \in \mc{F}^{\linf{\cdot},1}$,
  \begin{equation*}
    \sup_{\optvar \in \optdomain} \regret_{k, D}(\optvar)
    \le \sqrt{k} \sup_{\optvar \in \qhull(\optdomain)} \lone{\optvar}.
  \end{equation*}
\end{corollary}

\noindent

In brief, the regret of \eucgrad \emph{necessarily} scales with the size of
the quadratic hull $\qhull(\optdomain)$.  Contrasting this result with
Proposition~\ref{proposition:risks-with-l1-norms}, we see with nonlinear
methods, the regret need scale only as $\sup_{\optvar \in \optdomain}
\lone{\optvar}$ rather than as $\sup_{\optvar \in \qhull(\optdomain)}
\lone{\optvar}$, so that the gap between convergence
achievable by linear and nonlinear
methods is large precisely when
\begin{equation*}
  \frac{\sup_{\optvar \in \qhull(\optdomain)} \lone{\optvar}}{
    \sup_{\optvar \in \optdomain} \lone{\optvar}}
  \gg 1.
\end{equation*}


\section{The role of quadratic convexity in sequence models and
  first-order methods}
\label{sec:role-quadratic-convexity}

The results in Section~\ref{sec:gsm} highlight and recapitulate some of the
ways that quadratic convexity distinguishes linear and nonlinear methods in
Gaussian sequence models.  Theorems~\ref{th:regret-ub-lp}
and~\ref{theorem:regret-with-quadratic-hull}, along with the complimentary
results in Proposition~\ref{proposition:risks-with-l1-norms},
address these differences for stochastic optimization problems.
So in both sequence models and convex optimization,
geometric aspects of the underlying set $\optdomain$ determine
nonlinear methods' necessity.
The analogy between sequence models and stochastic optimization methods is
not perfect, however, as there are sets $\optdomain$ for which linear
methods are minimax rate optimal for stochastic optimization problems and
not for sequence models and vice versa.
In both problem families, a particular ``distance'' of a set
$\optdomain$ from quadratic convexity delineates determines
when nonlinear methods are necessary; we show
that these can be different.

We begin by translating the results in Section~\ref{sec:gsm} on Gaussian
sequence models into a more geometric form; \citet{DonohoLiMa90} more or
less give this translation but we make a few minor modifications for
convenience in exposition.
The measure of size most natural for Gaussian sequence models
turns out to be (duality-based variants of)
the Kolmogorov $n$-width of the underlying set:
\begin{definition}
  The \emph{$n$-width} of a set $\optdomain$ is
  \begin{equation*}
    \width^2(n) \defeq
    \inf_{\zeros \preceq d \preceq \ones, \<\ones, d\> = n}
    \sup_{\optvar \in \optdomain} \sum_j (1 - d_j) \optvar_j^2.
  \end{equation*}
  The \emph{nonlinear $n$-width} of $X$ is
  \begin{equation*}
    \nlwidth^2(n) \defeq \sup_{\optvar \in \optdomain}
    \inf_{\zeros \preceq d \preceq \ones,
      \<\ones, d\> = n} \sum_j (1 - d_j) \optvar_j^2.
  \end{equation*}
\end{definition}
\noindent
Recalling Corollaries~\ref{corollary:linear-risk-general-gsm}
and~\ref{corollary:soft-thresh-minimax},
the gap between the linear minimax and nonlinear minimax
risk is large for (compact) convex sets $\optdomain$ whenever
the difference between
\begin{equation*}
  \sup_{\optvar \in \optdomain}
  \sum_{j = 1}^\infty \frac{\optvar_j^2 \sigma^2}{\optvar_j^2 + \sigma^2}
  ~~ \mbox{and} ~~
  \sup_{q \in \qhull(\optdomain)}
  \sum_{j = 1}^\infty \frac{q_j^2 \sigma^2}{q_j^2 + \sigma^2}.
\end{equation*}
The next proposition, parts of which are present in \citet[Section
  6]{DonohoLiMa90}, connects the $n$-widths to the linear and nonlinear
minimax risks, where for a vector $\optvar \in \R^\N$, we let
$|\optvar_{(1)}| \ge |\optvar_{(2)}| \ge \cdots$ denote the entries of
$\optvar$ sorted by magnitude.

\begin{proposition}
  \label{proposition:n-width-risk-of-gsm}
  For any compact orthosymmetric convex set $\optdomain$,
  \begin{equation*}
    \width^2(n) = \sup_{q \in \qhull(X)}
    \sum_{j \ge n + 1} q_{(j)}^2
    ~~ \mbox{and} ~~
    \nlwidth^2(n) = \sup_{x \in X}
    \sum_{j \ge n + 1} x_{(j)}^2.
  \end{equation*}
  Additionally, for any $\sigma^2$,
  \begin{equation*} 	
    \sup_{x \in X} \sum_j x_j^2 \wedge \sigma^2
    = \inf_n \left\{\nlwidth^2(n) + n \sigma^2 \right\},
  \end{equation*}
  and
  \begin{equation*}
    \sup_{q \in \qhull(X)}
    \sum_j q_j^2 \wedge \sigma^2
    =
    \inf_n \left\{\width^2(n) + n \sigma^2 \right\}.
  \end{equation*}
\end{proposition}
\proof{}
  For the characterizations, we use duality to see that
  \begin{align*}
    \width^2(n) & =
    \inf_{\zeros \preceq d \preceq \ones, \<\ones, d\> = n} \sup_{v \in \squaredhull(X)}
    \<\ones - d, v\> \\
    & = \sup_{v \in \squaredhull(X)}
    \inf_{\zeros \preceq d \preceq \ones, \<\ones, d\> = n}
    \<\ones - d, v\>
    = \sup_{v \in \squaredhull(X)}
    \sum_{j \ge n + 1} v_{(j)},
  \end{align*}
  while it is immediate that $\nlwidth^2(n) = \sup_{\optvar \in \optdomain}
  \sum_{j \ge n + 1} \optvar_{(j)}^2$.
  Then we recognize that for any sorted nonnegative vector
  $a_1 \ge a_2 \ge \cdots$,
  \begin{equation*}
    \inf_n \bigg\{\sum_{j \ge n + 1} a_j
    + n \sigma^2 \bigg\}
    =
    \inf_n \bigg\{\sum_{j \ge n + 1} a_j
    + \sum_{j = 1}^n \sigma^2 \bigg\}
    = \sum_j a_j \wedge \sigma^2
  \end{equation*}
  by choosing any $n$ such that $a_j \le \sigma^2$ for all
  $j \ge n$, while $a_j \ge \sigma^2$ for $j < n$.
\endproof

Using Proposition~\ref{proposition:n-width-risk-of-gsm}, we see that under
the conditions of Corollary~\ref{corollary:soft-thresh-minimax},
because for any $a, b \ge 0$ we have $\half \min\{a, b\}
\le \frac{ab}{a + b} \le \min\{a, b\}$,
the linear sequence model risk satisfies
\begin{equation*}
  \half \inf_n \left\{\width^2(n) + n \sigma^2\right\}
  \le \lineargsmrisk(\optdomain) \le
  \inf_n \left\{\width^2(n) + n \sigma^2\right\},
\end{equation*}
while by Proposition~\ref{proposition:gsm-lecam-lower}, the (nonlinear) minimax risk satisfies
\begin{equation*}
  \frac{1}{10}
  \inf_n \left\{\nlwidth^2(n) + n \sigma^2\right\}
  \le \maxgsmrisk(\optdomain)
  \lesssim \log \frac{1}{\sigma^2}
  \cdot
  \inf_n \left\{\nlwidth^2(n) + n \sigma^2\right\}.
\end{equation*}
The linear and nonlinear $n$-widths of $\optdomain$
therefore (up to a logarithmic factor in $\frac{1}{\sigma}$)
determine the risk in sequence models, so that when
they are similar the linear and nonlinear minimax risks
coincide. In stochastic optimization, in
contrast,
Section~\ref{sec:not-qc} shows
that convergence guarantees for methods of linear type
coincide with those for arbitrary methods
(again, up to logarithmic factors) if and only if
\begin{equation*}
  \sup_{\optvar \in \qhull(\optdomain)}
  \lone{\optvar} \asymp
  \sup_{\optvar \in \optdomain} \lone{\optvar}.
\end{equation*}

%

The typical scenario in sequence models one considers is the
risk as $\sigma \downarrow 0$, and the following
essentially trivial observation shows that for regular
enough sets $\optdomain$, when the
linear and nonlinear $n$-widths differ, the rates at which
$\lineargsmrisk$ and $\maxgsmrisk$ converge to zero differ.
\begin{observation}
  Let $\alpha > 0$ be such that
  $n^\alpha \nlwidth^2(n) \asymp \width^* > 0$ as $n \to \infty$,
  and assume for some $\beta > 0$ that $\width^2(n)
  \ge n^\beta \nlwidth^2(n)$. Then
  \begin{equation*}
    \frac{\maxgsmrisk(\optdomain)}{\lineargsmrisk(\optdomain)}
    \lesssim \log \frac{1}{\sigma^2}
    \cdot \sigma^\frac{2 \beta}{(1 + \alpha)(1 + \alpha - \beta)}
    \to 0
    ~~ \mbox{as} ~~ \sigma \downarrow 0.
  \end{equation*}
\end{observation}
\proof{}
  The assumption that $\nlwidth^2(n) \lesssim \width^* n^{-\alpha}$
  guarantees the assumptions of Corollary~\ref{corollary:soft-thresh-minimax}
  apply, and we necessarily have $\beta \le \alpha$ (as otherwise
  we would have $\nlwidth^2(n) \to \infty$). Thus
  there exists a (numerical) constant $C < \infty$
  such that for all small enough $\sigma^2 > 0$, 
  \begin{equation*}
    \maxgsmrisk(\optdomain) \le C \log \frac{1}{\sigma^2} \cdot
    \inf_n \left\{ \width^* n^{-\alpha} + n \sigma^2 \right\}
    \asymp \log \frac{1}{\sigma^2}
    \cdot \sigma^{\frac{2 \alpha}{1 + \alpha}},
  \end{equation*}
  which follows by setting $n = \sigma^{-2 / (1 + \alpha)}$. In
  contrast, the linear risk satisfies
  \begin{equation*}
    \lineargsmrisk(\optdomain)
    \gtrsim \inf_n \left\{n^\beta \nlwidth^2(n) + n \sigma^2\right\}
    \asymp \sigma^{\frac{2(\alpha - \beta)}{1 + \alpha - \beta}}
  \end{equation*}
  as $\sigma \downarrow 0$. Then observe that
  $\frac{\alpha}{1 + \alpha}
  - \frac{\delta}{1 + \delta}
  = \frac{\alpha - \delta}{1 + \alpha + \delta
    + \alpha \delta}$,
  and set $\delta = \alpha - \beta$.
\endproof

The role of quadratic convexity of $\optdomain$ differs between
Gaussian sequence models and stochastic optimization problems, however, and
the remainder of this section explores two extended examples highlighting
this.  We focus on sets $\optdomain \subset \R^\N$ in sequence space.

\subsection{A constraint set requiring nonlinearity only in
  Gaussian sequence models}

We  show that for a large family of $\ell_1$-bodies $\optdomain$,
minimax (rate) optimal estimation requires nonlinearity for
Gaussian sequence models but not in stochastic optimization.
Take
\begin{equation*}
  \optdomain \defeq \bigg\{\optvar \in \R^\N \mid
  \sum_{j = 1}^\infty a_j |\optvar_j| \le 1 \bigg\}
  ~~ \mbox{so} ~~
  \qhull(\optdomain) =
  \bigg\{\optvar \in \R^\N \mid
  \sum_{j = 1}^\infty a_j^2 \optvar_j^2 \le 1\bigg\},
\end{equation*}
where $(a_j)$ is a nondecreasing positive sequence where (w.l.o.g.)
we take $a_1 = 1$. Computing the $\ell_1$-diameters of
$\optdomain$ and its quadratic hull, we then have
\begin{equation}
  \label{eqn:l1-diameters-l1-body}
  \sup_{\optvar \in \optdomain} \lone{\optvar} = 1
  ~~ \mbox{and} ~~
  \sup_{\optvar \in \qhull(\optdomain)} \lone{\optvar}
  = \bigg(\sum_{j = 1}^\infty a_j^{-2}\bigg)^{1/2}
\end{equation}
by the Cauchy-Schwarz inequality.
On the other hand, because the coefficients $a_j$ are increasing,
we can compute both the linear and nonlinear widths via
\begin{equation*}
  \width^2(n)
  = \sup \bigg\{
  \sum_{j = n + 1}^\infty
  \optvar_j^2 \mid \sum_{j = 1}^\infty
  a_j^2 \optvar_j^2 \le 1,
  ~ \optvar_1 \ge \optvar_2 \ge \cdots \ge 0\bigg\}
\end{equation*}
and
\begin{equation*}
  \nlwidth^2(n)
  = \sup\bigg\{ \sum_{j = n + 1}^\infty \optvar_j^2
  \mid \sum_{j = 1}^\infty a_j \optvar_j \le 1,
  ~ \optvar_1 \ge \optvar_2 \ge \cdots \ge 0 \bigg\}.
\end{equation*}
By convexity (we maximize a convex function over a convex set in each case),
for each width the maximizing point takes the form
$(t, t, \ldots, t, 0, \ldots)$, with $m$ repeated values $t$.
Then we obtain
\begin{align}
  \label{eqn:width-of-l1-bodies}
  \width^2(n)
  & = \sup_{m \ge n, t \ge 0}
  \bigg\{ (m - n) t^2 \mid t^2 \sum_{j = 1}^m a_j^2 \le 1\bigg\}
  = \sup_{m \ge n} \frac{m - n}{\sum_{j = 1}^m a_j^2} \\
  \nlwidth^2(n)
  & = \sup_{m \ge n, t \ge 0}
  \bigg\{(m - n) t^2 \mid t \sum_{j = 1}^{m} a_j \le 1 \bigg\}
  = \sup_{m \ge n} \frac{m - n}{(\sum_{j = 1}^m a_j)^2}.
  \nonumber
\end{align}

Comparing the $\ell_1$-diameters~\eqref{eqn:l1-diameters-l1-body}
and the widths~\eqref{eqn:width-of-l1-bodies}, we see that whenever
$\sum_j a_j^{-2}$ is summable \eucgrad are minimax (rate)
optimal for stochastic optimization. If
$a_j = j^{\alpha / 2}$ for some $\alpha > 1$, however,
we have
\begin{align*}
  \width^2(n) & = \sup_{m \ge n} \frac{m - n}{
    \sum_{j = 1}^m j^\alpha}
  \asymp \sup_{m \ge n} \frac{m - n}{m^{1 + \alpha}}
  \asymp \frac{1}{n^\alpha}
  ~~ \mbox{while} \\
  \nlwidth^2(n) & =
  \sup_{m \ge n} \frac{m - n}{(\sum_{j = 1}^m j^{\alpha/2})^2}
  \asymp \sup_{m \ge n} \frac{m - n}{m^{2 + \alpha}}
  \asymp \frac{1}{n^{1 + \alpha}}
\end{align*}
because $\sum_{j = 1}^m j^\beta \asymp
\int_0^m t^\beta dt = \frac{1}{\beta + 1} m^{1 + \beta}$
for $\beta > 0$. Summarizing, we have the following corollary.
\begin{corollary}
  \label{corollary:l1-bodies-gsm-nonlinear}
  Let $\optdomain = \{\optvar \mid \sum_{j = 1}^\infty a_j |\optvar_j| \le 1\}$,
  where $a_j = j^{\alpha/2}$ and $\alpha > 1$.
  Then minimax rate optimal estimation
  for Gaussian sequence models requires that the estimator
  $\what{\optvar}$ be nonlinear, while
  \eucgrad are minimax rate optimal for stochastic optimization.
\end{corollary}

\subsection{A constraint set requiring nonlinearity only in
  stochastic optimization}

To contrast with Corollary~\ref{corollary:l1-bodies-gsm-nonlinear},
we can also give families of underlying constraint
sets $\optdomain$ for which only nonlinear
methods can be rate-optimal for stochastic optimization problems,
while linear estimators $\what{\optvar} = A y$ can be rate-optimal
in Gaussian sequence models.
At the grossest level, we construct sets
$\optdomain$ for which $\sup_{\optvar \in \optdomain} \lone{\optvar} \lesssim 1$
while $\sup_{\optvar \in \qhull(\optdomain)} \lone{\optvar} = +\infty$, while
$\nlwidth^2(n)$ and $\width^2(n)$ are comparable. To give a slightly
more nuanced picture, we consider the rates at which the
two $\ell_1$-diameters approach $+\infty$, comparing
\begin{equation*}
  \sup_{\optvar \in \optdomain} \<\ones_n, \optvar\>
  ~~ \mbox{and} ~~
  \sup_{\optvar \in \qhull(\optdomain)} \<\ones_n, \optvar\>,
\end{equation*}
where $\ones_n \in \R^\N$ denotes the sequence with $1$ in its first
$n$ positions and $0$ elsewhere.

Here, we elaborate on the $\ell_1$ bodies yielding
Corollary~\ref{corollary:l1-bodies-gsm-nonlinear} to
consider smaller axis-aligned polyhedra. Letting $e_j$ be the
basis vectors (i.e.\ sequences with a $1$ in position $j$ and 0 elsewhere),
let $a_j$ be a nondecreasing sequence with $a_1 = 1$ and
$b_j$ be arbitrary (for now),
define the two sets
\begin{equation*}
  C_0 \defeq \left\{\sigma_j a_j e_j\right\}_{j \in \N}
  ~~~ \mbox{and} ~~~
  C_1 \defeq \bigg\{\frac{1}{Z(n)} \sum_{j = 1}^n \sigma_j b_j
  e_j \bigg\}_{n \in \N},
  ~~~ \mbox{where}~ \sigma_j \in \{\pm 1\}.
\end{equation*}
We choose the normalizing
constants $Z(n)$ so that the points in $C_1$ all lie in
\begin{equation*}
  \qhull(C_0)
  = \qhull\bigg\{\optvar \mid \sum_{j = 1}^\infty a_j |\optvar_j| \le 1
  \bigg\}
  = \bigg\{\optvar \mid \sum_{j = 1}^\infty a_j^2 \optvar_j^2 \le 1 \bigg\},
\end{equation*}
i.e.\ $Z(n) = (\sum_{j = 1}^n a_j^2 b_j^2)^{1/2}$ so that
$Z(n)^{-2} \sum_{j = 1}^n a_j^2 b_j^2 = 1$.
Then the set
\begin{equation}
  \label{eqn:domain-from-union-hull}
  \optdomain \defeq \conv\left(C_0 \cup C_1 \right)
  ~~ \mbox{satisfies} ~~
  \qhull(\optdomain) = \qhull(C_0).
\end{equation}

We obtain the following corollary of the our
convergence guarantees in Section~\ref{sec:not-qc}
and the technical lemmas we provide in Appendix~\ref{sec:proofs-qc-gaps}.
\begin{corollary}
  \label{corollary:polyhedra-stoch-opt-nonlinear}
  Let $a_j = j^{\alpha/2}$ for some $0 < \alpha < 1$ and
  $b_j^2 = 2^j$ in the construction of $\optdomain$ above.
  Then
  \begin{equation*}
    n^{-\alpha} \le \nlwidth^2(n) \le \width^2(n) \lesssim n^{-\alpha},
  \end{equation*}
  while
  \begin{equation*}
    \frac{\sup_{\optvar \in \qhull(\optdomain)} \<\ones_n, \optvar\>}{
      \sup_{\optvar \in \optdomain} \<\ones_n, \optvar\>}
    \gtrsim n^{\frac{1 - \alpha}{2}}.
  \end{equation*}
  In particular, linear methods are rate optimal for estimation
  in Gaussian sequence models, while
  stochastic optimization over $\optdomain$ requires nonlinear methods.
\end{corollary}
Summarizing the conclusions of the corollary, as $\alpha \downarrow 0$, the
ratio of the $\ell_1$-diameters of $\qhull(\optdomain)$ and $\optdomain$
grows as $\sqrt{n}$, which by Cauchy-Schwarz is as large as possible: there
is a large gap between the achievable optimization performance of nonlinear
methods, as Proposition~\ref{proposition:risks-with-l1-norms} demonstrates,
and linear methods, whose regret necessarily scales as the $\ell_1$-diameter
of $\qhull(\optdomain)$ (Theorem~\ref{theorem:regret-with-quadratic-hull}).
Yet the nonlinear widths are comparable for all $n$, so linear methods
are minimax rate optimal as $\sigma^2 \downarrow 0$.

%

\section{The need for adaptive methods}\label{sec:adaptivity}


We have so far demonstrated that diagonal re-scaling is sufficient to
achieve minimax optimal rates for problems over quadratically convex
constraint sets. In practice, however, it is often the case that we do not
know the gradient geometry in advance, precluding selection of the
optimal linear pre-conditioner.
To address this problem~\cite{DuchiHaSi11,McMahanSt10}, adaptive gradient
methods choose, at each step, a (usually diagonal) matrix $\Lambda_i$
conditional on the subgradients observed thus far, $\{g_l\}_{l\le i}$. The
algorithm then updates the iterate based on the distance generating function
$h_i(\optvar) \defeq \half \dotp{\optvar}{\Lambda_i\optvar}$.
In this section, we
present a problem instance showing that when the ``scale'' of the
subgradients varies across dimensions, adaptive gradient methods are crucial
to achieve low regret. While there exists an optimal pre-conditioner, if we
do not assume knowledge of the geometry in advance,
AdaGrad~\cite{DuchiHaSi11} achieves the minimax optimal regret while
standard (non-adaptive) subgradient methods can be $\sqrt{n}$ suboptimal
on the same problem.

We consider the following setting: $\optdomain = \ball{\infty}{0}{1}$ and
$\gamma_\beta(g) = \norm{\beta \odot g}_1$, for an arbitrary
$\beta \in \R^n, \beta\succ 0$. Intuitively, $\beta_j$ corresponds to the
``scale'' of the $j$-th dimension. On this problem, a straightforward
optimization of the regret bound~\eqref{eqn:md-regret} guarantees that
stochastic gradient methods achieve regret
$\sqrt{nk} / \min_j \beta_j$. We exhibit a problem instance in the following proof of Theorem~\ref{th:regret-sgm-wlp} such that, for any
stepsize $\alpha$, online gradient descent attains this worst-case regret.

\begin{theorem}
  \label{th:regret-sgm-wlp}
  Let $\regret_{k, \alpha}(\optvar) = \sum_{i\leq k} g_i^\top(\optvar_i - \optvar)$
  denote the regret of the online gradient descent method with stepsize
  $\alpha\geq 0$ for linear functions $\f_i(\optvar) = \dotp{g_i}{\optvar}$. For
  any choice of $\alpha \geq 0$ and $\beta \succ 0$, there exists a sequence of
  vectors $\{g_i\}_{i\leq k} \subset \R^n$, $\gamma_\beta(g_i) \le 1$ and point
  $\optvar\in\optdomain$ such that
  \begin{equation*}
    \regret_{k, \alpha}(\optvar) \ge \half\min\left\lbrace
    \frac{nk}{2\norm{\beta}_1},
    \frac{\sqrt{2nk}}{\min_{j\leq n}\beta_j}
    \right\rbrace.
  \end{equation*}
\end{theorem}
\begin{proof}
  The proof follows similar lines as those we develop in the proof of
  Theorem~\ref{th:regret-ub-lp}, but choosing different $u, v \in \R^n$.
  Let $\stepsize \geq 0$ be a stepsize.
  We consider linear functions
  $\f_i(\optvar) \defeq \dotp{g_i}{\optvar}$ with $\norm{\beta \odot g_i}_1
  \leq 1$. Let $\{\optvar_i\}_{i\leq k}$ be the iterates of online gradient
  descent. The regret with respect to $\optvar\in\R^n$ is
  \begin{equation*}
    \regret_{k, \alpha}(\optvar) = \sum_{i\leq k}\dotp{g_i}{(\optvar_i - \optvar)}.
  \end{equation*}
  This yields
  \begin{equation*}
    \regret_{k, \alpha}^\ast = \sup_{\norm{\optvar}_\infty \leq 1} \regret_{k,
      \alpha}(\optvar) = \normbigg{\sum_{i\leq k} g_i}_1 +
    \frac{\stepsize}{2}\sum_{i\leq k}\norm{g_i}_2^2 -
    \frac{\stepsize}{2}\normbigg{\sum_{i\leq k} g_i}_2^2.
  \end{equation*}
  
  Let $d = \arg\min_{j\leq n}\beta_j$, we choose
  \begin{equation*}
    u = e_d / \beta_d ~~\mbox{and}~~ v = \frac{\mathbf{1}}{\norm{\beta}_1}.
  \end{equation*}
  
  For $\delta\in[0, 1]$, we now choose the vectors $g_i \in \R^n$ as follows:
  \begin{enumerate}[(a)]
  \item $g_i = u$ for $k/4$ of the indices $i\in[k]$.
  \item $g_i = -u$ for $k/4$ of the indices $i\in[k]$.
  \item $g_i = v$ for $\frac{k}{4}(1 + \delta)$ of the indices $i\in[k]$.
  \item $g_i = -v$ for $\frac{k}{4}(1-\delta)$ of the indices $i\in[k]$.
  \end{enumerate}
  
  For this construction, we lower bound the regret
  \begin{align}\label{eq:regret-lb-wlp}
    \regret_{k, \alpha}^\ast & \ge \sup_{0\leq \delta \leq 1} \left\lbrace
    \frac{k\delta}{2}\norm{v}_1 + \frac{k\stepsize}{4} \norm{u}^2_2 -
    \frac{\stepsize \delta^2 k^2}{8}\norm{v}_2^2 \right\rbrace \nonumber \\
    & \ge \frac{k\stepsize}{4}\norm{u}_2^2 +
    \frac{k\norm{v}_1}{4}\min\left\lbrace 1,
    \frac{2\norm{v}_1}{k\alpha\norm{v}_2^2}\right\rbrace.
  \end{align}
  If the stepsize is too small (i.e. $\alpha \le
  \frac{2}{k}\frac{\norms{v}_1}{\norms{v}^2_2}$) then \eqref{eq:regret-lb-wlp}
  becomes
  \begin{equation*}
    \regret_{k, \alpha}^\ast \geq \frac{kn}{4\norm{\beta}_1}.
  \end{equation*}
  In the other case that $\alpha > \frac{2}{k}\frac{\norms{v}_1}{\norms{v}^2_2}$,
  \eqref{eq:regret-lb-wlp} yields
  \begin{equation*}
    \regret_{k, \alpha}^\ast \geq \frac{k}{4\alpha}\norm{u}_2^2 +
    \frac{\norm{v}_1^2}{\norm{v}_2^2} \frac{\alpha}{2} \geq
    \frac{\sqrt{2}}{2}\frac{\sqrt{kn}}{\min_{j\leq n}\beta_j},
  \end{equation*}
  which is the desired result.
\end{proof}

In contrast, AdaGrad~\cite{DuchiHaSi11} achieves regret
$\sqrt{k}\ltwo{1/\beta}$, demonstrating suboptimality gap as
large as $\sqrt{n}$ for some choices of $\beta$.
Indeed, let $\regret_{k, \mathsf{AdaGrad}}(\optvar)$ be the regret of
AdaGrad. Then
\begin{equation*}
  \regret_{k, \mathsf{AdaGrad}}(\optvar) \leq 2\sqrt{2}\sum_{j\leq
    n}\sqrt{\sum_{i\leq k} g_{i, j}^2}.
\end{equation*}
(see~\citet[Corollary~6]{DuchiHaSi11}), and by Cauchy-Schwarz, 
\begin{equation*}
  \sum_{j\leq n} \sqrt{\sum_{i\leq k} g_{i, j}^2} = \sum_{j\leq n}
  \frac{1}{\beta_j} \sqrt{\sum_{i\leq k} \beta_j^2 g_{i, j}^2} \leq
  \norm{1/\beta}_2 \sqrt{\sum_{i\leq k} \norm{\beta \odot g_i}^2_2} \leq
  \sqrt{k}\norm{1/\beta}_2.
\end{equation*}
To concretely consider different scales across dimensions, take
$\beta_j = j$. Theorem~\ref{th:regret-sgm-wlp} guarantees that there exists a
collection of linear functions such that stochastic gradient methods suffer
regret $\Omega(1)\sqrt{nk}$. Given that
$\norms{1/\beta}_2 \le \sqrt{\zeta(2)} \le \pi/\sqrt{6}$,
AdaGrad achieves regret
$O(1)\sqrt{k}$---amounting to a suboptimality gap of order
$\sqrt{n}$---exhibiting the need for adaptivity.
This $\sqrt{n}$ gap is also the largest possible over
subgradient methods, which achieve regret
$\sqrt{n \sum_{i \le k} \ltwo{g_i}^2}
\le \sqrt{n} \sum_{j \le n} \sqrt{\sum_{i \le k} g_{i,j}^2}$ for
$\optdomain = \ball{\infty}{0}{1}$.
Finally, we note in passing that AdaGrad is
minimax optimal on this class of problems via a straightforward application of
Theorem~\ref{th:qc-qc}.

\section{Discussion}

We provide concrete foundations to compare adaptive, mirror, or standard
gradient methods, showing how problem geometry necessarily impacts
convergence.  This paper puts a particular computational spin on
optimization by connecting to Gaussian sequence models and linear versus
nonlinear updates, which we advocate for its ability to paint a different
picture than pure polynomial versus non-polynomial computational
complexity. This perspective draws from information-based models in
optimization~\cite{NemirovskiYu83, AgarwalBaRaWa12} and  models in scientific
computing where one uses certain families of operations---e.g., matrix
vector multiplies---to build up optimal algorithms under these
constraints~\cite{TrefethenBa97, BallardDeHoSc11}.
We hope to see more exploration in these directions.

There remain a few additional open questions.
The lower bound in Corollary~\ref{corollary:lb} exhibiting a scaling of
$\sup_{\optvar \in \optdomain} \sup_{\gamma(g) \le 1} \dotp{\optvar}{g}$, is
sharp in some cases, but not in others.
A natural conjecture that we leave open for now is that
this bound is tight whenever the set $G = \{g \mid \gamma(g) \le 1\}$
is quadratically convex.
We believe this to be true because of the connections between martingale
type and the regret guarantees of mirror descent~\cite{SrebroSrTe11}: in
brief, whenever the norm $\gamma$ satisfies a so-called type-2
inequality~\cite{Pisier89}, meaning roughly that if $S_n = \sum_{i = 1}^n
g_i$ for a martingale sequence of $g_i \in G$, then $\E[\gamma(S_n)^2] \le C
\sum_{i = 1}^n \E[\gamma(g_i)^2]$ for a reasonable constant $C$.
All the ``usual'' quadratically convex norms---e.g., $\ell_q$ norms
for $q \ge 2$---satisfy such inequalities~\cite{DumbgenGeVeWe10}, and
so it is natural to imagine that quadratic convexity
coincides with such inequalities.
We leave this open.

While Section~\ref{sec:adaptivity}
emphasizes the importance of adaptivity, the picture is
not fully complete: for example, in the case of quadratically convex
constraint sets, while the best diagonal pre-conditioner achieves optimal
rates, the extent to which adaptive gradient algorithms find this optimal
pre-conditioner remains an open question. Another avenue to explore involves
the many flavors of adaptivity---while the minimax framework assumes
knowledge of the problem setting (e.g.\ a bound on the domain or the
gradient norms), it is often the case that such parameters are unknown to
the practitioner. To what extent can adaptivity mitigate this and achieve
optimal rates, and is minimax (i.e.\ worst-case) optimality truly the right
measure of performance?


\ifdefined\moorsubmission
\ACKNOWLEDGMENT{We thank Aditya Grover, Annie Marsden, and Hongseok Namkoong
  for valuable comments, and Quentin Guignard for pointing us to the
  Banach-Mazur distance for Theorem~\ref{th:regret-ub-lp}.}
\else
\paragraph{Acknowledgments.}
We thank Aditya Grover, Annie Marsden, and Hongseok Namkoong
for valuable comments, and Quentin Guignard for pointing us to the
Banach-Mazur distance for Theorem~\ref{th:regret-ub-lp}.
\fi

\setlength{\bibsep}{1pt}
\bibliography{bib}

\begin{thebibliography}{44}
\providecommand{\natexlab}[1]{#1}
\providecommand{\url}[1]{\texttt{#1}}
\expandafter\ifx\csname urlstyle\endcsname\relax
  \providecommand{\doi}[1]{doi: #1}\else
  \providecommand{\doi}{doi: \begingroup \urlstyle{rm}\Url}\fi

\bibitem[Agarwal et~al.(2012)Agarwal, Bartlett, Ravikumar, and
  Wainwright]{AgarwalBaRaWa12}
A.~Agarwal, P.~L. Bartlett, P.~Ravikumar, and M.~J. Wainwright.
\newblock Information-theoretic lower bounds on the oracle complexity of convex
  optimization.
\newblock \emph{IEEE Transactions on Information Theory}, 58\penalty0
  (5):\penalty0 3235--3249, 2012.

\bibitem[Assouad(1983)]{Assouad83}
P.~Assouad.
\newblock Deux remarques sur l'estimation.
\newblock \emph{Comptes Rendus des S\'eances de l'Acad\'emie des Sciences,
  S\'erie I}, 296\penalty0 (23):\penalty0 1021--1024, 1983.

\bibitem[Ballard et~al.(2011)Ballard, Demmel, Holtz, and
  Schwartz]{BallardDeHoSc11}
G.~Ballard, J.~Demmel, O.~Holtz, and O.~Schwartz.
\newblock Minimizing communication in numerical linear algebra.
\newblock \emph{SIAM Journal on Matrix Analysis and Applications}, 32\penalty0
  (3):\penalty0 866--901, 2011.

\bibitem[Bartlett et~al.(2007)Bartlett, Hazan, and Rakhlin]{BartlettHaRa07}
P.~L. Bartlett, E.~Hazan, and A.~Rakhlin.
\newblock Adaptive online gradient descent.
\newblock In \emph{Advances in Neural Information Processing Systems 20}, 2007.

\bibitem[Beck and Teboulle(2003)]{BeckTe03}
A.~Beck and M.~Teboulle.
\newblock Mirror descent and nonlinear projected subgradient methods for convex
  optimization.
\newblock \emph{Operations Research Letters}, 31:\penalty0 167--175, 2003.

\bibitem[Berthet and Rigollet(2013)]{BerthetRi13}
Q.~Berthet and P.~Rigollet.
\newblock Optimal detection of sparse principal components in high dimension.
\newblock \emph{Annals of Statistics}, 41\penalty0 (1):\penalty0 1780--1815,
  2013.

\bibitem[Billingsley(1999)]{Billingsley99}
P.~Billingsley.
\newblock \emph{Convergence of Probability Measures}.
\newblock Wiley, {S}econd edition, 1999.

\bibitem[Bottou et~al.(2018)Bottou, Curtis, and Nocedal]{BottouCuNo18}
L.~Bottou, F.~Curtis, and J.~Nocedal.
\newblock Optimization methods for large-scale learning.
\newblock \emph{SIAM Review}, 60\penalty0 (2):\penalty0 223--311, 2018.

\bibitem[Brennan et~al.(2018)Brennan, Bresler, and Huleihel]{BrennanBrHu18}
M.~Brennan, G.~Bresler, and W.~Huleihel.
\newblock Reducibility and computational lower bounds for problems with planted
  sparse structure.
\newblock In \emph{Proceedings of the Thirty First Annual Conference on
  Computational Learning Theory}, 2018.

\bibitem[Cai and Wu(2020)]{CaiWu20}
T.~T. Cai and Y.~Wu.
\newblock Statistical and computational limits for sparse matrix detection.
\newblock \emph{Annals of Statistics}, 48\penalty0 (3):\penalty0 1593--1614,
  2020.

\bibitem[Cesa-Bianchi and Lugosi(2006)]{CesaBianchiLu06}
N.~Cesa-Bianchi and G.~Lugosi.
\newblock \emph{Prediction, Learning, and Games}.
\newblock Cambridge University Press, 2006.

\bibitem[Cesa-Bianchi et~al.(2004)Cesa-Bianchi, Conconi, and
  Gentile]{CesaBianchiCoGe04}
N.~Cesa-Bianchi, A.~Conconi, and C.~Gentile.
\newblock On the generalization ability of on-line learning algorithms.
\newblock \emph{IEEE Transactions on Information Theory}, 50\penalty0
  (9):\penalty0 2050--2057, September 2004.

\bibitem[Cover and Thomas(2006)]{CoverTh06}
T.~M. Cover and J.~A. Thomas.
\newblock \emph{Elements of Information Theory, Second Edition}.
\newblock Wiley, 2006.

\bibitem[Cutkosky and Orabona(2018)]{CutkoskyOr18}
A.~Cutkosky and F.~Orabona.
\newblock Black-box reductions for parameter-free online learning in {B}anach
  spaces.
\newblock In \emph{Proceedings of the Thirty First Annual Conference on
  Computational Learning Theory}, 2018.

\bibitem[Cutkosky and Sarlos(2019)]{CutkoskySa19}
A.~Cutkosky and T.~Sarlos.
\newblock Matrix-free preconditioning in online learning.
\newblock In \emph{Proceedings of the 36th International Conference on Machine
  Learning}, 2019.

\bibitem[Donoho and Johnstone(1995)]{DonohoJo95}
D.~L. Donoho and I.~M. Johnstone.
\newblock Adapting to unknown smoothness via wavelet shrinkage.
\newblock \emph{Journal of the American Statistical Association}, 90\penalty0
  (432):\penalty0 1200--1224, 1995.

\bibitem[Donoho et~al.(1990)Donoho, Liu, and MacGibbon]{DonohoLiMa90}
D.~L. Donoho, R.~C. Liu, and B.~MacGibbon.
\newblock Minimax risk over hyperrectangles, and implications.
\newblock \emph{Annals of Statistics}, 18\penalty0 (3):\penalty0 1416--1437,
  1990.

\bibitem[Duchi(2018)]{Duchi18}
J.~C. Duchi.
\newblock Introductory lectures on stochastic convex optimization.
\newblock In \emph{The Mathematics of Data}, IAS/Park City Mathematics Series.
  American Mathematical Society, 2018.

\bibitem[Duchi et~al.(2011)Duchi, Hazan, and Singer]{DuchiHaSi11}
J.~C. Duchi, E.~Hazan, and Y.~Singer.
\newblock Adaptive subgradient methods for online learning and stochastic
  optimization.
\newblock \emph{Journal of Machine Learning Research}, 12:\penalty0 2121--2159,
  2011.

\bibitem[Duchi et~al.(2013)Duchi, Jordan, and McMahan]{DuchiJoMc13}
J.~C. Duchi, M.~I. Jordan, and H.~B. McMahan.
\newblock Estimation, optimization, and parallelism when data is sparse.
\newblock In \emph{Advances in Neural Information Processing Systems 26}, 2013.

\bibitem[D\"umbgen et~al.(2010)D\"umbgen, {van de Geer}, Veraar, and
  Wellner]{DumbgenGeVeWe10}
L.~D\"umbgen, S.~{van de Geer}, M.~Veraar, and J.~Wellner.
\newblock Nemirovski's inequalities revisited.
\newblock \emph{American Mathematical Monthly}, 117:\penalty0 138--160, 2010.

\bibitem[Fan(1953)]{Fan53}
K.~Fan.
\newblock Minimax theorems.
\newblock \emph{Proceedings of the National Academy of Sciences}, 39\penalty0
  (1):\penalty0 42--47, 1953.

\bibitem[Gentile(2003)]{Gentile03}
C.~Gentile.
\newblock The robustness of the $p$-norm algorithms.
\newblock \emph{Machine Learning}, 53\penalty0 (3):\penalty0 265--299, 2003.

\bibitem[Hiriart-Urruty and Lemar\'echal(1993)]{HiriartUrrutyLe93}
J.~Hiriart-Urruty and C.~Lemar\'echal.
\newblock \emph{Convex Analysis and Minimization Algorithms {I}}.
\newblock Springer, New York, 1993.

\bibitem[Johnstone(2017)]{Johnstone17}
I.~Johnstone.
\newblock \emph{Gaussian Estimation: Sequence and Wavelet Models}.
\newblock 2017.
\newblock Available online at
  \url{https://imjohnstone.su.domains/GE_08_09_17.pdf}.

\bibitem[McMahan and Streeter(2010)]{McMahanSt10}
B.~McMahan and M.~Streeter.
\newblock Adaptive bound optimization for online convex optimization.
\newblock In \emph{Proceedings of the Twenty Third Annual Conference on
  Computational Learning Theory}, 2010.

\bibitem[Nemirovski and Yudin(1983)]{NemirovskiYu83}
A.~Nemirovski and D.~Yudin.
\newblock \emph{Problem Complexity and Method Efficiency in Optimization}.
\newblock Wiley, 1983.

\bibitem[Nemirovski et~al.(2009)Nemirovski, Juditsky, Lan, and
  Shapiro]{NemirovskiJuLaSh09}
A.~Nemirovski, A.~Juditsky, G.~Lan, and A.~Shapiro.
\newblock Robust stochastic approximation approach to stochastic programming.
\newblock \emph{SIAM Journal on Optimization}, 19\penalty0 (4):\penalty0
  1574--1609, 2009.

\bibitem[Nesterov(2009)]{Nesterov09}
Y.~Nesterov.
\newblock Primal-dual subgradient methods for convex problems.
\newblock \emph{Mathematical Programming}, 120\penalty0 (1):\penalty0 261--283,
  2009.

\bibitem[Orabona and Crammer(2010)]{OrabonaCr10}
F.~Orabona and K.~Crammer.
\newblock New adaptive algorithms for online classification.
\newblock In \emph{Advances in Neural Information Processing Systems 23}, 2010.

\bibitem[Orabona and P{\'a}l(2018)]{OrabonaPa18}
F.~Orabona and D.~P{\'a}l.
\newblock Scale-free online learning.
\newblock \emph{Theoretical Computer Science}, 716:\penalty0 50--69, 2018.

\bibitem[Pisier(1989)]{Pisier89}
G.~Pisier.
\newblock \emph{The {V}olume of {C}onvex {B}odies and {B}anach {S}pace
  {G}eometry}, volume~94 of \emph{Cambridge Tracts in Mathematics}.
\newblock Cambridge University Press, Cambridge, UK, 1989.

\bibitem[Robbins and Monro(1951)]{RobbinsMo51}
H.~Robbins and S.~Monro.
\newblock A stochastic approximation method.
\newblock \emph{Annals of Mathematical Statistics}, 22:\penalty0 400--407,
  1951.

\bibitem[Shalev-Shwartz(2007)]{Shalev07}
S.~Shalev-Shwartz.
\newblock \emph{Online Learning: Theory, Algorithms, and Applications}.
\newblock PhD thesis, The Hebrew University of Jerusalem, 2007.

\bibitem[Shalev-Shwartz(2012)]{Shalev12}
S.~Shalev-Shwartz.
\newblock Online learning and online convex optimization.
\newblock \emph{Foundations and Trends in Machine Learning}, 4\penalty0
  (2):\penalty0 107--194, 2012.

\bibitem[Shalev-Shwartz et~al.(2009)Shalev-Shwartz, Shamir, Srebro, and
  Sridharan]{ShalevShSrSr09}
S.~Shalev-Shwartz, O.~Shamir, N.~Srebro, and K.~Sridharan.
\newblock Stochastic convex optimization.
\newblock In \emph{Proceedings of the Twenty Second Annual Conference on
  Computational Learning Theory}, 2009.

\bibitem[Sion(1958)]{Sion58}
M.~Sion.
\newblock On general minimax theorems.
\newblock \emph{Pacific Journal of Mathematics}, 8\penalty0 (1):\penalty0
  171--176, 1958.

\bibitem[Srebro et~al.(2011)Srebro, Sridharan, and Tewari]{SrebroSrTe11}
N.~Srebro, K.~Sridharan, and A.~Tewari.
\newblock On the universality of online mirror descent.
\newblock 2011.

\bibitem[Trefethen and Bau~III(1997)]{TrefethenBa97}
L.~N. Trefethen and D.~Bau~III.
\newblock \emph{Numerical Linear Algebra}.
\newblock SIAM, 1997.

\bibitem[Tsybakov(2009)]{Tsybakov09}
A.~B. Tsybakov.
\newblock \emph{Introduction to Nonparametric Estimation}.
\newblock Springer, 2009.

\bibitem[Vershynin(2009)]{Vershynin09}
R.~Vershynin.
\newblock Lectures in geometric functional analysis.
\newblock Unpublished manuscript, 2009.
\newblock URL \url{https://www.math.uci.edu/~rvershyn/papers/GFA-book.pdf}.

\bibitem[Wainwright(2019)]{Wainwright19}
M.~J. Wainwright.
\newblock \emph{High-Dimensional Statistics: A Non-Asymptotic Viewpoint}.
\newblock Cambridge University Press, 2019.

\bibitem[Wilson et~al.(2017)Wilson, Roelofs, Stern, Srebro, and
  Recht]{WilsonRoStSrRe17}
A.~C. Wilson, R.~Roelofs, M.~Stern, N.~Srebro, and B.~Recht.
\newblock The marginal value of adaptive gradient methods in machine learning.
\newblock In \emph{Advances in Neural Information Processing Systems 30}, 2017.

\bibitem[Yu(1997)]{Yu97}
B.~Yu.
\newblock Assouad, {F}ano, and {L}e {C}am.
\newblock In \emph{Festschrift for Lucien Le Cam}, pages 423--435.
  Springer-Verlag, 1997.

\end{thebibliography}
\bibliographystyle{abbrvnat}

\ifdefined\moorsubmission
\begin{APPENDICES}
\else
\appendix
\fi


%




\section{Duality results and minimax linear estimators}
\label{sec:proofs-minimax-linear-estimators}

In this section, we collect the deferred
proofs of various optimality results for the Gaussian
sequence model in Section~\ref{sec:gsm} that we use.


\subsection{Proof of Proposition~\ref{proposition:gsm-linear}}
\label{sec:proof-gsm-linear}

Recalling the form~\eqref{eqn:linear-gsm-shorthand} of the risk $\gsmrisk(A,
x) = x^T(A - I)^T(A - I) x + \sigma^2 \lfro{A}^2$, it is immediate that $A$
is unique, as $\lfro{A}^2$ is strongly convex in $A$ and $\sup_{x \in X}
\ltwo{(A - I)x}^2$ is convex in $A$. As $X$ is compact, the risk
$\maxgsmrisk(A, X)$ is finite for all $A$, and hence its minimum is attained.

We now consider the saddle point problem
\begin{equation*}
  \inf_A \sup_{x \in X} \gsmrisk(A, x)
  = \inf_A \sup_{x \in X} \left\{
  x^T A^T A x - 2 x^T A x + \sigma^2 \lfro{A}^2 \right\}.
\end{equation*}
By lifting to the space $\measures(X)$ of probability measures defined on
$X$, we have for each $A$ that
\begin{equation*}
  \maxgsmrisk(A, X)
  = \sup_{\nu \in \measures(X)}
  \int \tr((A - I)^T (A - I) xx^T) d\nu(x)
  + \sigma^2 \lfro{A}^2.
\end{equation*}
We shall apply Fan's minimax theorem~\cite{Fan53}, a specialization
of Sion's minimax theorem~\cite{Sion58}
to construct the $A$ minimizing this worst-case risk.
\begin{lemma}
  \label{lemma:fan}
  Let $X$ and $Y$ be compact convex subsets of (possibly distinct)
  topological vector spaces and $L : X \times Y \to \R$ be convex in its
  first argument, concave in its second, and continuous. Then
  \begin{equation*}
    \inf_{x \in X} \sup_{y \in Y} L(x, y)
    = \sup_{y \in Y} \inf_{x \in X}  L(x, y)
  \end{equation*}
  and the infimum and supremum are attained.
\end{lemma}
\noindent
To apply the theorem, we also require a vector space structure on
$\measures(X)$. For this, we rely on the Wasserstein distances, defined as
follows: for a (subset of a) vector space $X$ with norm $\norm{\cdot}$ and
function $f : X \to \R$, the Lipschitz norm on $f$ is $\lipnorm{f} =
\sup_{x \neq y \in X} |f(x) - f(y)| / \norm{x - y}$.  The Wasserstein
distance between measures $\nu, \mu$ on the space $X$ is then
\begin{equation*}
  W(\mu, \nu) = \sup_{\lipnorm{f} \le 1}
  \int f (d\mu - d\nu).
\end{equation*}
The Wasserstein distance metrizes convergence in distribution $\mu_n \cd
\mu$ if the measures $\mu_n, \mu$ have bounded first moment and turns
the space of measures on $X$ into a (topological) vector space.

In our case, we define
\begin{equation*}
  L(A, \nu) \defeq
  \int \tr((A - I)^T (A - I) xx^T) d\nu(x)
  + \sigma^2 \lfro{A}^2,
\end{equation*}
which is clearly convex/concave and continuous in $A$. To see
the continuity of $L$ in
$\nu$, we note that for any matrix $B$ and any $x, y \in X$ we
have
\begin{equation*}
  \<x x^T - yy^T, B\> = (x - y)^T (B+B^T)(x + y)/2 \le 2 \opnorm{B}
  \ltwo{x - y} \diam(X),
\end{equation*}
so that $x \mapsto \<xx^T, B\>$ is Lipschitz
and
\begin{equation*}
  |L(A, \nu) - L(A, \mu)|
  \le 2 \opnorm{A - I}^2 \diam(X) \cdot W(\nu, \mu).
\end{equation*}
Thus $L$ is continuous in $\nu$ for for the Wasserstein distance, and
$\measures(X)$ is compact for this topology by Prokhorov's
theorem~\cite{Billingsley99}.  In particular, if we use the shorthand
$X_\nu = \int xx^T d\nu(x)$, there exist $A, \nu$ such that
\begin{equation}
  \label{eqn:saddle-point}
  \inf_B L(B, \nu) = L(A, \nu)
  = \tr((A - I)^T (A - I) X_\nu) + \sigma^2 \lfro{A}^2
  = \sup_{\mu \in \measures(X)}
  L(A, \mu).
\end{equation}

We now construct the $A$ solving the saddle point
problem~\eqref{eqn:saddle-point}, that is, given $\nu$, we show the
(unique) $A$ minimizing $L(A, \nu)$. Taking derivatives of $L(A, \nu)$ and
recalling the shorthand $X_\nu = \int xx^T d\nu(x)$, we see that $A$ must
satisfy
\begin{equation*}
  A X_\nu + X_\nu A^T - 2 X_\nu + \sigma^2 (A + A^T) = 0.
\end{equation*}
If $X_\nu$ has spectral decomposition $X_\nu = U \Lambda U^T$
we let $A = UDU^T$ for a diagonal
matrix $D$ to be determined, and it is evidently enough to solve
\begin{equation*}
  2 D \Lambda - 2 \Lambda + 2 \sigma^2 D = 0,
  ~~~ \mbox{or} ~~~
  D = (\Lambda + \sigma^2 I)^{-1} \Lambda.
\end{equation*}

In particular, the choice
$A = (X_\nu + \sigma^2 I)^{-1/2} X_\nu (X_\nu + \sigma^2 I)^{-1/2}$ is
optimal; it is also unique for the given $\nu$ as $A \mapsto L(A, \nu)$
is strongly convex in $A$.

Finally, we show that without loss of generality, we may take $A$ to be
diagonal. If $S$ is a diagonal matrix of independent random signs, then
$\E_\nu[(Sx)(Sx)^T] = \diag(X_\nu) = \E_\nu[\diag(x)^2]$. Let $\wb{\nu}$
be the measure on $X$ induced by drawing $x \sim \nu$ and then multiplying
$x$ by the random signs $S x$.  Notably, we have $\tr(X_\nu) =
\tr(X_{\wb{\nu}})$ and $\tr(D X_\nu) = \tr(D X_{\wb{\nu}})$ for any
diagonal matrix $D$, as $\diag(X_\nu) = \diag(X_{\wb{\nu}})$.  Suppose for
the sake of contradiction that $L(\diag(A), \nu) > L(A, \nu)$. In this
case, $A$ must be non-diagonal, and so we have
\begin{align*}
  L(\diag(A), \nu)
  & = \tr(\diag(A)^2 X_\nu) - 2 \tr(\diag(A) X_\nu)
  + \tr(X_\nu) + \sigma^2 \lfro{\diag(A)}^2 \\
  & = \tr(\diag(A)^2 X_{\wb{\nu}})
  - 2 \tr(\diag(A) X_{\wb{\nu}}) + \tr(X_{\wb{\nu}}) + \sigma^2 \lfro{\diag(A)}^2 \\
  & \stackrel{(\star)}{<}
  \tr(\diag(A^2) X_{\wb{\nu}}) - 2 \tr(\diag(A) X_{\wb{\nu}}) + \tr(X_{\wb{\nu}})
  + \sigma^2 \lfro{A}^2 \\
  & =
  \tr(A^2 X_{\wb{\nu}}) - 2 \tr(\diag(A) X_{\wb{\nu}}) + \tr(X_{\wb{\nu}})
  + \sigma^2 \lfro{A}^2,
\end{align*}
where inequality~$(\star)$ follows because $\lfro{A} > \lfro{\diag(A)}$
while $\diag(A)^2 \preceq \diag(A^2)$.
The final line follows because $X_{\wb{\nu}}$ is diagonal,
so $\diag(A^2 X_{\wb{\nu}}) = \diag(A^2) X_{\wb{\nu}}$.
Finally,
noting that $\tr(\diag(A) X_{\wb{\nu}}) = \tr(A X_{\wb{\nu}})$, we see that
$L(\diag(A), \nu) < L(A, \wb{\nu})$, and so we have demonstrated
that
$L(A, \nu) < L(A, \wb{\nu})$. But this contradicts the assumed maximality
of $\nu$, and so it must be the case that $A$ is diagonal.

Now that we have $A$ diagonal, the claimed equality is immediate, and
we also notice that $L(D, \nu) = L(D, \wb{\nu})$ for any diagonal $D$.

\subsection{Proof of Corollary~\ref{corollary:finite-dim-linear-risk-gsm}}
\label{sec:proof-finite-dim-linear-risk-gsm}

Proposition~\ref{proposition:gsm-linear} implies that
\begin{align}
  \inf_A \sup_{\optvar \in \optdomain}
  \E_\optvar\left[\ltwos{Ay - \optvar}^2\right]
  & = \inf_{d \in \R^n}
  \sup_{\optvar \in \optdomain}
  \sum_{j = 1}^n \left((d_j - 1)^2 \optvar_j^2 + \sigma^2 d_j^2\right)
  \nonumber \\
  & = \inf_d \sup_{v \in \squaredhull(\optdomain)}
  \sum_{j = 1}^n \left((d_j - 1)^2 v_j + \sigma^2 d_j^2 \right)
  \nonumber \\
  \label{eqn:swap-inf-sup-qhull}
  & = \sup_{v \in \squaredhull(X)} 
  \inf_d \left\{\sum_{j = 1}^n
  \left((d_j - 1)^2 v_j + \sigma^2 d_j^2 \right) \right\},
\end{align}
where equality~\eqref{eqn:swap-inf-sup-qhull}
is a standard convex/concave saddle-point result
(we may without loss of generality restrict $d$ to the set $[0, 1]^n$).
Continuing
the equalities, we have
\begin{equation*}
  \inf_d \left\{(d_j - 1)^2 v_j + \sigma^2 d_j^2 \right\}
  = \frac{\sigma^2 v_j}{v_j + \sigma^2},
\end{equation*}
which implies the corollary.

\subsection{Proof of Corollary~\ref{corollary:linear-risk-general-gsm}}
\label{sec:proof-linear-risk-general-gsm}

For any linear operator $A : \R^\N \to \R^\N$,
we may write
\begin{equation*}
  (Ay)_j = a_j(y)
  = a_j(x + \noise)
  = a_j(x) + a_j(\noise)
\end{equation*}
for each $j$. Let $\Pi_n : \R^\N \to \R^n$ be the projection onto the
first $n$ coordinates of a vector and $Z_n : \R^\N \to \R^\N$ be the
projection zeroing the first $n$ elements. Then
\begin{equation*}
  \inf_A \sup_{\optvar \in \optdomain}
  \E_\optvar\left[\ltwos{Ay - \optvar}^2\right]
  \ge \inf_A \sup_{\optvar \in \Pi_n\optdomain}
  \E_\optvar[\ltwo{Ay - x}^2]
  = \inf_A \sup_{\optvar \in \Pi_n\optdomain}
  \sum_{j = 1}^n \E_\optvar[(a_j((x, \zeros) + \noise) - \optvar_j)^2],
\end{equation*}
where $(\optvar, \zeros) \in \Pi_n \optdomain \times \R^\N$. Then
by linearity, for $\optvar \in \R^n$ we write
\begin{equation*}
  a_j((\optvar, \zeros) + \noise)
  = \varphi_j(\optvar) + a_j(Z_n \noise) + a_j((I - Z_n) \noise)
\end{equation*}
where $\varphi_j : \R^n \to \R$ is the linear function
$\varphi_j(x) = a_j((x, \zeros))$. But $Z_n \noise$ and $(I - Z_n) \noise
= (\noise_1, \ldots, \noise_n, \zeros)$
are independent,
and because $\E[a_j(Z_n \noise)] = 0$ we obtain
\begin{align*}
  \E_\optvar[(a_j((\optvar, \zeros) + \noise) - x_j)^2]
  & = \E[(\varphi_j(\optvar + [\noise_j]_{j \le n})
    - \optvar_j
    + a_j(Z_n \noise))^2] \\
  & = \E[(\varphi_j(\optvar + [\noise_j]_{j \le n}) - \optvar_j)^2]
  + \var(a_j(Z_n \noise)).
\end{align*}
The optimal choice of $a_j$ then necessarily satisfies $a_j(Z_n u) = 0$
for all $u$.  Thus, by restricting to finite dimensions, we use
Proposition~\ref{proposition:gsm-linear} to see that for any $n$,
\begin{align*}
  \inf_A \sup_{\optvar \in \optdomain}
  \E_\optvar\left[\ltwos{Ay - \optvar}^2\right]
  & \ge \inf_{d \in \R^n}
  \sup_{\optvar \in \optdomain}
  \sum_{j = 1}^n \left((d_j - 1)^2 \optvar_j^2 + \sigma^2 d_j^2\right) \\
  & = \sup_{v \in \squaredhull(\optdomain)}
  \inf_d \sum_{j = 1}^n \left((d_j - 1)^2 v_j + \sigma^2 d_j^2 \right)
  = \sup_{v \in \squaredhull(\optdomain)}
  \sum_{j = 1}^n \frac{\sigma^2 v_j}{v_j + \sigma^2}
\end{align*}
as in the proof of Corollary~\ref{corollary:finite-dim-linear-risk-gsm}.

By compactness, for each $\epsilon > 0$, we can choose $N < \infty$
such that $\sup_{\optvar \in \optdomain} \sum_{j > N} x_j^2 < \epsilon$.
We thus have upper bound
\begin{equation*}
  \inf_A \maxgsmrisk(A, \optdomain)
  \le \inf_A \sup_{\optvar \in \Pi_N \optdomain}
  \E_\optvar[\ltwo{Ay - \optvar}^2]
  + \epsilon.
\end{equation*}
Apply the same proof technique as that in
Corollary~\ref{corollary:finite-dim-linear-risk-gsm} to obtain
\begin{equation*}
  \inf_A \maxgsmrisk(A, \Pi_N\optdomain)
  = \sup_{v \in \squaredhull(\Pi_N \optdomain)}
  \sum_{j = 1}^N \frac{\sigma^2 v_j}{v_j + \sigma^2}
  = \sup_{v \in \squaredhull(\optdomain)}
  \sum_{j = 1}^N \frac{\sigma^2 v_j}{v_j + \sigma^2}.
\end{equation*}
Now use that $\epsilon > 0$ was arbitrary, $\optdomain$ is compact, and
take $n$ and $N$ to infinity.

\subsection{Proof of Proposition~\ref{proposition:gsm-lecam-lower}}
\label{sec:proof-gsm-lecam-lower}

We apply Le Cam's two point method via the reduction from estimation
to testing that Lemma~\ref{lemma:est-to-test} implies. Consider
for any fixed $b > 0$ the
problem of estimating a single value $\optvar \in [-b, b]$
with $y \sim \normal(\optvar, \sigma^2)$. Then
for $v \in \{-1, 1\}$, letting $P_v$ be the
normal distribution $\normal(\delta v, \sigma^2)$ for some $\delta \in [0, b]$
to be chosen,
we have
\begin{equation*}
  \inf_{\what{\optvar}} \sup_{\optvar \in [-b, b]}
  \E_\optvar[(\what{\optvar} - \optvar)^2]
  \ge \inf_{\what{\optvar}}
  \half \left\{\E_{P_1}[(\what{\optvar} - \optvar)^2]
  + \E_{P_{-1}}[(\what{\optvar} - \optvar)^2] \right\}
  \ge \frac{\delta^2}{2}
  \left(1 - \tvnorm{P_{-1} - P_1}\right).
\end{equation*}
Recalling
the Hellinger
distance
$\dhel^2(P, Q) = 1 - \int \sqrt{dP dQ}$
between probabilities and
the standard
relationship
$\tvnorm{P_{-1} - P_1}
\le \dhel(P_1, P_{-1})
\sqrt{2 - \dhel^2(P_1, P_{-1})}$,
we note that
$\dhel^2(\normal(\mu_0, \sigma^2),
\normal(\mu_1, \sigma^2))
= 1 - \exp(-\frac{1}{8 \sigma^2} (\mu_0 - \mu_1)^2)$
to obtain
\begin{equation*}
  \inf_{\what{\optvar}} \sup_{\optvar \in [-b, b]}
  \E_\optvar\left[(\what{\optvar} - \optvar)^2\right]
  \ge \sup_{0 \le \delta \le b}
  \frac{\delta^2}{2} \left(1 -
  \sqrt{1 - \exp(-\delta^2 / \sigma^2)}\right)
  \ge \frac{\sigma^2 \wedge b^2}{10}
\end{equation*}
via the choice $\delta = \min\{\sigma, b\}$.
We thus obtain for any hypercube $H = [-\optvar_j, \optvar_j]_{j \ge 1}$
that
\begin{equation*}
  \maxgsmrisk(\optdomain)
  \ge \maxgsmrisk(H) \ge \frac{1}{10} \sum_{j \ge 1}
  \sigma^2 \wedge \optvar_j^2.
\end{equation*}

\subsection{Proof of Corollary~\ref{corollary:soft-thresh-minimax}}
\label{sec:proof-soft-thresh-minimax}

We adapt the arguments in~\citep[pg.~1430]{DonohoLiMa90}.
For $N = N(\sigma, \optdomain)$ we have
\begin{align*}
  \E[\ltwos{\what{\optvar} - \optvar}^2]
  & = \sum_{j = 1}^N
  \E[(\softthresh_\lambda(y_j) - \optvar_j)^2]
  + \sum_{j > N} \optvar_j^2 \\
  & \le \sigma^2 + (1 + 2 \log N) \sum_{j = 1}^N \optvar_j^2 \wedge \sigma^2
  + \sum_{j > N} \optvar_j^2 \wedge \sigma^2,
\end{align*}
as $\optvar_j^2 \le \sigma^2$ for
$j > N(\sigma, \optdomain)$, which implies the first result.
Proposition~\ref{proposition:gsm-lecam-lower} implies the second.


\section{Proofs for Section~\ref{sec:general}}

\subsection{Proof of Theorem~\ref{th:qc-not-qc}}
\label{sec:proof-qc-not-qc}

The upper bound is simply Corollary~\ref{cor:ub}.
For the lower bound,
similar to our warm-up in Section~\ref{sec:warm-up}, we consider ``sparse''
gradients, though instead of using Fano's method we use Assouad's method to
more carefully relate the geometry of the norm $\gamma$ and constraint set
$\optdomain$.

Let $a \in \R^n_+$ be such that $\rect(a) \subset \optdomain$. We consider the
sample space $\statdomain := \{\pm e_j\}_{j\leq n}$ and functions
\begin{equation*}
  \f(\optvar, \statval) \defeq \sum_{j\leq n}\frac{1}{\gamma(e_j)}|\statval_j||\optvar_j - a_j\statval_j|.
\end{equation*}
For any $\statval\in\statdomain$, the subdifferential $\partial_\optvar \f(\optvar, \statval)$
has at most one non-zero
coordinate; the orthosymmetry of $\gamma$ implies $\f\in\F{\gamma}{1}$. Let
$p \in \R^n_+$ (to be specified presently)
be such that $\mathbf{1}^\top p=1$ and for $1\leq j
\leq n$, let $\delta_j \in [0, 1/2]$. We define the distributions
$P_v$ on $\statdomain$ by
\begin{equation*}
  \statrv = \begin{cases}
    v_j e_j & \mbox{with~probability~} \frac{p_j(1+\delta_j)}{2} \\
    -v_j e_j & \mbox{with~probability~} \frac{p_j(1-\delta_j)}{2}.
  \end{cases}
\end{equation*}
With this choice, we evidently have
\begin{equation*}
  \ff_v(\optvar) = \E_{\statrv\sim P_v}\f(\optvar, \statrv) = \sum_{j\leq n}
  \frac{p_j}{\gamma(e_j)}\left[\frac{1+\delta_j}{2}|\optvar_j - a_j v_j| +
  \frac{1-\delta_j}{2}|\optvar_j + a_j v_j|\right]
\end{equation*}
and immediately that $\inf_\optdomain \ff_v = \sum_{j\leq n}
\frac{p_ja_j}{\gamma(e_j)}(1-\delta_j)$. As a consequence,
we have the Hamming separation (recall Eq.~\eqref{eqn:hamming-separation})
\begin{equation*}
  \ff_v(\optvar) - \inf_\optdomain \ff_v = \sum_{j\leq n}
  \frac{p_ja_j\delta_j}{\gamma(e_j)} \mathbf{1}_{\sign(\optvar_j) \neq v_j},
\end{equation*}
which allows us to apply Assouad's method via Lemma~\ref{lem:assouad}.

Using the same notation as Lemma~\ref{lem:assouad}, we have
\begin{equation*}
  \tvnorm{\P^k_{+j} - \P^k_{-j}}^2 \le \frac{1}{2}\dkl{\P^k_{+j}}{\P^k_{-j}} \le
  \log 3 \cdot k p_j\delta_j^2.
\end{equation*}
Choosing $\delta_j = \min\lbrace\frac{1}{2},
\frac{1}{2\sqrt{k p_j\log(3)}}\rbrace$ yields the lower bound
\begin{equation*}
  \minimaxs(\optdomain, \gamma) \ge \frac{1}{8}\sum_{j\leq n}
  \frac{a_j}{\gamma(e_j)}
  \min\left\{ p_j, \frac{\sqrt{p_j}}{\sqrt{k \log 3}}\right\},
\end{equation*}
and by taking $p_j = (\frac{a_j}{\gamma(e_j)})^2 / \ltwo{
  \rescale(a, \gamma)}^2$, we obtain for any $a \in \optdomain$ that
\begin{align*}
  \minimaxs(\optdomain, \gamma) &  \ge
  \minimaxs(\rect(a), \gamma) \\
   & \ge \frac{1}{8} \sum_{j \leq n} \frac{a_j}{\gamma(e_j)}
  \min\left\{ \frac{a_j^2}{\gamma(e_j)^2
    \ltwo{\rescale(a, \gamma)}^2},
    \frac{1}{\sqrt{k \log 3}}\frac{a_j}{\gamma(e_j)
      \ltwo{\rescale(a, \gamma)}}
    \right\} \\
  & = \frac{1}{8 \ltwo{\rescale(a, \gamma)}^2}
  \sum_{j = 1}^n \frac{a_j^2}{\gamma(e_j)^2}
  \min\left\{ \frac{a_j}{\gamma(e_j)},
  \frac{\ltwo{\rescale(a, \gamma)}}{\sqrt{k \log 3}}\right\}
\end{align*}

For notational simplicity, define the set $T \defeq \{\rescale(\optvar,
\gamma)) \mid \optvar \in \optdomain\}$, which is evidently orthosymmetric
and convex (it is a diagonal scaling of $\optdomain$). Then
\begin{equation}
  \label{eqn:intergalactic-planetary}
  \minimaxs(\optdomain, \gamma)
  \ge \sup_{u \in T}
  \frac{1}{8 \ltwo{u}^2}
  \sum_{j = 1}^n u_j^2 \min\left\{u_j, \frac{\ltwo{u}}{\sqrt{k \log 3}}
  \right\}.
\end{equation}
For any vector $u \in \R_+^n$ and $c < 1$, if we define $J = \{j \in [n] \mid
u_j \ge \frac{c}{\sqrt{n}} \ltwo{u}\}$, then
\begin{equation*}
  \ltwo{u}^2 = \ltwo{u_J}^2 + \ltwo{u_{J^c}}^2
  \le \ltwo{u_J}^2 + \ltwo{u}^2 \sum_{j \in J^c} \frac{c^2}{n}
  \le \ltwo{u_J}^2 + c^2 \ltwo{u}^2,
\end{equation*}
i.e. $\ltwo{u_J} \ge \sqrt{1 - c^2} \ltwo{u}$. Now, fix $d \in \N$. If in the supremum~\eqref{eqn:intergalactic-planetary}
we consider any vector $u \in T, u \ge 0$ satisfying $\norm{u}_0 \le d$,
then setting the index set $J = \{j : u_j \ge \ltwo{u} / \sqrt{k \log 3}\}
= \{j : u_j \ge \ltwo{u} / \sqrt{d (k/d) \log 3}\}$
we have
\begin{align*}
  \minimaxs(\optdomain, \gamma)
  & \ge
  \frac{1}{8 \ltwo{u}^2}
  \sum_{j = 1}^n u_j^2 \min\left\{u_j,
  \frac{\ltwo{u}}{\sqrt{k \log 3}}
  \right\}
  \ge
  \frac{1}{8 \ltwo{u}^2}
  \sum_{j \in J} u_j^2 \frac{\ltwo{u}}{\sqrt{k \log 3}}
  \\ & \ge \frac{1}{8} \left(1 - \frac{d}{k \log 3}\right)
  \frac{\ltwo{u}}{\sqrt{k \log 3}}.
\end{align*}
Taking a supremum over $u$ with $\norm{u}_0 \le d$ gives the theorem.

\subsection{Proof of Corollary~\ref{cor:qc-wlp}}
\label{sec:proof-cor-qc-not-qc}

Given the proof of Theorem~\ref{th:qc-not-qc}, the proof is nearly
immediate. Let $p \in [1, 2]$, $\beta \in \R_{++}^n$,
and $\gamma(v) = \norm{\beta \odot v}_p$.  For the lower bound,
the final display of the proof of Theorem~\ref{th:qc-not-qc} above
guarantees the lower bound
$\minimaxs(\optdomain, \gamma)
\ge \frac{1}{16} \ltwo{u} / \sqrt{k}$ for all $u
\in \{\rescale(\optvar, \gamma) \mid \optvar \in \optdomain\}$ and $k \ge 2n$.
We first observe that
$\qhull (\ball{\gamma}{0}{1}) = \{v \in \R^n \mid \norm{\beta \odot v}_2 \leq
1\}$. Thus, the upper bound in Theorem~\ref{th:qc-not-qc} is
\begin{equation*}
  \minimaxr(\optdomain, \gamma) \le \frac{1}{\sqrt{k}}\sup_{\optvar\in\optdomain}
  \sup_{g: \norm{\beta \odot g}_2 \leq 1} \optvar^\top g.
\end{equation*}
Using
\begin{equation*}
  \sup_{g: \norm{\beta \odot g}_2 \leq 1} u^\top g =
  \sup_{z: \norm{z}_2 \leq 1} u^\top\left(z / \beta\right) = \norm{u / \beta}_2,
\end{equation*}
and recalling $\beta_j = \gamma(e_j)$ concludes the proof.

\section{Proofs for Section~\ref{sec:not-qc}}

\subsection{Proof of Theorem~\ref{th:sharp-md}}\label{prf:th-sharp-md}

Let us tackle the first case stated in the theorem; we reduce the second
case to the first one by scaling the dimension.
As in the proofs of Proposition~\ref{prop:lp-ball-q12}
and~\ref{prop:lp-ball-q2infty}, let $c_p = n^{-1/p}$, and let $p^* =
\frac{p}{p - 1}$ be conjugate to $p$, so that
$c_{p^*} \norm{\ones}_{p^*} = 1$.

\subsubsection{Case $1 \leq p \leq 1 + 1/ \log(2n)$}
Lemma~\ref{lem:1-d} always implies the lower bound $1 / \sqrt{k}$, which we
may obtain by reducing to a lower-dimensional problem, so we assume without
loss of generality that $n \ge 8$.

\paragraph{Separation}
Take $\packset = \{\pm e_j\}_{j\leq n}$,
and for $\packval = \pm e_j \in \packset$,
we define $P_\packval$ on $\statrv\in\{\pm 1\}^n$ by choosing coordinates
of $\statrv$ independently via
\begin{equation*}
  \statrv_j =
  \begin{cases}
    1 & \mbox{~~with probability~~} \frac{1+\delta v_j}{2} \\
    -1 & \mbox{~~with probability~~} \frac{1-\delta v_j}{2},
  \end{cases}
\end{equation*}
whence evidently $\E_{P_\packval}[\statrv] = \delta \packval$.
For
$\statval\in\{\pm 1\}^n$, construct $\f \in \F{\gamma}{1}$ by defining
\begin{equation*}
  \f(\optvar, \statval) \defeq
  \frac{c_{p^*}}{2} \left[\dotp{\statval}{\optvar}
    + \hinge{\lone{\optvar} - 1}\right].
\end{equation*}
%
Then the population objective $\ff_\packval(\optvar) \defeq
\E_{P_\packval}[\f(\optvar; \statrv)] = \frac{c_{p^*}}{2} [\delta
  \dotp{\packval}{\optvar} + \hinge{\lone{\optvar} - 1}]$, which has
minimizer $\optvar_\packval = -\packval \in \optdomain$
and minimal value
$\inf_\optvar \ff_\packval(\optvar)
= \inf_{\optvar \in \optdomain} \ff_\packval(\optvar)
= -\frac{c_{p^*}}{2} \delta \ltwo{\packval}^2
= -\frac{c_{p^*}}{2} \delta$.
The sum $\ff_{\packval}(\optvar) + \ff_{\packval'}(\optvar)
= \frac{c_{p^*}}{2}[\delta \dotp{(\packval + \packval')}{\optvar}
  + 2 \hinge{\lone{\optvar} - 1}]$
has minimizer $\optvar = -\half(\packval + \packval')$.
Combining these calculations gives separation
\begin{align*}
  \dopt(\packval, \packval', \optdomain)
  = c_{p^*}
  \left[\delta
    - \frac{\delta}{4} \ltwo{\packval + \packval'}^2
    \right]
  \ge \frac{c_{p^*} \delta}{2}
  = \frac{n^{-1/p^*} \delta}{2},
\end{align*}
the inequality holding whenever $\packval \neq \packval'
\in \{\pm e_j\}_{j=1}^n$.
Lemma~\ref{lemma:opt-to-est} yields
\begin{equation*}
  \minimaxs(\optdomain, \gamma) \ge \frac{\delta n^{-1/p^\ast}}{4}
  \inf_{\psi:\statdomain^k \to \mc{V}}\P(\psi(\statrv_1^k) \neq V).
\end{equation*}
It now remains to bound the testing error.

\paragraph{Bounding the testing error} Noting that $|\mc{V}| = \log(2n)$, we
lower bound the testing error via Fano's inequality
\begin{equation*}
  \inf_{\psi:\statdomain^k \to \mc{V}}\P(\psi(\statrv_1^k) \neq V) \ge \left(1-\frac{\mathsf{I}(\statrv_1^k;V) + \log 2}{\log(2n)}\right).
\end{equation*}
For any $v\neq v'\in\mc{V}$, we have for $\delta \in [0, \half]$ that
\begin{equation*}
  \dkl{P_v}{P_{v'}} = \delta\log\frac{1+\delta}{1-\delta} \le 3\delta^2.
\end{equation*}
We can thus bound the mutual information between $\statrv_1^n$ and $V$
\begin{equation*}
  \mi{\statrv_1^k}{V} \le k\max_{v\neq v'}\dkl{P_v}{P_{v'}} \le 3k\delta^2.
\end{equation*}
In the case that $n < 8$, the lower bound holds trivially via
Lemma~\ref{lem:1-d}.
In the case that $n \ge 8$,
choosing $\delta^2 = \frac{\log(2n)}{6k} \wedge \half$
yields
\begin{equation}\label{eqn:lb-log2d}
  \minimaxs(\optdomain, \gamma) \geq
  \frac{n^{-1/p^*}}{4}
  \min\left\{\sqrt{\frac{\log(2 n)}{6k}}, \half\right\}
  \left(1 - \frac{1}{2} - \frac{1}{4}\right),
\end{equation}
which is valid for all $p\in[1, 2]$. In the case that
$1\leq p \leq 1+1/\log(2n)$, we note that
$n^{-1/p^\ast} = 1 / n^{\frac{p-1}{p}} \geq 1/e$, which yields
\begin{equation*}
  \minimaxs(\optdomain, \gamma) \geq c \cdot
  \sqrt{\frac{\log (2 n)}{k}} \wedge 1
\end{equation*}
for a numerical constant $c > 0$.

To conclude, we see that the upper bound follows immediately from
Proposition~\ref{prop:rate-md}.
Choosing $a = 1 + 1/\log(2n)$ and $a^* = \frac{a}{a-1} = 1 + \log(2n)$, the
quantity $\frac{\sup_{\optvar\in\optdomain} \norm{\optvar}_a \sup_{g \in
    \ball{\gamma}{0}{1}} \norm{g}_{a^*}}{\sqrt{a-1}\sqrt{k}}$ upper bounds
the minimax regret.
As $a \ge p$, $\sup_{\optvar\in\optdomain} \|\optvar\|_a
= 1$, and $a^* \le p^*$ so
\begin{equation*}
  \|g\|_{a^\ast} \leq n^{\frac{1}{a^\ast} - \frac{1}{p^\ast}} \|g\|_{p^\ast}
  \leq n^{\frac{1}{a^\ast}}
\end{equation*}
as $g\in \ball{p^\ast}{0}{1}$. Once we note that both $n^{1/a^\ast} =
\exp(\frac{\log n}{\log(2 n) + 1}) \leq e$ and $1 /
\sqrt{a-1} = \sqrt{\log(2n)}$, we conclude this case.

\subsubsection{Case $1 + 1/\log(2n) < p \leq 2$}

Let $n_0 \leq n$. We can embed a function $\f_{n_0}:\R^{n_0}\times
\statdomain \to \R$ as a function $\f:\R^n \times \statdomain \to \R$ by
letting $\pi_{n_0}$ denote the projection onto the first $n_0$-components,
and defining
\begin{equation*}
  \f(\optvar, \statval) =
  \f_{n_0}(\pi_{n_0}\optvar, \statval).
\end{equation*}
If the subgradients of $\f_{n_0}$ lie in $\ball{p^\ast}{0}{1}$, so do those of
$\f$. Similarly, if $\optvar_0 \in \{\tau \in \R^{n_0}, \|\tau\|_p \leq 1 \}$
then $\optvar = (\optvar_0, \mathbf{0}_{n_0+1:n}) \in \ball{p}{0}{1}$. As such,
any lower bound for the $n_0$-dimensional problem implies an identical one for
all $n \ge n_0$-dimensional problems. For $1 + 1/\log(2n) < p \leq 2$, let us
define $n_0 = \lceil \exp(\frac{1}{p-1}) / 2 \rceil$, so $n_0 \leq n$ as
desired. In the case that $p > 1 + 1/\log 16$, Lemma~\ref{lem:1-d} yields the
desired lower bound. In the case that $p \le 1 + 1 / \log 16$, we have that
$n_0 \geq 8$, and the lower
bound~\eqref{eqn:lb-log2d} holds so that for a numerical constant
$c > 0$,
\begin{equation*}
  \minimaxs(\optdomain, \gamma) \geq c
  n_0^{-1/p^\ast} \cdot \sqrt{\frac{\log(2n_0)}{n}} \wedge 1.
\end{equation*}
Once we use that $n_0^{-1/p^\ast} \geq (1/2)^{\frac{1}{p} - 1}\exp(-1/p)
\geq \sqrt{2/e}$, substituting $n_0 = \lceil \exp(\frac{1}{p-1}) / 2\rceil$
above gives the final lower bound
\begin{equation*}
  \minimaxs(\optdomain, \gamma) \geq c \cdot
  \frac{1}{\sqrt{2(p-1) n}} \wedge 1.
\end{equation*}
Proposition~\ref{prop:rate-md}
yields the upper bound and concludes this proof.

%

\subsection{Proof of Observation~\ref{observation:self-similar}}
\label{sec:proof-self-similar}

The first claim of the observation
is trivial. For the second and third, we use a bit of additional
notation, saying that $\optdomain$ is $\beta$-self-similar if
\begin{equation}
  \label{eqn:polynomial-l1-similar}
  \sup_{\optvar \in \optdomain} \sum_{j = 1}^\infty j^\beta |\optvar_j|
  \le e \sup_{\optvar \in \optdomain} \lone{\optvar}.
\end{equation}
We show that if $\lim_{\beta \downarrow 0}
\sup_{\optvar \in \optdomain} \sum_{j = 1}^\infty j^\beta |\optvar_j|
< \infty$, then there is some $\beta > 0$ for which
$\optdomain$ is $\beta$-self-similar~\eqref{eqn:polynomial-l1-similar}.
The key is the following claim:
\begin{lemma}
  Let $\optdomain$ satisfy
  $\lim_{\beta \downarrow 0} \sup_{\optvar \in \optdomain}
  \sum_{j = 1}^\infty j^\beta |\optvar_j| < \infty$. Then there
  exists $\beta_0 > 0$ such that $\beta
  \mapsto s(\beta) \defeq \sup_{\optvar \in \optdomain} \sum_{j = 1}^\infty
  j^\beta |\optvar_j|$ is continuous on $\openright{0}{\beta_0}$, and
  $\lim_{\beta \downarrow 0} s(\beta) = \sup_{\optvar \in
    \optdomain} \lone{\optvar}$.
\end{lemma}
\begin{proof}
  Let $\beta_0 = \sup\{\beta > 0 \mid s(\beta) < \infty\}$. Take $\beta \in
  \openright{0}{\beta_0}$, and let $\beta' \to \beta$ in
  $\openright{0}{\beta_0}$. By the assumption
  that $\sup_{\optvar \in \optdomain} \sum_{j = 1}^\infty j^\beta |\optvar_j|
  < \infty$ for all $\beta < \beta_0$, for any
  $\epsilon > 0$ and all suitably small $\gamma > 0$, we may take $N$ such
  that $\sup_{\optdomain} \sum_{j > N} j^{\beta + \gamma} |\optvar_j| <
  \epsilon$.  Then for $\beta'$ close enough to $\beta$, we obtain
  \begin{equation*}
    |s(\beta) - s(\beta')|
    \le \sup_{\optvar \in \optdomain}
    \sum_{j \le N} |j^\beta - j^{\beta'}| |\optvar_j|
    + \sup_{\optvar \in \optdomain}
    \sum_{j > N} (j^\beta + j^{\beta'}) |\optvar_j|
    \le 
    \sup_{\optvar \in \optdomain}
    \sum_{j \le N} |j^\beta - j^{\beta'}| |\optvar_j|
    + 2 \epsilon.
  \end{equation*}
  As $\beta \mapsto j^\beta$ is continuous in $\beta$,
  when $\beta'$ is close enough to $\beta$ we obtain
  $N |j^\beta - j^{\beta'}| \le \epsilon$ for all
  $j \le N$. So
  $|s(\beta) - s(\beta')| \le \epsilon \sup_{\optvar \in \optdomain}
  \lone{\optvar} + 2 \epsilon$. As $\epsilon > 0$ was arbitrary
  this completes the proof.
\end{proof}

\noindent
The preceding lemma demonstrates that
$\lim_{\beta \downarrow 0} \sup_{\optvar \in \optdomain} \sum_{j = 1}^\infty
j^\beta |\optvar_j| = \sup_{\optvar \in \optdomain}
\lone{\optvar}$, so there is some $\beta > 0$
for which
$\sup_{\optvar \in \optdomain} \sum_{j = 1}^\infty
j^\beta |\optvar_j|
\le e \sup_{\optvar \in \optdomain} \lone{\optvar}$. Thus
$\effdim(\optdomain) < \infty$.

Finally, for the final claim of the observation,
let $\ones_n$ denote the linear functional
$\<\ones_n, \optvar\> = \sum_{j = 1}^n \optvar_j$, and let
$N_\gamma$ be the smallest
$N$ such that
$\sup_{\optvar \in \optdomain}
\<\ones_N, \optvar\> \ge \sup_{\optvar \in \optdomain}
n^{2\gamma} \sum_{j > n} |\optvar_j|$
for all $n \ge N$.
Define
\begin{equation*}
  N \defeq \max\left\{N_\gamma, \exp\left(1 + \frac{3}{2 \gamma}
  \log \frac{3}{2\gamma}\right),
  \exp\left(\frac{1}{2 \log(e - \sqrt{e})}\right)
  \right\}.
\end{equation*}
We claim that $\optdomain$ is
$\beta$-self-similar~\eqref{eqn:polynomial-l1-similar}
for $\beta = \frac{1}{2 \log N}$, so that
$\effdim(\optdomain)
\le \exp(1/\beta)
= N^2$ as desired.
To see the claim of self-similarity, note that
the choice $\beta = \frac{1}{2 \log N}$ guarantees
$\beta - 2 \gamma \le
-\frac{3 \gamma}{2}$, and
\begin{align*}
  \sum_{j = 1}^\infty
  j^\beta |\optvar_j|
  & \le N^{\frac{1}{2 \log N}}
  \sum_{j = 1}^N |\optvar_j|
  + \sum_{j = N + 1}^\infty
  j^\beta |\optvar_j|
  \le \sqrt{e} \sup_{\optvar \in \optdomain}
  \<\ones_N, x\>
  + \sum_{j = N + 1}^\infty
  j^\beta |\optvar_j|.
\end{align*}
Define the index blocks $\mc{B}_k \defeq
\{i \in \N \mid e^k \le i < e^{k + 1}\}$, so that
\begin{align*}
  \sum_{j = N + 1}^\infty
  j^\beta |\optvar_j|
  & \le \sum_{k = \floor{\log N}}^\infty
  \sum_{j \in \mc{B}_k} j^\beta |\optvar_j|
  \le \sum_{k = \floor{\log N}}^\infty
  e^{\beta (k + 1)}
  \frac{1}{e^{2 k \gamma}} \sup_{\optvar \in \optdomain}
  \<\ones_N, \optvar\>
\end{align*}
because $\sum_{j \in \mc{B}_k} |\optvar_j|
\le e^{-2 k \gamma} \sup_{\optvar \in \optdomain}
\<\ones_N, \optvar\>$ by assumption.
Computing the infinite sums, we have
that because $\frac{1}{2 \log N} \le \frac{\gamma}{2}$,
\begin{align*}
  \sum_{k = \floor{\log N}}^\infty
  e^{\beta (k + 1)}
  \frac{1}{e^{2 k \gamma}}
  & = e^\beta \sum_{k = \floor{\log N}}^\infty
  e^{k(\beta - 2 \gamma)}
  \le
  e^\beta e^{-3 \gamma \floor{\log N} / 2}
  \sum_{k = 0}^\infty e^{-3 k \gamma / 2} \\
  & =   
  \frac{e^{\beta - 3 \gamma \floor{\log N} / 2}}{
    1 - e^{-3 \gamma / 2}}
  \le \left(\frac{N}{e}\right)^{-\frac{3}{2}\gamma}
  \frac{2 e^\beta}{3 \gamma},
\end{align*}
where we used that
$e^x \ge 1 + x$, or
$1 - e^x \le -x$.
Finally, noting that
$(N/e)^{-\kappa} \le \kappa$
iff $\log N \ge 1 + \frac{1}{\kappa} \log \frac{1}{\kappa}$,
we substitute $\kappa = \frac{3 \gamma}{2}$ to
obtain
$\sum_{j = N + 1}^\infty j^\beta |\optvar_j|
\le e^\beta
\sup_{\optvar \in \optdomain}
\<\ones_N, \optvar\>$.
We have shown that
\begin{equation*}
  \sum_{j = 1}^\infty j^\beta |\optvar_j|
  \le (\sqrt{e} + e^\beta) \sup_{\optvar \in \optdomain}
  \<\ones_N, \optvar\>
\end{equation*}
for $\beta = \frac{1}{2 \log N}$.
So long as $\beta \le \log(e - \sqrt{e})$ we have
$\sqrt{e} + e^\beta \le e$.



\section{Proof of Theorem~\ref{theorem:regret-with-quadratic-hull}}
\label{sec:proof-regret-with-quadratic-hull}

As in the proof of Theorem~\ref{th:regret-ub-lp}, let $A \succ 0$ be a
positive definite matrix, so that
the regret of (Euclidean) mirror descent with distance generating
function $h_A(\optvar) = \half \optvar^\top A \optvar$ is
\begin{equation*}
	\regret_{k,A}(\optvar) = \bigg\<\sum_{i \le k} g_i, \optvar\bigg\>
	+ \half \sum_{i \le k} \norm{g_i}_{A^{-1}}^2
	- \half \normbigg{\sum_{i \le k} g_i}_{A^{-1}}^2.
\end{equation*}
For vectors $u, v \in \gradomain$ to be chosen and $p \in (0, 1)$ to be
chosen as well, we set
\begin{enumerate}[(a)]
	\item $g_i = u$ for $p k / 2$ of the indices $i \in [k]$
	\item $g_i = -u$ for $p k / 2$ of the indices $i \in [k]$
	\item $g_i = v$ for $(1 - p) k$ of the indices $i \in [k]$.
\end{enumerate}
This then yields regret lower bound
\begin{equation}
	\label{eqn:regret-for-linf-intermediate}
	\regret_{k,A}(\optvar) \ge
	(1 - p) k v^\top \optvar
	- \frac{(1 - p)^2 k^2}{2} v^\top A^{-1} v
	+ \frac{p k}{2}
	u^\top A^{-1} u.
\end{equation}

We argue we may assume w.l.o.g.\ that $\limsup_k \opnorm{A(k)} / \sqrt{k}
< \infty$. Suppose to the contrary that along a subsequence, which
for simplicity we take to be the entire sequence, that
$\opnorm{A(k)} \gg \sqrt{k}$. Let $C < \infty$ be arbitrary,
take $k$ large enough that $\opnorm{A(k)} \ge C \sqrt{k}$,
and assume w.l.o.g.\ that $C \le \sqrt{k}$
(we can always take $k$ larger).
Let $\delta_0 > 0$ be such that $\ball{2}{0}{\delta_0} \subset \gradomain$,
which of course exists because $\gradomain$ is a convex body.
Let $w$ be the unit eigenvector
of $A = A(k)$ achieving $w^\top A w = \opnorm{A}$, and set
$v = \delta w$ for a $0 \le \delta \le \delta_0$ to be chosen,
so that $v = \delta w \in \gradomain$. Then
at such indices $k$ the lower bound~\eqref{eqn:regret-for-linf-intermediate}
implies
\begin{align*}
	\regret_{k,A(k)}(\optvar)
	&\ge (1 - p)k \delta w^\top \optvar
	- \frac{(1 - p)^2 k^2 \delta^2}{2 \opnorm{A(k)}}
	+ \frac{k p}{2} u^\top A^{-1} u
	\\
	&\ge \frac{k (1 - p) \delta}{2}
	\left[w^\top \optvar - \frac{(1 - p)\sqrt{k} \delta}{2 C}\right].
\end{align*}
Taking a supremum over $\optvar \in \optdomain$, which is a convex
body (so that it has interior), we have
$\sup_{\optvar \in \optdomain} w^\top \optvar = c(\optdomain) > 0$,
whence for large enough $k$ we have
\begin{align*}
	\sup_{\optvar \in \optdomain} \regret_{k,A(k)}(\optvar)
	&\ge \sup_{0 \le \delta \le \delta_0}
	k (1 - p) \delta \left[c(\optdomain) - \frac{(1 - p)
		\sqrt{k} \delta}{2C}
	\right] \\
	&\ge \frac{C (1 - p) \cdot \min\{c(\optdomain), 1\}}{2} \sqrt{k},
\end{align*}
where the second inequality sets $\delta = C \cdot \min\{c(\optdomain),
1\} / \sqrt{k}$, which satisfies $\delta \le \delta_0$ for $k$ large enough.
As $c(\optdomain) > 0$ and $C < \infty$ was otherwise
arbitrary, whenever $p < 1$ the preceding lower bound is stronger than
that the theorem claims.

Returning to the main thread, we may therefore assume that
$\sup_k \opnorm{A(k)} / \sqrt{k} \le C$ for some finite $C$. Returning
to the regret bound~\eqref{eqn:regret-for-linf-intermediate},
we optimize over $v$ and $u$ to obtain
\begin{equation*}
	\regret_{k,A}(\optvar)
	\ge 
	\sup_{u, v \in \gradomain}
	\left[k(1 - p) v^\top \optvar - \frac{k^2(1 - p)^2}{2}
	v^\top A^{-1} v + \frac{k p}{2} u^\top A^{-1} u \right].
\end{equation*}
Considering the supremum over $v$, we have for any $A \succ 0$ that
\begin{equation*}
	\argmax_v \left\{k(1 -p) v^\top \optvar - \frac{k^2 (1 - p)^2}{2}
	v^\top A^{-1} v \right\}
	= \frac{1}{k(1 - p)} A \optvar.
\end{equation*}
Because $\optdomain$ is bounded and $\opnorm{A} \le C / \sqrt{k}$,
for $p < 1$ and suitably large $k$ the $v$ achieving this supremum evidently
satisfies $\ltwo{v} = \ltwo{Ax} / (k(1 - p)) \le \delta_0$, so that
$v \in \gradomain$. Then for
any (fixed) $p \in (0, 1)$, for all large enough $k$ we obtain
\begin{equation}
	\label{eqn:regret-with-matrix-sum}
	\regret_{k,A}(\optvar) \ge \sup_{u \in \gradomain}
	\frac{1}{2} \left[k p u^\top A^{-1} u
	+ \optvar^\top A \optvar \right].
\end{equation}

We use a duality argument to lower bound the
quantity~\eqref{eqn:regret-with-matrix-sum}.  Let $\mc{P}(\optdomain)$
and $\mc{P}(\gradomain)$
denote the collections of probability measures on $\optdomain$ and
$\gradomain$, respectively. Then
\begin{align*}
	&\inf_{A \succeq 0}
	\sup_{\optvar \in \optdomain}
	\sup_{u \in \gradomain} \left\{k p u^\top A^{-1} u
	+ \optvar^\top A \optvar \right\} \\
	& \ge \inf_{A \succeq 0} \sup_{\nu \in \mc{P}(\optdomain)}
	\sup_{\mu \in \mc{P}(\gradomain)}
	\left[k p \<A^{-1}, \E_\mu[uu^\top]\> + \<A, \E_\nu[\optvar \optvar^\top]\>
	\right] \\
	& \ge \sup_{\nu \in \mc{P}(\optdomain)}
	\sup_{\mu \in \mc{P}(\gradomain)}
	\inf_{A \succeq 0} \left\{ k p \cdot
	\<A^{-1}, \E_\mu[uu^\top]\>
	+ \<A, \E_\nu[\optvar \optvar^\top]\> \right\}.
\end{align*}
For shorthand, fix $\optdomain_\mu = \E_\nu[xx^\top]$ and $\gradomain_\mu =
\E_\mu[uu^\top]$. Then taking derivatives with respect to $A$, we see that the
inner infimum is achieved whenever
\begin{equation}
	\label{eqn:solve-best-A-regret}
	-kp A^{-1} \gradomain_\mu A^{-1} + \optdomain_\nu = 0,
	~~ \mbox{i.e.} ~~
	A \optdomain_\nu A = kp \gradomain_\mu.
\end{equation}
For $X \succ 0$ and $G \succ 0$,
the matrix equation $A X A = G$ has solution
\begin{equation*}
	A = X^{-1/2} (X^{1/2} G X^{1/2})^{1/2} X^{-1/2},
\end{equation*}
so the solution~\eqref{eqn:solve-best-A-regret} satisfies
\begin{equation*}
	A = \sqrt{kp} \, \optdomain_\nu^{-1/2}
	\left(\optdomain_\nu^{1/2} \gradomain_\mu \optdomain_\nu^{1/2}\right)^{1/2}
	\optdomain_\nu^{-1/2}.
\end{equation*}

We must argue that (for an appropriate constant $C < \infty$) we
have $\opnorm{A} \le C \sqrt{k}$ for this choice  of $A$. But
for any finite $b < \infty$,
so long as $\nu$ satisfies $\E_\nu[\optvar \optvar^\top] \succeq b^{-1}
I_n$, we certainly have $\opnorm{A}
\le C \sqrt{k}$, where $C < \infty$ grows at most as $1 / b$ and depends
on the diameters of $\optdomain$ and $\gradomain$.
Substituting this value for $A$ then yields
\begin{align*}
	\half kp \cdot \<A^{-1}, \gradomain_\mu\>
	+ \<A, \optdomain_\nu\>
	= \sqrt{kp} \cdot
	\tr \left(\left(\optdomain_\nu^{1/2}
	\gradomain_\mu \optdomain_\nu^{1/2}\right)^{1/2}\right).
\end{align*}
Substituting into the regret lower
bound~\eqref{eqn:regret-with-matrix-sum}, we have for any $b < \infty$ and
$p \in (0, 1)$ that for all large enough $k$
\begin{align*}
	\lefteqn{\sup_{\optvar \in \optdomain}
		\regret_{k,A}(\optvar)} \\
	& \ge \sqrt{kp} \sup_{\nu \in \mc{P}(\optdomain)}
	\sup_{\mu \in \mc{P}(\gradomain)}
	\left\{ \tr\left(
	\left(\E_\nu[\optvar \optvar^\top]^{1/2}
	\E_\mu[gg^\top] \E_\nu[\optvar \optvar^\top]^{1/2}\right)^{1/2}\right)
	\mid \E_\nu[\optvar \optvar^\top] \succeq b^{-1} I_n
	\right\}.
\end{align*}  

Finally, we use the following lemma relating quadratic
hulls and measures.
\begin{lemma}
	\label{lemma:concave-to-qhull}
	Let $\optdomain$ and $\gradomain$ be orthosymmetric convex bodies and
	$\mc{P}(\optdomain)$ denote the collection of probability measures
	on $\optdomain$ (and similarly for $\gradomain$).
	Define
	\begin{equation*}
		R(\mu, \nu)
		\defeq 
		\tr\left(
		\left(\E_\nu[\optvar \optvar^\top]^{1/2}
		\E_\mu[gg^\top] \E_\nu[\optvar \optvar^\top]^{1/2}\right)^{1/2}\right).
	\end{equation*}
	Then
	\begin{align*}
		\sup_{\mu \in \mc{P}(\gradomain)}
		\sup_{\nu \in \mc{P}(\optdomain)}
		R(\mu, \nu)
		& = \sup_{\mu \in \mc{P}(\gradomain)}
		\sup_{\nu \in \mc{P}(\optdomain)}
		\sum_{j = 1}^n \sqrt{\E_\nu[\optvar_j^2]
			\E_\mu[g_j^2]}
		= \sup_{q \in \qhull(\optdomain)}
		\sup_{y \in \qhull(\gradomain)} \<q, y\>.
	\end{align*}
	Moreover, the suprema can be taken over symmetric measures.
\end{lemma}
\begin{proof}
	We prove the first equality first. The function
	$R(\mu, \nu)$ is jointly concave in $\mu$ and $\nu$, as
	we may write it as the infimum
	$R(\mu, \nu) = \half \inf_{A \succeq 0}
	\{\<A^{-1}, \E_\mu[gg^\top]\>
	+ \<A, \E_\nu[\optvar \optvar^\top]\>\}$ of linear functions
	of $\mu$ and $\nu$. Fix measures
	$\mu_+$ and $\nu_+$, and for
	a diagonal matrix of signs $S$, let
	$\mu_-$ and $\nu_-$ be the induced measures
	on $g \sim \mu$ and $\optvar \sim \nu$. Then
	note that
	\begin{equation*}
		R(\mu_-, \nu_-)
		= \half \inf_{A \succ 0}
		\tr\left(\<S A^{-1} S, \E_{\mu_+}[gg^\top]\>
		+ \<S A S, \E_{\nu_+}[\optvar \optvar^\top]\> \right)
		= R(\mu_+, \nu_+),
	\end{equation*}
	because $SA^{-1} S = (SA S)^{-1}$, as $S^2 = I$.
	(Joint) concavity then yields
	\begin{equation*}
		R\left(\half (\mu_+ + \mu_-),
		\half (\nu_+ + \nu_-)\right)
		\ge \half R(\mu_+, \nu_+)
		+ \half R(\mu_-, \nu_-)
		= R(\mu_+, \nu_+).
	\end{equation*}
	Then for any measure $\mu$, if we allow $\wb{\mu}$ to be the
	induced measure on $S \optvar$ for
	$\optvar \sim \mu$ and $S$ drawn uniformly from diagonal sign matrices,
	and similarly for $\nu$ and $\wb{\nu}$, we obtain
	\begin{equation*}
		R\left(\wb{\mu}, \wb{\nu}\right)
		\ge R(\mu, \nu).
	\end{equation*}
	Of course, $\E_{\wb{\mu}}[gg^\top] =
	\E_\mu[S gg^\top S]
	= \diag(\E_\mu[g_j^2]_{j = 1}^n)$ is diagonal, as is
	$\E_{\wb{\nu}}[\optvar \optvar^\top]$, and so
	\begin{equation*}
		\sup_{\mu,\nu} R(\mu, \nu)
		= \sup_{\mu, \nu}
		\sum_{j = 1}^n
		\sqrt{\sqrt{\E_\nu[\optvar_j^2]}
			\E_\mu[g_j^2] \sqrt{\E_\nu[\optvar_j^2]}}
		=
		\sup_{\mu, \nu}
		\sum_{j = 1}^n \sqrt{\E_\mu[g_j^2]}
		\sqrt{\E_\nu[\optvar_j^2]}.
	\end{equation*}
	
	For the second equality, recall that for
	any vector $q \in \qhull(\optdomain)$ with $q \succeq 0$, we
	may write $q = \sqrt{z}$ (applied elementwise), where
	$z$ is a convex combination of vectors of the form
	$[\optvar_j^2]_{j = 1}^n$, $\optvar \in \optdomain$.
	Let $\nu \in \mc{P}(\optdomain)$. Then
	for $\optvar^i \simiid \nu$, the strong law of large numbers
	guarantees that $\frac{1}{n} \sum_{i = 1}^n \optvar^i
	(\optvar^i)^\top \to \E_\nu[\optvar \optvar^\top]$ with probability
	1, and so for any measure $\nu \in \mc{P}(\optdomain)$
	and any $\epsilon > 0$,
	there
	exists a finite set of vectors $\{x^i\}_{i = 1}^m$
	with $|\frac{1}{m} \sum_{i = 1}^m (x^i_j)^2
	- \E_\nu[\optvar_j^2]| \le \epsilon$.
	The same is true for $\mu \in \mc{P}(\gradomain)$, and so
	a continuity argument gives the second equality.
\end{proof}

By a slight perturbation, we therefore
obtain that for any $\epsilon > 0$, we can choose
$b$ large enough that for all large $k$,
using the $R$ in Lemma~\ref{lemma:concave-to-qhull} we have
\begin{align*}
	\sup_{\optvar \in \optdomain}
	\regret_{k,A(k)}(\optvar)
	& \ge \sqrt{kp}
	\sup_{\mu \in \mc{P}(\gradomain), \nu \in \mc{P}(\optdomain)}
	\left\{R(\mu, \nu) \mid \E_\nu[\optvar  \optvar^\top]
	\succeq b^{-1} I_n\right\} \\
	& \ge (1 - \epsilon) \sqrt{k p}
	\sup_{q \in \qhull(\optdomain)}
	\sup_{y \in \qhull(\gradomain)} \<q, y\>.
\end{align*}
As $\epsilon > 0$ and $p < 1$ were arbitrary, this
completes the proof of the lower bound.

The upper bound on the regret follows immediately from
Corollary~\ref{cor:ub} once we recognize that any orthosymmetric compact
convex set $G$ defines a norm via $\gamma(g) = \inf\{t > 0 \mid t^{-1} g \in
G\}$ (see~\cite[Def.~V.1.2.4 and Ch.~V.3.2]{HiriartUrrutyLe93}).



\section{Proofs related to $\ell_1$-diameters and $n$-widths}
\label{sec:proofs-qc-gaps}

Here, we collect the lemmas necessary to prove
Corollary~\ref{corollary:polyhedra-stoch-opt-nonlinear}.  Recall that
$\ones_n \in \R^\N$ denotes the vector with $1$ in the first $n$ positions
and 0 elsewhere.

\begin{lemma}
  \label{lemma:l1-diameters-union-hull}
  For the set $\optdomain = \conv(C_0 \cup C_1)$
  that equation~\eqref{eqn:domain-from-union-hull}
  defines, we have
  \begin{equation*}
    \sup_{q \in \qhull(\optdomain)} \<\ones_n, q\>
    = \sqrt{\sum_{j = 1}^n a_j^{-2}}
    ~~ \mbox{and} ~~
    \sup_{\optvar \in \optdomain} \<\ones_n, \optvar\>
    = \max_{m \le n}
    \frac{\sum_{j = 1}^m b_j}{\sqrt{\sum_{j = 1}^m a_j^2 b_j^2}}
    = \max_{m \le n} \frac{\<\ones_m, b\>}{Z(m)}.
  \end{equation*}
\end{lemma}
\begin{proof}
  Let $Q = \qhull(\optdomain)
  = \{q \mid \sum_{j = 1}^\infty a_j^2 q_j^2 \le 1\}$.
  By Cauchy-Schwarz, the suprema of
  $\<\ones_n, q\>$ over $q \in Q$ satisfy
  $q_j = \frac{\lambda}{a_j^2}$ where $\lambda > 0$ normalizes
  $q$ so that $\sum_{j = 1}^n a_j^2 q_j^2 = 1$, that is,
  $\lambda = (\sum_{j = 1}^n a_j^{-2})^{-1/2}$.
  For the second equality,
  note that $\<\ones_n, \optvar\>$ is linear in $\optvar$, and so
  the supremum is achieved at one of the vertices of $C_0$ or $C_1$.
  Thus
  \begin{equation*}
    \sup_{\optvar \in \optdomain}\<\ones_n, \optvar\>
    = \max_{\optvar \in C_0} \<\ones_n, \optvar\>
    \vee \max_{\optvar \in C_1} \<\ones_n, \optvar\>
    = \max_{j \le n} \frac{1}{a_j}
    \vee \max_{m \le n} \frac{1}{Z(m)} \<\ones_m, b\>.
  \end{equation*}
  Substitute $1 = \max_{j \le n} \frac{1}{a_j}$, as $a_1 = 1$ and $a_j$
  are nondecreasing, then recognize that
  $b_1 / \sqrt{b_1^2} = 1$ to obtain the lemma.
\end{proof}

We now give rough bounds on the widths of the set $X$ and its hull.
\begin{lemma}
  \label{lemma:widths-of-polyhedron}
  For the set $\optdomain = \conv(C_0 \cup C_1)$, we have
  \begin{equation*}
    \width^2(n) = \sup_{m \ge n} \frac{m - n}{\sum_{j = 1}^m a_j^2}
  \end{equation*}
  and
  \begin{equation*}
    \width^2(n) \ge \nlwidth^2(n) \ge \sup_{m \ge n} \frac{1}{Z(m)^2}
    \sum_{j = n + 1}^m b_j^2
    = \sup_{m \ge n} \frac{\sum_{j = n + 1}^m b_j^2}{\sum_{j = 1}^m
      a_j^2 b_j^2}.
  \end{equation*}
\end{lemma}
\begin{proof}
  For the linear width, we recognize that $Q = \{q \mid \sum_{j = 1}^\infty
  a_j^2 q_j^2 \le 1\}$ is elliptical, so using the
  characterization~\eqref{eqn:width-of-l1-bodies} of $\width^2(n)$ gives the
  first claim of the lemma.  For the second
  we can take as a lower bound the nonlinear width of the set
  $C_1$, so that
  \begin{equation*}
    \nlwidth^2(n)
    \ge \sup_{x \in C_1}
    \bigg\{\sum_{j > n} x_{(j)}^2\bigg\}
    = \sup_{m \ge n}
    \bigg\{ \frac{1}{Z(m)^2} \sum_{j = n + 1}^m b_j^2\bigg\}
  \end{equation*}
  as desired.
\end{proof}

Finally, we take the scalars $a_j$, $b_j$ as in the statement of
Corollary~\ref{corollary:polyhedra-stoch-opt-nonlinear}.
Set
\begin{align*}
  a_j = j^{\alpha/2} ~~ \mbox{and} ~~ b_j^2 = 2^j,
\end{align*}
where $0 < \alpha < 1$. Then direct
calculations yield the
asymptotics that for $m \ge n$,
\begin{align}
    \sum_{j=1}^m a_j^2 & = \sum_{j=1}^m j^\alpha
    \in \left[\int_0^m t^\alpha dt,
      \int_1^{m + 1} t^\alpha dt \right]
    \asymp m^{\alpha+1}, \nonumber
    \\
    2^m m^\alpha
    = a_m^2 b_m^2
    \le \sum_{j=1}^m a_j^2 b_j^2
    & = 2^m m^\alpha \sum_{j = 1}^m \left(\frac{j}{m}\right)^\alpha
    2^{j - m}
    \le 2^m m^\alpha \sum_{j = 0}^{m - 1} 2^{-j}
    \le 2^{m + 1} m^\alpha \nonumber \\ 
    2^m = b_m^2
    \le \sum_{j=n+1}^m b_j^2 &
    \leq \sum_{j=1}^m 2^j \le 2^{m + 1}.
  \label{eqn:growth-beta-a}
\end{align}
The first equation in~\eqref{eqn:growth-beta-a}
implies that
\begin{align*}
  \width^2(n) =
  \sup_{m \ge n} \frac{m - n}{\sum_{j = 1}^m a_j^2}
  \asymp \sup_{m \geq n} \frac{m-n}{m^{1 + \alpha}} \asymp n^{-\alpha},
\end{align*}
while the last two equations lower bound the nonlinear width (via
Lemma~\ref{lemma:widths-of-polyhedron}) by
\begin{align*}
  \nlwidth^2(n) \ge
  \sup_{m \ge n}
  \frac{\sum_{j = n + 1}^m b_j^2}{\sum_{j = 1}^m a_j^2 b_j^2}
  \ge \sup_{m \ge n} \frac{2^m}{2^m m^\alpha} = n^{-\alpha},
\end{align*}
and so we have $\width^2(n) \lesssim \nlwidth^2(n) \le \width^2(n)$ for all
$n$.

To prove Corollary~\ref{corollary:polyhedra-stoch-opt-nonlinear},
it remains to compute $\ell_1$ diameter ratio. Applying
Lemma~\ref{lemma:l1-diameters-union-hull},
for $\alpha < 1$ we have
\begin{equation*}
  \sup_{q \in \qhull(\optdomain)}
  \<\ones_n, q\> =
  \sqrt{\sum_{j = 1}^n a_j^{-2}} = \sqrt{\sum_{j =
    1}^n j^{-\alpha}} \asymp n^\frac{1 - \alpha}{2}.
\end{equation*}
On the other hand, because $\sum_{j=1}^m b_j \asymp 2^{m/2}$, the
bounds~\eqref{eqn:growth-beta-a} and Lemma~\ref{lemma:l1-diameters-union-hull}
give
\begin{align*}
  \sup_{\optvar \in \optdomain}
  \<\ones_n, \optvar\>
  = \max_{m \le n}
  \frac{\sum_{j = 1}^m b_j}{\sqrt{\sum_{j = 1}^m
      a_j^2 b_j^2}} \lesssim
  \max_{m \le n}
  \frac{2^{m/2}}{m^{\alpha/2} 2^{m/2}}
  = 1.
\end{align*}

\ifdefined\moorsubmission
\end{APPENDICES}
\fi

\end{document}